\documentclass[11pt,twoside]{article}
\usepackage{bbm}
\usepackage{mathrsfs}
\usepackage{amssymb}
\usepackage{amsmath}
\usepackage{amsthm}
\usepackage{amsfonts}
\usepackage{color}
\usepackage{graphicx}
\usepackage[active]{srcltx}
\usepackage{CJK}
\usepackage{geometry}
\usepackage[cp1252]{inputenc}

\usepackage{mathrsfs}
\usepackage{graphicx}
\usepackage{booktabs} 
\usepackage{longtable}

\usepackage[active]{srcltx}

\allowdisplaybreaks

\usepackage[hidelinks]{hyperref}
\hypersetup{
    colorlinks=blue,
    linkcolor=blue,
    filecolor=magenta,
    urlcolor=cyan,
    pdftitle={Overleaf Example},
    pdfpagemode=FullScreen,
    }
\usepackage{titletoc}
\titlecontents{section}[0pt]{\addvspace{2pt}\filright}
              {\contentspush{\thecontentslabel\ }}
              {}{\titlerule*[8pt]{.}\contentspage}

\def\rr{{\mathbb R}}

\def\zz{{\mathbb Z}}
\def\nn{{\mathbb N}}

\def\ch{{\mathcal H}}

\def\ce{{\mathcal E}}

\def\cl{{\mathcal L}}

\def\cm{{\mathcal M}}
\def\car{{\mathcal R}}

\def\ccc{{\mathcal C}}

\def\cn{{\mathcal N}}

\def\fz{\infty}
\def\az{\alpha}

\def\std{{\mathop\mathrm{\,std\,}}}
\def\dist{{\mathop\mathrm{\,dist\,}}}

\def\lz{\lambda}
\def\dz{\delta}

\def\ez{\epsilon}

\def\kz{\kappa}
\def\bz{\beta}

\def\gz{{\gamma}}

\def\tz{\theta}
\def\sz{\sigma}
\def\pa{\partial}

\def\wz{\widetilde}
\def\Gz{\Gamma}
\def\Lz{\Lambda}

\def\bint{{\ifinner\rlap{\bf\kern.35em--}
\int\else\rlap{\bf\kern.45em--}\int\fi}\ignorespaces}

\def\bbint{{\ifinner\rlap{\bf\kern.35em--}
\hspace{0.078cm}\int\else\rlap{\bf\kern.45em--}\int\fi}\ignorespaces}

\def\minm{{\mathop\mathrm{\,minmax\,}}}

\def\Span{{\mathop\mathrm{\,span\,}}}
\def\Seq{{\mathop\mathrm{\,Seq\,}}}

\def\diam{{\mathop\mathrm{\,diam\,}}}

\def\la{\langle}
\def\ra{\rangle}

\newtheorem{thm}{Theorem}[section]
\newtheorem{lem}[thm]{Lemma}
\newtheorem{prop}[thm]{Proposition}
\newtheorem{rem}[thm]{Remark}
\newtheorem{cor}[thm]{Corollary}
\newtheorem{defn}[thm]{Definition}
\newtheorem{eg}[thm]{Example}

\numberwithin{equation}{section}

\begin{document}

\arraycolsep=1pt

\title{\Large\bf
One Dimensional Asymptotic Plateau Problem in $n$-Dimensional Asymptotically Conical  Manifolds
 \footnotetext{\hspace{-0.35cm}
\endgraf 2020 {\it Mathematics Subject Classification:} Primary 53A04 $\cdot$ 58E10; Secondary 53C21.
\endgraf Data sharing not applicable to this article as no datasets were generated or analysed during the current study.
\endgraf {\it Keywords:} Geodesic, Asymptotically Conical Manifold, Asymptotically Flat Manifold, Ideal Boundary}
}
\author{Jiayin Liu, Shijin Zhang and Yuan Zhou}
\date{}

\maketitle

\begin{center}
\begin{minipage}{13.5cm}\small
{\noindent{\bf Abstract.}\quad
Let  $(M,g)$ be an asymptotically conical Riemannian manifold having
dimension $n\ge 2$, opening angle $\az \in (0,\pi/2) \setminus \{\arcsin \frac{1}{2k+1}\}_{k \in \nn}$ and positive asymptotic rate.
Under the assumption that the exponential map is proper at each point,
we give a solution to the one dimensional asymptotic Plateau problem on  $M$.
Precisely, for any pair of antipodal points  in the ideal boundary $\pa_\fz M = \mathbb S^{n-1}$, we prove the existence of  a geodesic line   with asymptotic prescribed boundaries
and the Morse index $\le n-1$.
This extends the remarkable result of  Carlotto-De Lellis 2019 in dimension 2 and opening angle $\alpha\in (0,\pi/2)$ under the nonnegative Gaussian curvature assumption.

\quad\quad  Our proof relies on a min-max approach similar to Carlotto-De Lellis. However, to run the min-max argument, there are several new essential difficulties
caused by dimension and non-Gaussian curvature assumptions.
We are able to overcome these difficulties, but the proof is rather involved. In particular, to run the min-max argument,
we construct a nontrivial homotopy class via the geometry of geodesic  segments of standard cones serving as sweep-outs around the reference point $o$, and then bound the min-max values in a quantitative manner.
By deforming the gradient flow of certain truncation of the energy functional,
we are able to obtain a sequence of min-max geodesic segments
$\{\Gz_j\}$.
To obtain the geodesic line $\Gz$ with the morse index $\le n-1$ from
$\{\Gz_j\}$,
we establish a uniform distance estimate of $\{\Gz_j\}$ from $o$
and also a local uniform length upper bound via some blowing down arguments.   Finally, we establish  a  non-twisting property for
 $\{\Gz_j\}$ by comparing the behaviors of $\{\Gz_j\}$ with those of the standard cone and then conclude that $\Gz$ admits the prescribed antipodal ideal boundaries.

%
%
}
\end{minipage}
\end{center}

 \tableofcontents
 \contentsline {section}{\numberline { } References}{53}{section.6}%

\section{Introduction}\label{s1}
The classical Plateau problem in a Riemannian manifold $M$,
 as originated from J. Plateau \cite{p73}  when he was studying the shape of soap films,
 aims to seek for minimal surfaces $S$ with a prescribed boundary $\Sigma \subset M$.
In a non-compact
 Riemannian manifold  $M$ with ideal boundary  $\pa_\fz M$ (also known as boundary at infinity),
a counterpart problem (called  as the asymptotic Plateau problem)  is  to show the existence of unbounded minimal surfaces $S$ with a prescribed ideal boundary $\Sigma \subset \pa_\fz M$.

In   Euclidean spaces, the asymptotic Plateau problem is closely related to the  problem  posed by Bernstein \cite{b14}:  Determine all the complete minimal hypersurfaces $S \subset \rr^n$ which can be seen as a graph of a function defined in $\rr^{n-1}$.
For major progress to the Bernstein problem, we refer the readers to Bernstein \cite{b14,b27}, Moser \cite{m61}, Fleming \cite{f62}, De Giorgi \cite{d65}, Almgren \cite{a66}, Simons \cite{s68}, Bombieri-De Giorgi-Giusti \cite{bdg69} and Chern-Osserman \cite{co67}.
In  hyperbolic spaces,
the study of the asymptotic Plateau problem was initiated by Anderson \cite{a82} in 1982; see \cite{c13} for a nice survey.
We also refer to \cite{a83,bl96,l03} and references therein for  the asymptotic Plateau problem on manifolds with negative sectional/scalar curvature as generalisations of the hyperbolic space.

We are interested in noncompact manifolds possessing
certain product structure, including  the  asymptotically conical Riemannian manifolds
(see  Section \ref{s23} for the definition).
The study of asymptotic Plateau problems  in such manifolds also received  several attentions; see \cite{b81,b812,b813,bl03,r02,adr06,rss13,c16,ck18,bt20,mr22,cd19} and references therein.
To the best of the authors' knowledge,
all  these  results   focus  on finding unbounded
 minimal hypersurfaces,  that is, minimal surfaces  with codimension $1$.
 When the ambient manifold has dimension $ n\ge3$, it remains widely open to find   unbound minimal surfaces with   codimension   $\ge 2$, equivalently, with dimension $1\le k\le n-2$.

Towards this, the first step is to   find unbounded minimal  surfaces with dimension 1,
 that is,  geodesic lines.  
Throughout the paper, by a geodesic line $\gz$ (resp. ray, segment) on a Riemannian manifold $(M,g)$, we mean under its arc-length parametrisation with respect to $g$, $\gz$ is a local isometry from $(-\fz,+\fz)$ (resp. an interval $[a,+\fz)$, $[a,b] \subset \rr$) to $(M,g)$. In this sense, a geodesic line (ray, segment) is always a critical point of the length functional,  and is locally length-minimising,
but is not necessarily length-minimising on all subintervals
   $[s,t] \subset \rr$ (resp. $[s,t] \subset [a,+\fz)$, $[s,t] \subset [a,b]$).
  And in the following, we will sometimes parametrise a geodesic $\gz$ proportional to its arc-length.

 In the special case that the ambient manifold has dimension $2$,
   geodesic lines coincide with
    minimal hypersurfaces and also  minimal surface with codimension 1.  Early contributions to find geodesic lines are due to Bangert \cite{b81,b812,b813} in 1980s (see also \cite[Prop 6.1]{bl03} for a partially sharpened result   by Bonk-Lang) as mentioned above.  The
     following  remarkable contribution was made by Carlotto-De Lellis \cite{cd19} recently.

\begin{thm}
\label{cd}
  Let $(M,g)$ be an asymptotically conical surface with nonnegative Gaussian curvature. Then, for any pair of antipodal points $a,-a \in \pa_\fz M = \mathbb S^{1}$, there exists an embedded geodesic line $\Gz$ with asymptotic prescribed boundary $\pa_\fz \Gz = \{a,-a\}$ and Morse index $\le 1$.
\end{thm}

Partially motivated by this and also \cite{b81,b812,b813},
in this paper, we   establish the following existence of geodesic lines in
  higher dimensional asymptotically conical Riemannian manifolds.
See  Section \ref{s23} for  the related concepts appeared in the statement.

\begin{thm}\label{ndimanti}
  Let $(M,g)$ be an asymptotically conical Riemannian manifold with opening angle $\az \in (0,\pi/2) \setminus \{\arcsin \frac{1}{2k+1}\}_{k \in \nn}$ and positive asymptotic rate. Moreover, assume that,  at each point, the exponential map is proper. Then, for any pair of antipodal points $a,-a \in \pa_\fz M = \mathbb S^{n-1}$, there exists a geodesic line $\Gz$ with asymptotic prescribed boundary $\pa_\fz \Gz = \{a,-a\}$ and Morse index $\le n-1$.

\end{thm}

One of the motivations for establishing Theorem \ref{ndimanti}
  comes from seeking for geodesic lines in
 asymptotically  flat manifolds,  corresponding to asymptotically
 conical manifolds  with opening angle $\az=\pi/2$,
  which
  has a close relation with the study of unbounded orbits of systems of physics and celestial mechanics,
such as the Schwarzschild and Kerr metric which are the most fundamental classes of solutions to the Einstein field equation characterising black holes in general relativity.
Meanwhile, in celestial mechanics, the Jacobi-Maupertuis metric of a positive energy level from the Newtonian $N$-body system is asymptotically Euclidean away from the collisions. It is well-known that particles moving in these spaces are precisely along the geodesic lines of the corresponding asymptotically Euclidean/flat metrics (see for example \cite{o75,c03}). Thus characterising geodesic lines in these manifolds will provide understanding to the behavior of the orbit of particle as time goes to infinite past and infinite future. In particular, showing the existence of a geodesic line with prescribed ideal boundary on these manifolds will imply that there exists orbits of particles in these physical/mechanical models which escape to infinity as time goes to both infinite past and infinite future.
We emphasis that the existence of such escaping geodesic lines are rarely investigated in the literature. Especially, this is a well-known open problem for the Jacobi-Maupertuis metric (\cite[Section 5]{mv20}) while, on the contrary, escaping geodesic rays (with the prescribed ideal boundary as a point) of the Jacobi-Maupertuis metric are recently shown to be widely existing in \cite{mv20,lyz23,pt24} using completely different methods.
We hope that the approach  to showing Theorem \ref{ndimanti}    will open a way
for these studies.

\medskip

To prove Theorem \ref{ndimanti}, we will borrow some ideas from the proof of Theorem \ref{cd} in \cite{cd19}. In particular,
min-max theory was employed by Carlotto-De Lellis  \cite{cd19} to  establish  a one-dimensional mountain-pass theorem,  and is also used here to construct  a
similar but complicated $(n-1)$-dimensional  mountain-pass theorem.   As already noticed by Carlotto-De Lellis in \cite{cd19},  it   may be effortless to use the minimisation approach to find  geodesic lines.
For example, if an $n$-dimension manifold with nonnegative Ricci curvature does not split as a Riemannian product, then, thanks to the splitting theorem by Cheeger-Gromoll \cite{cg71} it  can not contain a length-minimising geodesic line.
It should be remarked that the idea of using min-max theory to study the existence of closed geodesics   has a long history. Indeed,  it is dated back to the work Birkhoff  \cite{b17}  who showed that a two-sphere with any Riemannian metric admits a closed geodesic which
answers a question posed by Poincar\'{e} \cite{p05}.
Later on, a series of works are devoted to showing the existence of closed geodesics; see \cite{l47,ls47,g89,j89,t92,k92}. For the existence result of geodesics rays, we recommend to consult \cite{s94,sst03} for a comprehensive study.
 The min-max theory also showed its powerful strength for finding  closed minimal hypersurfaces in a closed manifold under various settings. Remarkably and far from exhausted, by employing and developing Almgren-Pitts min-max method \cite{p81}, Marques-Neves \cite{mn14} resolved the Willmore conjecture, Agol-Marques-Neves \cite{amn16} proved the Freedman conjecture on the energy of links, Zhou \cite{z20} solved the multiplicity one conjecture (see also \cite{mn21}) and Song \cite{s23} confirmed the Yau conjecture on minimal hypersurfaces. Besides, the min-max theory, in recent years, also turned out to be successful to seeking for closed minimal hypersurfaces in non-compact manifolds. We refer to the works \cite{m16,chmr17,cl20,c21,s24} and references therein.

However, 
  to run the min-max argument and then to show the existence of the resulting geodesic segments, 
there are  several  new and essential difficulties  compared with the
 proof of  Theorem \ref{cd} by  Carlotto-De Lellis  \cite{cd19}. In particular, the proof of Theorem \ref{cd}   heavily relies on the
 nonnegative Gaussian curvature and dimension 2 assumptions for $M$, which, indeed, allow the authors to employ  Gauss-Bonnet formula and so to control the behavior of geodesics on a surface repeatedly.
 However, there is no such formula for curves in higher dimensions and we do not assume nonnegative Gaussian curvature in Theorem \ref{ndimanti}, either.
  Moreover, with the help of a topological result in \cite{cl14},  to prove Theorem \ref{cd},
   one only needs to work with embedded geodesic segments in dimension $2$ as did  in   \cite{cd19}   (see also Remark \ref{cl2014} below).
  However,  it was already pointed out in \cite{cl14} that such topological result does not hold in dimension $3$ and only hold   in dimension $\ge 4$ in a generic sense.
  Thus to prove Theorem \ref{ndimanti}, we have  to treat immersed geodesic segments, which causes extra obstacle.
  To overcome these difficulties  so to get Theorem \ref{ndimanti},  new ideas are definitely needed.

We outline the proof of     Theorem \ref{ndimanti}  in Section 1.1 and clarify our new ideas and
   novelty  briefly in Remark 1.4 and 1.5 therein.
   The details to the proof of    Theorem \ref{ndimanti} will be given in Section 4 and Section 5. Indeed,   Section 4 is devoted to   a min-max construction of the geodesic segments and  Section 5  gives the desired geodesic line via some approximating and contradicting arguments.
We refer to Section 2 for several  basic facts about the geometry of the standard circular cone $\mathcal C_\az^n$, and Section 3  for the definition and properties of asymptotically conical manifolds and also a  series of useful lemmas. All of these are needed in Section 4 and Section 5.
 For the readers' convenience, we collect  in Section 1.2 several notions and notation appear  repeatedly.

Finally, we list several remarks about   Theorem \ref{ndimanti}.
 \begin{rem} \label{rmk}\rm
(i) The range  $\{\arcsin \frac{1}{2k+1}\}_{k \in \nn}$ is missing from  the assumption
of    Theorem \ref{ndimanti}
for the opening angle. This is caused by our approach, and will be   explained  in Remark \ref{az} in Section 1.1.
We also explain why this range is allowed in dimension 2 as in Theorem \ref{cd}. See Remark \ref{az}.

 (ii)    There is a wide class of manifolds enjoying
 the  properness assumption of   exponential maps  made in Theorem \ref{ndimanti}.
By \cite[Theorem 4]{gm69},   each non-compact manifold of positive sectional curvature fulfils such   properness assumption.  Moreover, this class contains certain manifolds which can have negative or non-positive curvatures.
In Example \ref{egnon}, we construct a manifold with certain vanishing sectional curvatures.

The properness assumption of exponential maps is required by our proof.
It is not clear whether Theorem \ref{ndimanti} holds for asymptotically conical manifolds which has nonnegative sectional curvature but
exponential maps are not necessarily proper.
 In the proof of Theorem \ref{cd} for the $2$-dimensional case, the nonnegative Gaussian curvature assumption can be applied in accordance with the Gauss-Bonnet formula. See also Remark \ref{cl2014} (iii) below.

 (iii) The Morse index bound $\le n-1$ in Theorem \ref{ndimanti} comes directly from the fact that our homotopy class in our mountain-pass theorem (see Proposition \ref{p14} below) consists of $(n-1)$-dimensional families of curves. Recall that the Morse index bound $\le 1$ in Theorem \ref{cd} is directly from the fact that the homotopy class in \cite[Proposition 14]{cd19} consists of one-dimensional families of curves.

   (iv)  We remark that with the aid of a delicate topological result in \cite{cl14} by Chambers-Liokumovich, Carlotto-De Lellis were able to substitute the homotopies by isotopies in \cite[Proposition 14]{cd19}, which results in that the final geodesic line obtained in Theorem \ref{cd} is embedded. There, they use the fact that if a sequence of curves in a surface with no transverse self-intersections has a limit in $C^1$-norm, then the limit curve also has no transverse self-intersections.
However, this fact can be easily seen to be false for curves in dimension $\ge 3$. Thus we are not able to guarantee the immersed geodesic line in Theorem \ref{ndimanti} is embedded.

  (v)  The asymptotically conical manifolds considered in Theorem \ref{ndimanti} are always assumed to be diffeomorphic to $\rr^n$; see Definition 3.1. In particular, they are contractible and have no boundary.
  These two properties cause essential differences, compared with  the case that $M$ has nontrivial topology or $M$ has boundary.

 For instance, if $(M,g)$ is homeomorphic to $\mathbb S^{n-1}\times \rr$ and $g$ is asymptotically flat to the standard metric on $\mathbb S^{n-1}\times \rr$, then we know $\pa_\fz M$ consists of a disjoint union of two $(n-1)$-spheres $\mathbb S^{n-1}$ and $\bar {\mathbb S}^{n-1}$. For any $a \in \mathbb S^{n-1}$ and $b \in \bar { \mathbb S}^{n-1}$, it is not hard to show the existence of a geodesic line $\Gz_{a,b}$ with prescribed ideal boundary $\pa_\fz \Gz_{a,b} =\{a,b\}$, which can be realized as a limit of the sequence of length-minimising geodesic segments $\Gz_{a,b}^i$ joining $a_i$ and $b_i$ with $\lim_{i \to \fz} a_i = a$ and $\lim_{i \to \fz} b_i = b$. See the left of Figure \ref{inpic0} for an illustration.

Furthermore, when $M$ has boundary, for instance, if $M$ is the shaded region in $\rr^2$ in the right of Figure \ref{inpic0} with the length metric induce by Euclidean metric, then $\pa_\fz M$ is $\mathbb S^{1}\setminus \{a,-a\}$ where $a,-a$ are a pair of antipodal points of $\mathbb S^1$. Thus $\pa_\fz M$ is a disjoint union of two open intervals $\mathbb S^1_- \cup \mathbb S^1_+$. In this case, it is also straightforward to show the existence of a geodesic line $\Gz_{b,c}$ with prescribed ideal boundary $b \in \mathbb S^1_-$ and $c \in \mathbb S^1_+$, again by considering an exhaustion sequence of length-minimising geodesic segments.

\vspace*{1pt}
\begin{figure}[h]
\centering
\includegraphics[width=13cm]{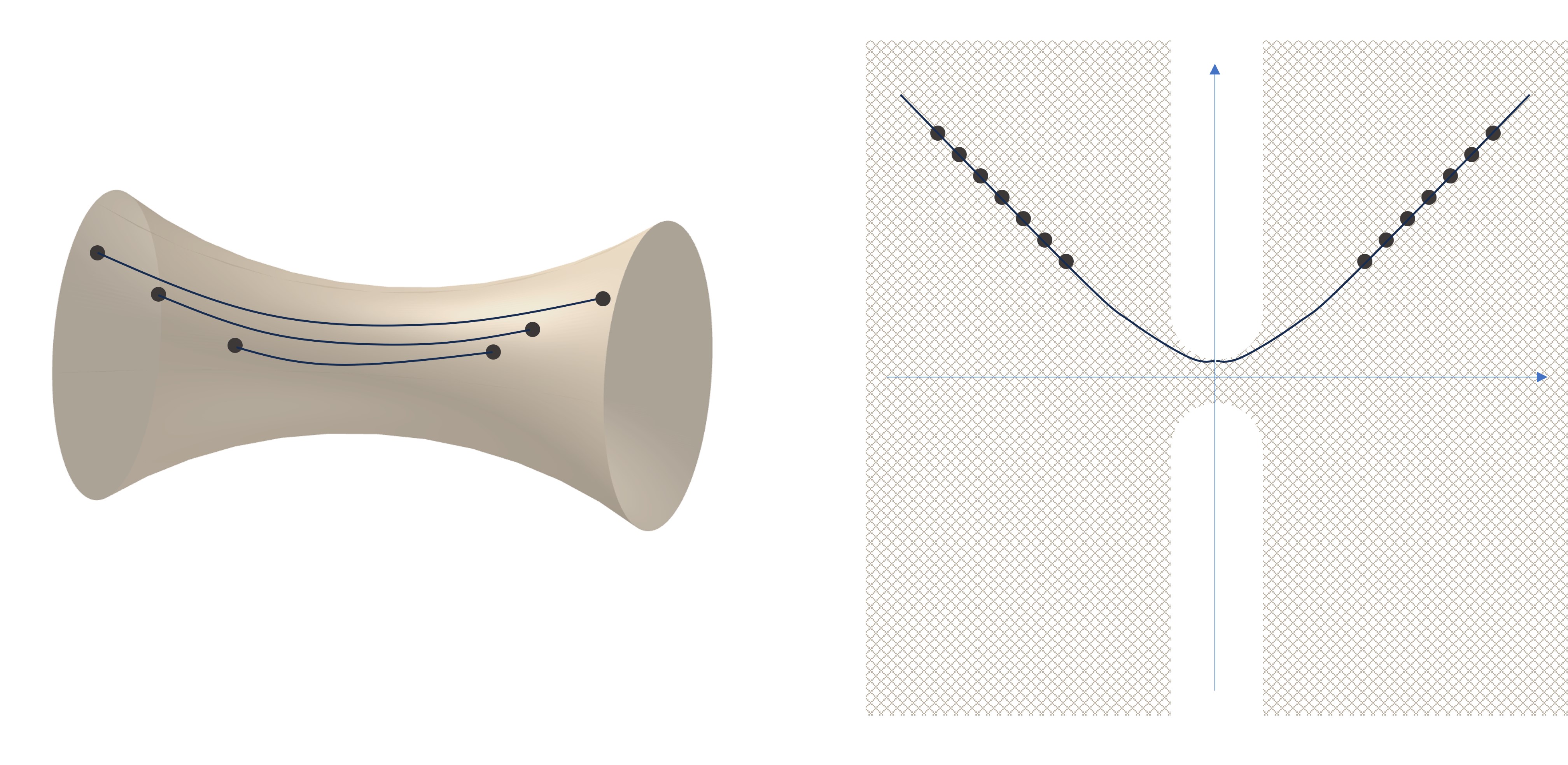}
\caption  {$M$ has nontrivial topology or $M$ has boundary.}
\label{inpic0}
\end{figure}


 \end{rem}

\subsection{Ideas of  the proof of of Theorem \ref{ndimanti}} \label{s11}

 Let $M=(\mathbb R^n,g)$ 
be an asymptotically conical manifold as in Theorem \ref{ndimanti} with reference point $ o $.
Given a pair of ideal antipodal boundaries $\{\tz,\tz^\star\} \subset \pa_\fz M =\mathbb S^{n-1}$,
we aim to construct a geodesic  line $\Gz$ with $\pa_\fz \Gz = \{\tz,\tz^\star\}$ so to prove Theorem \ref{ndimanti}. 
As it is always the case, $\Gz$ will be constructed as the limit
 in certain sense of  a suitably chosen sequence of geodesic segments $\Gz_j$ joining the
 pair $p_j=(\rho_j,\tz)$ and $q_j=(\rho_j,\tz^\star)$ of  antipodal points in $\mathbb R^n$, where $\lim_{j \to \fz}\rho_j = +\fz$.  Here,  $(\rho,\tz)$ is  the standard polar coordinate of $\rr^n$.
To guarantee the desired convergence, 
 it is necessary and also sufficient to require   that
 $\{ \Gz_j\}$ (up to some subsequence)   are close to $o$ uniformly, that is,
\begin{equation}\label{unifbdd}
  \Gz_j \cap B(o,C) \ne \emptyset \mbox{ for all   $  {j\in\mathbb N} $ 
  }
\end{equation}
for some constant $C>0$, and also 
have  bounded local length uniformly, that is,
  \begin{equation}\label{Ascolli2}
    \sup_{j \in \nn}\{\ell_g(\Gz_j\cap B(o,r))\}< \fz\ \mbox{ for any $r>C$}.
    \end{equation}

To choose a stuaible sequence $\{\Gamma_j\}$ satisfying  \eqref{unifbdd} and \eqref {Ascolli2},
 we must turn to  geodesic segments which are not length-minimising for the reason that
  any length-minimising geodesic segment joining $p_j$ and $q_j$    will escape every geodesic ball as $j\to\infty$
  that is,
  that is, for any $C>0$,   \eqref{unifbdd} fails for all sufficiently large $j$.

To choose  non-length-minimising geodesics correctly,   we are motivated by a geometric observation in  the standard (circular) cone $\mathcal C_\az^n=(\mathbb R^n,g_{\mathcal C_\az^n})$ with the apex $o$.  Indeed,   denote by $G_{p_j}$ (resp. $G_{q_j}$)  the generatrix segment joining $p_j$ (resp. $ q_j$) and   $o$ in $\mathcal C_\az^n$, which is a geodesic segment therein.
The concatenation $G_j$ of $G_{p_j}$ and $G_{q_j}$ has length $2\rho_j$ and contains $o$.
%
 Although  $G_j$ is not a geodesic segment in  $\mathcal C_\az^n$ (not length-minimising around apex $o$),
 we still wish to find some geodesic segment $ \Gamma_j$
in $M$ joining $p_j, q_j$  and playing the role of  $ G_j$ in  $\mathcal C_\az^n$, in particular, satisfying \eqref{unifbdd} for all $j$,
 and having length around  $2\rho_j$.
See Figure \ref{inpic1} for an illustration.

\vspace*{1pt}
\begin{figure}[h]
\centering
\includegraphics[width=16cm]{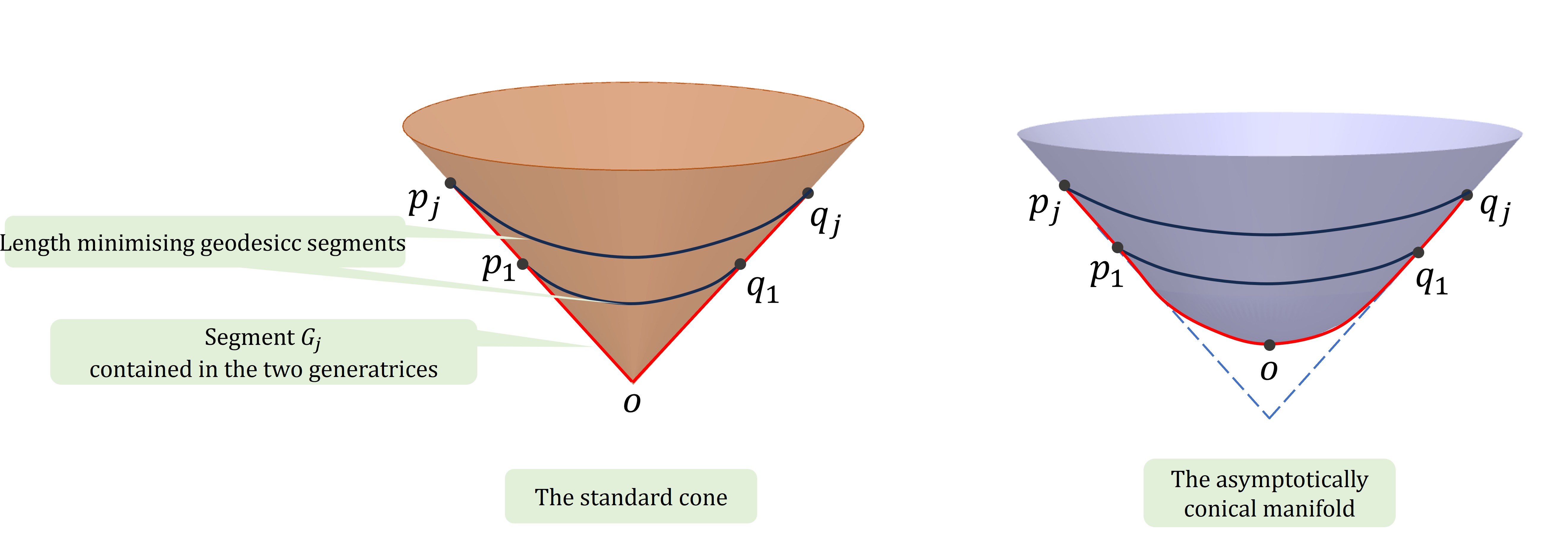}
\caption  {Geodesic segments joining $p_j$ and $q_j$.}
\label{inpic1}
\end{figure}


Surprisingly,  we do find  the desired geodesic line-segments $\{ \Gamma_j\}$
   with the help of the min-max theory built up  in the following Step 1 and Step 2,
   and then, from them  we conclude  a
geodesic line $\Gz$ with morse index $\le n-1$ in Step 3 and show
 the ideal bounardies
 $\partial \Gamma =\{\theta,\tz^\star\}$ in Step 4.
Here we only consider the case $n\ge3$ since the case $n=2$ is easier (See Remark \ref{cl2014} for the reason why we need to treat the case $n=2$ and cases $n\ge 3$ separately).
Without loss of generality, we may assume that $\theta$ and $\tz^\star$ are contained in the first coordinate $\Span\{e_1\}$ and $\tz =e_1$.
Write the class
 $$  \ccc_j:=\{\gz \in W^{1,2}([0,1],M) : \gz(0)=p_j, \ \gz(1) =q_j\},  $$
and  the energy
  $$\ce(\sz):= \int_{0}^1 |\dot \sz(t)|_g
 ^2 \, dt, \quad \forall \sz \in W^{1,2}([0,1],M).$$

{\it Step 1.    Construct a homotopy class $\{\mathcal M_j\}$ and  bound the corresponding min-max value $ \{\Lambda_j\}$. }

To  construct a suitable homotopy class, firstly we  give  a full characterisation for length-minimising geodesic segments joining
two antipodal points   of the standard (circular) cone $\mathcal C_\az^n:=(\mathbb R^n,g_{\mathcal C_\az^n})$ with opening angle $\az$;
see Lemma \ref{nmingeo} in Section \ref{s22}.
Indeed,   there exists an $(n-2)$-dimensional family of length-minimising geodesic segments:
$$\cl_{p_j,q_j}:=\{L \in C^\fz([0,1], \mathcal C_\az^n): L(0)=p_j, \ L(1)=q_j, \mbox{ $L$ is a length-minimising geodesic segment} \}.$$
Note that the length 
of each $L$ under $g_{\mathcal C_\az^n}$ is $ 2\rho \sin(\frac{\pi}2 \sin\az)$ and the metric space $(\cl_{p_j,q_j},\|\cdot\|_{C^0([0,1])})$ is homeomorphic to $\mathbb S^{n-2}$.
We parametrise $\cl_{p_j,q_j}$ as
$$\cl_{p_j,q_j}=\{L_\bz \in  C^\fz([0,1], \mathcal C_\az^n) : \bz \in \mathbb S^{n-2}\}.$$

 Next, viewing $\mathbb S^{n-2}$ as the unit sphere of $(\Span\{e_1\})^\perp=\Span \{e_2, \cdots, e_{n}\}$,
 we   identify each $L_\bz \in \cl_{p_j,q_j}$ as a point therein, 
in particular, a point in $\rr^n$.
  Consider the reflection $I$ of $\rr^{n}$ with respect to $(\Span\{e_{n}\})^\perp$, i.e.
  \begin{equation}\label{refi}
     I(x)=I((x_1,x_2,\cdots,x_{n-1},x_n)):= (x_1,x_2,\cdots,x_{n-1},-x_{n})
  \end{equation}
 where $I(x)=x$ for all $x \in \Span\{e_1,\cdots,e_{n-1}\} = \{x \in \rr^n: x_n=0\}$.  In particular, $I(p_j)=p_j$ and $I(q_j)=q_j$.
 We have the important correspondence,
 $$  I(L_\bz)= L_{I(\bz)}, \quad \forall \bz \in \mathbb S^{n-2}.  $$
 See Section 4 for this fact.

Using the above correspondence $ I$, we   define the desired class of homotopies as
 \begin{equation}\label{homotopy0}
   \cm_j:=\{H \in C^0(\mathbb S^{n-2} \times [0,1],\ccc_j) :  H(\bz,0)=  L_\bz, \ H(\bz,1) =(L_{I(\bz)}), \ \forall \bz \in \mathbb S^{n-2}\}
 \end{equation}
and the corresponding min-max value
$$  \Lz_{j}:= \inf_{H \in \cm} \max_{(\bz,s) \in \mathbb S^{n-2} \times [0,1]} \ce(H(\bz,s)).$$

Recalling that the expected length for the desired geodesic segment $\Gz_j$ is approximately $2\rho_j$ and noting that $\ce(\gz)= \ell_g(\gz)^2$ for any $\gz\in \ccc$, we know $\Lz_j$ should satisfy
\begin{equation}\label{sweep}
  \Lz _{j} \approx 4\rho_j^2.
\end{equation}
To obtain the estimate \eqref{sweep}, each homotopy $H \in \cm_j$ should be a ``sweep-out'' around $o$, which, in particular, implies that $o \in H(\mathbb S^{n-2}\times [0,1])$. This is shown in Section 4.1 and Section 4.2.

\medskip
{\it Step 2.   Find a sequence $\{\Gamma_j\}$ of min-max geodesic segments  with respect to $g$ 
 by constructing a deformation map. }

 For each $j\in \nn$,  from the definition of $\Lambda_j$,  clearly we could find a min-max sequence in $\Seq(\ccc)$ with limit $\Gamma_j$ joining $p_j,q_j$ and the energy $\ce(\Gamma_j)=\Lambda_j$.
To  show that $\{\Gamma_j\}$ are  geodesic  segments with respect to $g$,
that is, critical point of $\ce$,
firstly we  construct a suitable  deformation map in Section 4.3 (the idea is sketched as below),
and then   exploit an argument by contradiction in Proposition \ref{p14} in Section 4.4.

Generally, a deformation map $\Phi: \ccc_j \times [0,\dz] \to [0,\fz)$ is a continuous map for some small $\dz>0$ which satisfies
\begin{enumerate}
  \item[(i)]  $\Phi(\gz,\cdot)$ is strictly decreasing in $[0,\dz]$ if $\gz \in \ccc_j$ is not a critical point of $\ce$;
  \item[(ii)] $\Phi(\gz,\cdot)$ is constant if $\gz \in \ccc_j$ is a critical point of $\ce$.
\end{enumerate}

We choose the deformation map  as the gradient flow of a carefully chosen truncated energy functional $\ce_\ast: = \eta \circ \ce$,
  where $\eta \in C^\fz([0,\fz))$ is some cut-off function
(See Section 4.3).
Roughly speaking, $\ce_\ast=\ce$ for curves in $\ccc_j$ of length $\ell_g(\gz)$ around $2\rho_j$ and $\ce_\ast=0$ for curves in $\ccc_j$
of length $\ell_g(\gz)$ around $2\rho_j \sin(\frac{\pi}2 \sin\az)$.
To see the reason why we can not use the energy functional $ \mathcal E$ directly, see Remark \ref{cl2014} (ii) below.

The corresponding gradient flow $\Phi(\gz,\tau)$ with respect to $\ce_\ast$ is the solution of the ODE
\begin{equation}\label{gflow0}
 \left \{ \begin{array}{l} \frac{d}{d\tau}\Phi(\gz,\tau) = -\nabla \ce_\ast(\Phi(\gz,\tau)),
  \\
  \Phi(\gz,0)=\gz
\end{array}
\right.
\end{equation}
which is well-defined by the $C^2$-Hilbert structure on $\ccc_j$; see Section 4.3.
An important property is that for $\tau>0$, if $H \in \cm_j$, then
 $  \Phi(H,\tau) : \mathbb S^{n-2} \times [0,1] \to \ccc_j $
is also a homotopy in the class $\cm_j$.

\medskip
\emph{Step 3. Show that    $\{\Gamma_j\}$   satisfying  \eqref{unifbdd} and \eqref{Ascolli2}.}



 To this end,  in the standard cone $\mathcal C_\az^n$,  we build   a full characterisation  for geodesic segments joining two antipodal points via winding number with respect to $o$.
Moreover, we clarify the limit behaviour of geodesic lines in $\mathcal C_\az^n$, and also the
 limit behaviour of some family of geodesic lines and segments.  For details see Section \ref{s21} and Section \ref{s22}.

The verification of \eqref{Ascolli2} is then achieved by utilising the property of the properness of the exponential map and
 some blow-down arguments, and the limit behaviour of some family of geodesic segments in $\mathcal C_\az^n$.
See Lemma \ref{geoflow}, where the properness assumption of exponential maps is used.

We  prove
\eqref{unifbdd} by contradiction and by using some blow-down arguments repeatedly;  see Lemma \ref{unif} in Section 5. 
The point is that, assuming the failure of \eqref{unifbdd},  we will show that the family
  $\{\Gz_j\}_{j \in \nn}$ after blowing down will converge to a geodesic segments in $\mathcal C_\az^n$ and the behaviors of the blowing down sequence
  would violate the characterisation of geodesic segment in $\mathcal C_\az^n$ above.  The range $\alpha\in(0,\pi)\setminus
   \{\arcsin\frac{1}{2k+1}\}_{k \in \nn}$ is required to get the desired contradiction.

\medskip
\emph{Step 4. Construct a geodesic line $\Gz$ from  $\{\Gamma_j\}$
and show $\partial\Gz=\{\tz,\tz^\star\}$.}

Given a sequence  $\{  r_i\}$ with $r_i\to\infty$,
for all large $j$ denote by $t^-_{j,r_i}$ (resp. $t^+_{i,r_i}$) the first (resp. last) time
that $\Gz_j$ enters into (resp. exits from) $\overline{B(o,r_i)}$.
We show that,
 as $j\to\infty$ (up to some subsequence), the restriction
$\Gz_{j}|_{[t^-_{j,r_i},t^+_{j,r_i}]}$ will converge to some geodesic segment
$\Gamma_i^\ast$ with the property that $\Gamma_i^\ast \subset \Gamma_{i+1}^\ast$ for $i \in \nn$.
Then
the limit of certain reparametrization of $\Gz^\ast_i$ gives a geodesic line $\Gz$.

 It remains to show $\Gz$ having antipodal ideal boundaries $\{\tz,\tz^\star\}$.
 The  core is to establish the non-twisting property for the min-max geodesic segments
 $\{\Gz_j\}_j$, in other words,
  outside any given ball, $\{\Gz_j\}_j$  asymptotes to antipodal  directions $\{\tz,\tz^\star\}$ in a uniform sense.
  The proof is highly nontrivial.
 To this end, we  compare the behaviors of geodesic segments of the standard cone with $\Gz_j$ based on several blow down procedures; see  Section 5.2 for details.

This finishes the outline of the proof of Theorem \ref{ndimanti}.

\medskip
 Below  we clarify  our new ideas and novelty in the proof of Theorem \ref{ndimanti}, and also list several key differences against  the proof of Theorem \ref{cd} by Carlotto-De Lellis  \cite{cd19} in the $2$-dimensional case.

\begin{rem} \label{cl2014} \rm

(i)  We emphasis that the homotopy class $\cm_j$ in \eqref{homotopy0} is a non-trivial higher dimensional analogue of the one in \cite{cd19}. Indeed, for the above two antipodal points $p_j$ and $q_j$, Carlotto-De Lellis constructed the homotopy $\cn_j$ as
\begin{equation}\label{nj}
  \cn_j :=\{H \in C^0([0,1],\ccc_j) :  H(0)= \Gz_{j,1}^{g}, \ H(1) =\Gz_{j,2}^{g}\}
\end{equation}
where $\Gz_{j,i}^{g}$ is a geodesic segments joining $p_j$ and $q_j$, $i=1,2$ with respect to the asymptotically conical metric $g$, and $\Gz_{j,1}^{g}\cup\Gz_{j,2}^{g}$ is topologically an $\mathbb S^1$ for each $j \in \nn$. The existence of two geodesic segments joining $p_j$ and $q_j$ with respect to $g$ is shown in \cite[Lemma 18]{cd19}. Furthermore, let $D_j$ be the bounded component of $M(\approx \rr^2) \setminus [\Gz_{j,1}^{g}\cup\Gz_{j,2}^{g}]$, then $o \in D_j$ for all $j \in \nn$. This property guarantees that for any homotopy $H$ in $\cn_j$, $o \in H(s)$ for some $s \in [0,1]$, since in dimension $2$, $D_j \setminus \{o\}$ is not simply-connected.

However, in an asymptotically conical manifold $(M,g)$ with dimension $n\ge 3$, it may not be possible to always find an $(n-2)$-dimensional family of geodesic segments $\{\Gz_{j,\bz}^g\}_{\bz \in \mathbb S^{n-2}}$ joining $p_j$ and $q_j$ with respect to $g$ to construct the desired homotopy.
 Our new idea here is  to use the length-minimising geodesic segments $L_\bz$ of the standard cone $\mathcal C_\az^n$ in the construction of our homotopy class $\cm_j$;
 as explained in Step 1, the family $\cl_{p_j,q_j}$ of length-minimising geodesic segments joining $p_j$ and $q_j$ forms an $(n-2)$-dimensional family of curves. Moreover,
 \begin{equation}\label{bv1}
   \bigcup_{L_\bz \in \cl_{p_j,q_j} }L_\bz \subset M \approx \rr^{n}  \mbox{ is topologically an $\mathbb S^{n-1}$},
 \end{equation}
and letting $D_j$ be the bounded component of $M \setminus [\cup_{L_\bz \in \cl_{p_j,q_j} }L_\bz]$, then
\begin{equation}\label{bv2}
  o \in D_j, \quad \forall j \in \nn.
\end{equation}
Thus, the $(n-2)$-dimensional family $\cl_{p_j,q_j}$ of length-minimising geodesic segments of the standard cone $\mathcal C_\az^n$ is a nice substitute for $\Gz_{j,i}^{g}$ in dimension $2$.

Nonetheless, unlike the case in dimension $2$, in dimensions $n \ge 3$, the above two properties \eqref{bv1} and \eqref{bv2} are not enough to guarantee the following crucial property:  for a general homotopy $H: \mathbb S^{n-2}\times [0,1] \to \ccc_j$
with boundary value
$$ H(\mathbb S^{n-2},0)\cup H(\mathbb S^{n-2},1) \subset \cl_{p_j,q_j}, $$
there always exists $(\bz,s) \in \mathbb S^{n-2}\times [0,1]$ such that
$$o \in H(\bz,s).$$
This is because in dimensions $n \ge 3$, $M \setminus \{o\}$ is still simply-connected. Indeed, in Section 4, we provide an example of such homotopy $H$ without passing though $o$ (See Example \ref{o} in Section 4.1).

By the above discussion, the reflection map $I$ is indispensable in the definition of the homotopy class $\cm_j$, which guarantees that $o \in H(\mathbb S^{n-2}\times [0,1])$ for all $H \in \cm_j$ and thus enables us to conclude the desired estimate \eqref{sweep}.



(ii)   The above different choices of homotopy class also lead to different choices of deformation maps. To run the min-max argument successfully, an important property for deformation map $\Phi: \ccc \times [0,\dz] \to [0,\fz)$ is that $\Phi(\gz,\cdot)$ should be constant when $\gz$ is the boundary value of a homotopy $H$ in the homotopy class $\cm_j$ (resp. $\cn_j$), that is,
\begin{equation}\label{value}
  \Phi(\gz,\tau) \equiv \Phi(\gz,0), \ \tau \in [0,\dz] \mbox{ if $\gz \in \{H(\bz,0), H(\bz,1)\}_{\bz\in \mathbb S^{n-2}}$ (resp. $\gz \in \{H(0), H(1)\}$) .}
\end{equation}
In \cite{cd19}, Carlotto-De Lellis could directly use the gradient flow $\Phi$ of the energy functional $\ce$ with respect to $g$, since by their choice of homotopy in \eqref{nj}, the boundary values $H(0)$ and $H(1)$ are geodesic segments with respect to $g$, thus $H(0)$ and $H(1)$ are critical points of $\ce$, which are equilibriums of the flow $\Phi$ and \eqref{value} holds under this choice.

In our case, the boundary values of $H$ are elements in $\cl_{p_j,q_j}$ consisting of geodesic segments respect to the standard cone $\mathcal C_\az^n$, thus they may not be equilibriums of the gradient flow of the energy functional $\ce$ with respect to $g$ and consequently \eqref{value} may not hold for $\gz \in \cl_{p_j,q_j}$. However, noticing that $\gz \in \cl_{p_j,q_j}$ always has length $\ell_g(\gz)$ approximately $2\rho_j \sin(\frac{\pi}2 \sin\az)$, which is strictly smaller than the min-max value $\Lz_j \approx 2\rho_j$, our new idea is to find a cut-off function $\eta$ to truncate the energy functional $\ce$ between these two values and so to consider the functional $$
\ce_\ast=\eta\circ\mathcal E.$$
The point is that  $\gz \in \cl_{p_j,q_j}$ will be  equilibriums of the truncated gradient flow $\Phi$ of $\ce_\ast$
and \eqref{value} holds under this choice.
See also Remark \ref{flowast} for the role of property \eqref{value} in the min-max argument Proposition \ref{p14}.

(iii)  It is  also our novelty  to prove  \eqref{unifbdd} and \eqref{Ascolli2} in Step 3 under  the  properness assumption of exponential maps.
Our  verifications of \eqref{unifbdd} and \eqref{Ascolli2}  are totally different from the proof of Theorem \ref{cd} by Carlotto-De Lellis, where they applied Gauss-Bonnet theorem multiple times for showing both \eqref{unifbdd} and \eqref{Ascolli2} combined with the information of embeddedness of $\Gz_j$ and the nonnegative curvature condition.
Since there is no Gauss-Bonnet formula for geodesics in higher dimensional manifolds $M$ and we do not have the embeddedness of $\Gz_j$, nor assume $M$ has nonnegative curvature,
 the chance to adapt their approach into higher dimensions is slim.

 (iv)   In Step 4, it is also our novelty  to  figure out a new argument to build up the desired uniform non-twisting phenomenon.  Recall
  that in dimension 2,
Carlotto-De Lellis \cite{cd19} first showed the uniform non-twisting phenomenon
 so to conclude  $\partial \Gz=\{\theta,\theta^\star\}.
 $  To this end, they compared the behaviors of length-minimising geodesic rays with $\Gz_j$ and, again, employed Gauss-Bonnet formula.  However, in dimensions $n \ge 3$, an analogous argument is unavailable by the same reasons above.
\end{rem}

The following remark  explains the reason why the values $\{ \arcsin\frac{1}{2k+1}\}_{k \in \nn}$ are missing from
the assumption for the opening angles in Theorem \ref{ndimanti}.

\begin{rem} \label{az} \rm
In Step 3, we mention that our approach could not deal with the case that the opening angle $\az \in\{ \arcsin\frac{1}{2k+1}\}_{k \in \nn}$.
Here is an explanation. For each $\az \in [\arcsin \frac{1}{2k+1},\arcsin \frac{1}{2k-1})$, all the geodesic lines on the standard $2$-dimensional circular cone $ \mathcal C_\az^2$ with opening angle $\az$ has winding number $k-1$ with respect to the apex $o$. Thus the values $\{\arcsin \frac{1}{2k+1}\}_{k \in \nn}$ are the thresholds where this winding number decreases as $\az$ increases (see Lemma \ref{antigeo}).
   Moreover, geodesic segments $\Gz_\std$ of the standard circular cone with winding number $>0$ have self-intersections which means that they are immersed, not embedded curves. Since we can not restricted ourselves to embedded geodesic segments $\Gz$ on the asymptotically conical manifold $M$ as mentioned in Remark \ref{rmk} (iii) and in the third paragraph below Theorem \ref{ndimanti}, in the proof of Lemma \ref{unif} in Section 5, we need compare the behaviour of  $\Gz$ with an immersed $\Gz_\std$. The presence of self-intersections will lead to an inefficiency in
   the proof of Lemma \ref{unif} for the values $\{\arcsin \frac{1}{2k+1}\}_{k \in \nn}$.

  On the contrary, in dimension $2$, Carlotto-De Lellis could  focus on embedded geodesic segments $\Gz$ on $M$ where they only need to
compare the behaviour of $\Gz$ with the embedded geodesic segments $\Gz_\std$ of the standard circular cone. This reduction enables them to obtain Theorem \ref{cd} for all $\az \in (0,\pi/2)$.


 \end{rem}


\subsection{Notations}\label{s13}
In this paper, unless otherwise stated, we will always work in the canonical Cartesian coordinate $(x_1,\cdots,x_n)$ and the polar coordinate $(\rho,\tz)=(\rho,\tz_1,\cdots,\tz_{n-1}) \in [0,+\fz) \times [0,2\pi)\times [-\pi/2,\pi/2]^{n-2}$ in $\rr^n$.
For any $p=(p_1,\cdots,p_n) \in \rr^n \setminus \{o\}$, we define its antipodal point $q$ by $q=(-p_1,\cdots,-p_n)$. In the polar coordinate, if $p=(\rho, \tz)$, then we denote by $q:=(\rho,\tz^\star)$, i.e., we define $\tz^\star$ as the antipodal direction of $\tz$.

For any two points $p_1,p_2 \in \rr^n\setminus \{o\}$, $\angle(p_1,p_2)$ is the Euclidean angle and when $p_i \in \mathcal{T}_{x_i}\rr^n$ for some $x_i \in \rr^n$, we identify $\mathcal{T}_{x_i}\rr^n$ as $\rr^n$ in the natural way and $\angle(p_1,p_2)$ still denotes the Euclidean angle between $p_1$ and $p_2$.

For $\az \in (0,\pi/2]$, we define the wedge domain
\begin{equation*}
  W_{2\pi\sin \az}= \{(\rho,\phi): \rho\in [0,\fz), \ \phi \in [-\pi\sin \az,\pi\sin \az) (= \rr \mod 2\pi\sin \az)\} \subset \rr^2.
\end{equation*}
We use the notation $\bar p, \bar q$ to denote the points in the wedge domain.

Finally, we gather the notations that are repeatedly appeared in this paper.
\begin{longtable}{ll}
$\mathcal C_\az^n=(\rr^n,g_{\mathcal C_\az^n})$  & $n$-dimensional circular cone with Riemannian metric $g_{\mathcal C_\az^n}$ in \eqref{stdncone}   \\
$(M,g)$ & $n$-dimensional manifold with an asymptotically conical metric $g$  \\
$\mathcal{T}_pM$ (resp. $\mathcal{S}_pM$) & Tangent (resp. Unit tangent) space at $p \in M$  \\
$|\cdot|_{\mathcal C_\az^n}$ (resp. $|\cdot|_{g}$, $|\cdot|$) \qquad {} & Norm on $\mathcal{T}_pM$ induced by $g_{\mathcal C_\az^n}$ (resp. $g$, the Euclidean metric)\\
$d_{\mathcal C_\az^n}$ (resp. $d_{g}$) \qquad {} & Distance function induced by $g_{\mathcal C_\az^n}$ (resp. $g$)\\
$\dist(\cdot,\cdot)=|\cdot-\cdot|$  \qquad {} &  Euclidean distance \\
$ \ell_{\mathcal C_\az^n}$ (resp. $\ell_{g}$, $\ell$)  &  Length functional induced by $g_{\mathcal C_\az^n}$ (resp. $g$, the Euclidean metric)  \\
$B_g(x,r)$ & Open geodesic ball with center $x$ and radius $r$ induced by $g$ \\
(resp. $B(x,r)$) & (resp. the Euclidean metric) \\
$\exp^{\mathcal C_\az^n}$ (resp. $\exp^g$)  & Exponential map induced by $g_{\mathcal C_\az^n}$ (resp. $g$) \\
\end{longtable}

\section{Geometry of  standard cones}\label{s2}

We start with the definition of $n$-dimensional circular cone with $ n\ge2$.
Write $\mathcal C_\az^n$ as the (upper) standard $n$-dimensional circular cone with opening angle $\az \in (0,\pi/2)$, i.e.
\begin{equation}\label{subset}
  \mathcal C_\az^n=\left\{(x_1, \cdots ,x_{n+1})\in\rr^{n+1} \times [0,\fz): \sum_{i=1}^n x_i^2= (x_{n+1}\tan\az)^2\right\}
\end{equation}
with the metric inherited from $\rr^{n+1}$.
Notice that $\mathcal C_\az^n \setminus \{o=(0,\cdots,0)\}$ is a smooth Riemannian manifold homeomorphic to $\rr^n \setminus \{o\}$ with Riemannian metric
\begin{equation}\label{stdncone}
  g_{\mathcal C_\az^n}(\rho,\tz)= d\rho \otimes d\rho + \rho^2 (\sin^2 \az ) g_{\mathbb S^{n-1}}(\tz),
\end{equation}
 where we parameterise $\mathcal C_\az^n$ in the standard $n$-dimensional polar coordinate as
 \begin{equation}\label{pan1}
   \{(\rho,\tz):\rho\in [0,\fz), \ \tz=(\tz^1, \cdots, \tz^{n-1}) \in [0,2\pi)\times [-\pi/2,\pi/2]^{n-2}  \}
 \end{equation}
  and
  $g_{\mathbb S^{n-1}}$ is the $(n-1)$-dimensional spherical metric.
Sometimes we  also write $g_{\mathcal C_\az^n}$ in the Cartesian coordinate as follows:
\begin{equation}\label{stdcar}
  g_{\mathcal C_\az^n}(x)= \sum_{i,j=1}^{n}g_{ij}^{\mathcal C_\az^n}(x) \, dx_i \otimes dx_j,\quad   \forall x=(x_1,\cdots,x_n)\in \rr^n\setminus\{o\}.
\end{equation}
Note that there exists constant $C$ only depending on $n$ and $\az$ such that
\begin{equation}\label{compare}
  |v|_{g_{\mathcal C_\az^n}}= \sqrt{\la g_{\mathcal C_\az^n}(x)  \xi,\xi\ra} \in [\frac{1}{C}|\xi|,C|\xi|], \quad \forall \xi \in \mathcal{T}_x\mathcal C_\az^n \mbox{ and } x \in \mathcal C_\az^n\setminus\{o\}.
\end{equation}

  In the following of this paper, the notation $\mathcal C_\az^n$ always refers to
$$ \mathcal C_\az^n = (\rr^n, g_{\mathcal C_\az^n}), $$
which is understood in the sense that $\mathcal C_\az^n$ is the metric completion of the (non-complete) $n$-dimensional manifold
$(\rr^n \setminus \{o\},g_{\mathcal C_\az^n})$.
Thus $\rr^n$ naturally forms a global chart of $\mathcal C_\az^n$ with the origin $o$ the only non-smooth point. We call $o$ the apex of $\mathcal C_\az^n$.
Writing $d_{\mathcal C_\az^n}$ as the distance function induced from $g_{\mathcal C_\az^n}$, directly from \eqref{stdncone}, we have
\begin{equation}\label{do}
  d_{\mathcal C_\az^n}(x,o) = |x-o|= |x|, \quad x \in \rr^n,
\end{equation}
where $|x-o|$ is the Euclidean distance between $x$ and $o$ and $|x|$ is the Euclidean norm of $x$. Denote by $\ell_{\mathcal C^n_\alpha}(\gamma)$
the length of a   curve $\gamma:[a,b]\to\mathbb R^n$ with respect to the metric
  $g_{\mathcal C_\az^n}$.

We call $G$ a generatrix of $\mathcal C_\az^n$ if
$$ G = \{a t : t \ge 0 \} \mbox{ for some } a \in \mathbb S^{n-1}.$$
Note that $\mathcal C_\az^n$ is the union of its generatrices, i.e.,
  $$\mathcal C_\az^n= \bigcup_{a \in \mathbb S^{n-1}} \{a t : t \ge 0 \}.$$

Here is the definition for antipodal generatrices and points.
  \begin{defn}  \rm
 (i) We say that two generatrices $G_1$ and $G_2$ of $\mathcal C_\az^n$ are antipodal if $G_1 \cup G_2$ form a Euclidean line.

 (ii)   We say two points $p$ and $q$ of $\mathcal C_\az^n$ are antipodal if they are contained in a pair of antipodal generatrices respectively and have same distance to the origin. In the Cartesian chart, $p=-q$.  In the polar coordinate
   if $p=(\rho,\tz)=(\rho,\tz_1,\tz_2,\cdots,\tz_{n-1})$, then
   $$q=(\rho,\tz_1+\pi(\mod 2\pi),-\tz_2,\cdots,-\tz_{n-1})  =:(\rho,\tz^\star),
  $$
 \end{defn}


Below we give  several useful facts for the $2$-dimensional and $n$-dimensional cone in Subsection \ref{s21} and \ref{s22} respectively.
Mostly we focus on the behavior of geodesic segments, rays and lines on the standard cone. 

\subsection{Geometry of   $ \mathcal C_\az^2$}\label{s21}

 In the 2-dimensional cone $ \mathcal C_\az^2= (\rr^2, g_{ \mathcal C_\az^2})$,
  we  write \eqref{stdncone} explicitly as
 \begin{equation}\label{pa1}
g_{ \mathcal C_\az^2}(\rho,\tz)= d\rho \otimes d\rho + \rho^2 \sin^2 \az \, d\tz\otimes d\tz.
 \end{equation}

 To understand the behaviour of geodesic segment (resp. ray, line) in
 $ \mathcal C_\az^2$, we consider the     unfolding of $ \mathcal C_\az^2$ to some wedge domain. To be precise,
 by cutting $ \mathcal C_\az^2$ along a generatrix, say $\{\tz=-\pi\}$, we   unfold $ \mathcal C_\az^2$ to the wedge domain
\begin{equation}\label{wedge}
  W_{2\pi\sin\az}= \{(\rho,\phi): \rho\in [0,\fz), \ \phi \in [-\pi\sin \az,\pi\sin \az) (= \rr \mod 2\pi\sin \az)\} \subset \rr^2
\end{equation}
under the map  $  (\rho,\tz) \mapsto  (\rho,\phi(\tz,\az))$
where $(\rho,\tz) \in [0,\fz) \times [-\pi,\pi)$ is the standard polar coordinate of $\rr^2$ and
\begin{equation}\label{angleast}
  \phi(\tz,\az) = \sin (\az)\tz.
\end{equation}
Note that  $W_{2\pi\sin\az}$ is equipped with the planar Euclidean metric.
It is well-known that the unfolding map is an isometry, we know $ \mathcal C_\az^2$ is isometric to the planar wedge $W_{2\pi\sin\az}$ by identifying $\{\tz=-\pi\sin \az\}$ and $\{\tz= \pi\sin \az\}$.

 To distinguish notations, we  always use   $p=(\rho ,\tz)$ for the parametrisation \eqref{pa1} of $ \mathcal C_\az^2$ and  $\bar p =(\rho ,\phi)$
for the parametrisation \eqref{wedge} of the wedge domain $W_{2\pi\sin\az}$.
For any two points $p,q \in  \mathcal C_\az^2 = (\rr^2,g_{ \mathcal C_\az^2})$,
write $\angle(p,q)$ as the Euclidean angle between the vector $op$ and $oq$,
 and $\angle (\bar p,\bar q)$ as the corresponding
 Euclidean angle betwee the vector $\bar o\bar p$ and $\bar o\bar q$ 
  in  $W_{2\pi\sin\az}$.
One has
\begin{equation}\label{wed}
  \angle(\bar p,\bar q) = \angle(p,q) \sin\az.
\end{equation}
Since  the  unfolding map $p\mapsto\bar p$ is an isometry,
the length of any curve   $\Gamma$   with respect to $g_{\mathcal C^2_\alpha}$
coincide with Euclidean length of  the image $\bar \Gamma$ of $ \Gamma $ under the  unfolding map. Moreover, a curve
  $\Gamma$ is a geodesic segment (resp. line, ray) in $\mathcal C^2_\alpha$ if and only if $\bar \Gz$ is a  geodesic segment (resp. line, ray) in   $W_{2\pi\sin\az}$ or in Euclidean sense.
In particular,  each
  generatrix  $G$ (given by $\rr_+(1,\theta) $ for some $ \theta\in[0,2\pi)$
  of $ \mathcal C_\az^2$ is a geodesic ray   in  $\mathcal C^2_\alpha$  since
  the induced
  $\bar G$ (given by $\rr_+(1,\phi ) $  with
   $\phi= (\sin\alpha) \theta$)    is a Euclidean  ray passing through the origin $o$
   (a geodesic ray   in the wedge domain).

\begin{rem}
  \rm We remark that in the Cartesian chart $\rr^2$ and $p,q \in \rr^2$, the angle $\angle poq$ under the Euclidean metric and the angle $\angle_{g_{ \mathcal C_\az^2}} poq$ induced by the metric completion of $g_{ \mathcal C_\az^2}$  are different. The reason to use Euclidean angle throughout the paper is that in the following proofs, for curves $\gz \in C^1([0,1],\rr^2)$, we need to estimate the angles such as $\angle(\dot\gz(t_0),\dot\gz(t_1))$ for some $t_0,t_1 \in [0,1]$. Note that the notion $\angle(\dot\gz(t_0),\dot\gz(t_1))$ is naturally defined for Euclidean angles (under the standard identification $\mathcal{T}_{\gz(t_i)}\rr^2=\rr^2$, $i=0,1$). However, since $\dot\gz(t_i) \in \mathcal{T}_{\gz(t_i)}\mathcal C_\az^2$ are contained in two distinct tangents spaces, $i=0,1$ (if $\gz(t_0)\ne \gz(t_1)$), $\angle_{g_{ \mathcal C_\az^2}}(\dot\gz(t_0),\dot\gz(t_1))$ is not well-defined.
\end{rem}

\begin{lem} \label{coneabc}
  Let $
  p,q \in  \mathcal C_\az^2$.

 (i) There exists a length-minimising geodesic segment $\Gz_A$ joining $p$ and $q$ with length
  $$\ell_{\mathcal C_\az^2}(\Gz_A) = 2\rho \sin(\frac{\angle(p,q)}2 \sin\az).$$

  (ii) There exists a piecewise geodesic segment $\Gz_B$ joining $p$ and $q$ passing through $o$ with length $\ell_{\mathcal C_\az^2}(\Gz_B) = 2\rho$.

(iii)  If $\az \in (0,\pi/6)$, or  if $\az \in [\pi/6,\pi/2)$ and
     $\angle(p,q) \in (2\pi-\frac{\pi}{\sin \az},\pi]$, then
     there exists another  geodesic segment $\Gz_C$ joining $p$ and $q$ with length $$\ell_{\mathcal C_\az^2}(\Gz_C) = 2\rho \sin(\frac{2\pi-\angle(p,q)}2 \sin\az). $$
\end{lem}

\begin{proof}
By the rotational symmetry of $ \mathcal C_\az^2$, we may assume that
 $p=(\rho,0)$ and $q=(\rho,\theta)$ for some $ \tz\in(0,\pi]$. Then the angle
  $ \angle(p,q)=\theta$.
   Note that $ p,q $ correspond to
  $\bar p=(p,0)$ and $\bar q=(\rho,\phi)$  with $\phi = \tz  \sin\az$ in the wedge domain $W_{2\pi\sin\az}$.  To get the desired $  \Gamma_A$, $  \Gamma_B$ and $\Gamma_C$ in $\mathcal C^2_\alpha$,
it suffices to find
$ \bar \Gamma_A$, $\bar \Gamma_B$ and $\bar
\Gamma_C$  in the wedge domain $W_{2\pi\sin\az}$ joining $\bar p$ and $ \bar q$
and fulfilling our requirements respectively. We refer to Figure \ref{ABC} for an illustration for the following proof.

 (i) Since $W_{2\pi\sin\az} \subset \rr^2$ is in the Euclidean plane,  the line segment joining $\bar p$ and $\bar q$ gives  the desired length-minimising geodesic $\bar \Gz_A$.
 By the planar geometry and recalling \eqref{wed},  the Euclidean length
 of $\bar \Gz_A$ is   $$
 2\rho \sin(\frac{\phi }2  )
 = 2\rho \sin(\frac{\tz }2 \sin\az).
 $$

(ii)
%
 Write $\bar \Gamma_B$ as the union of
 $\{(t,0) \}_{t\in[0,\rho]}$  and $\{(t,\theta) \}_{t\in[0,\rho]}$.
Then $\bar \Gz_B$ is a piecewise geodesic segment joining $\bar p$ and $\bar q$ with  Euclidean length $ 2\rho$.

  \vspace*{1pt}
\begin{figure}[h]
\centering
\includegraphics[width=16cm]{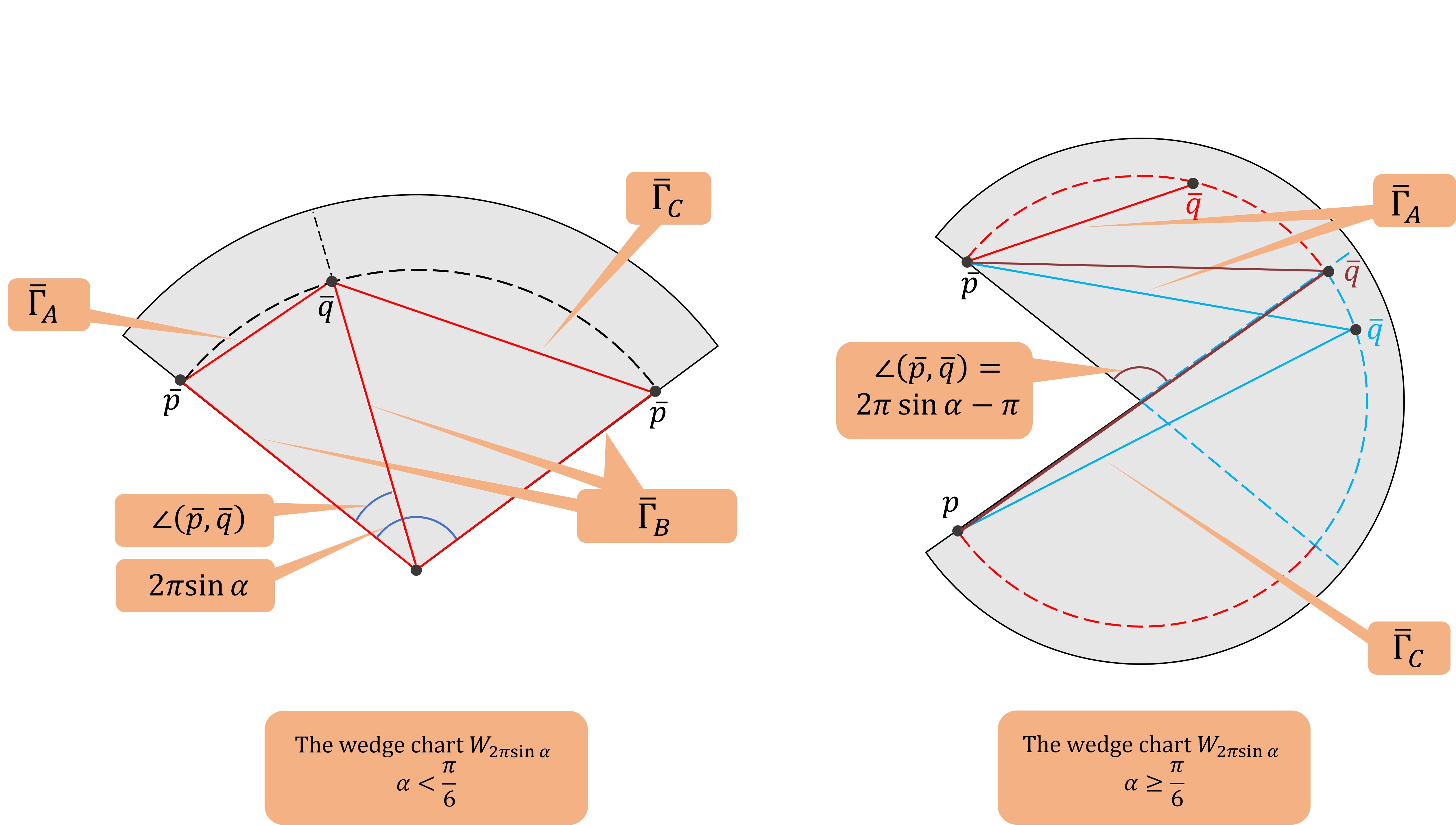}
\caption  {An illustration of $\Gz_A$, $\Gz_B$ and $\Gz_C$.}
\label{ABC}
\end{figure}

 (iii) If $\sin \az \le 1/2$, i.e. $\az \le \pi/6$, then the wedge domain $W_{2\pi\sin\az}$ is less than a half-disk.  There is a   line-segment
 $\bar \Gz_C$ joining $\bar p$ and $\bar q$ in $W_{2\pi\sin\az}$, which gives the desired
  $\Gz_C$ in   $ \mathcal C_\az^2$. See the left of Figure \ref{ABC}.

  If $\sin \az > 1/2$, i.e. $\az > \pi/6$, then the wedge domain $W_{2\pi\sin\az}$ is more than a half-disk as depicted in the right of Figure \ref{ABC}. Observe that
  there exists a line-segment   $\bar \Gz_C$ joining $\bar p$ and $\bar q$
  different with $\bar \Gamma _A$  if and only if
  $q$ lies in the blue arc,  that is
    $$ \angle(\bar p, \bar q)  \in (2\pi\sin \az-\pi,\pi],\mbox{ or equivalently},
   \angle( p,  q) \in (2\pi-\frac{\pi}{\sin \az},\pi]  .$$

   In both cases, the Euclidean length of $\bar \Gz_C$  is
    $$    2\rho \sin(\frac{2\pi\sin\alpha-\angle(\bar p, \bar q)}2)=
     2\rho \sin(\frac{2\pi-\angle(p,q)}2 \sin\az).$$
     We omit the details.
  \end{proof}

%

Recall that for any planar curve $\gz\in C^0([0,1],\rr^2)$,
 the  winding number $k \in \zz$ with respect to some point $w \in \rr^2$ can be defined as the unique integer $[\wz \gz] \in \zz=\pi_1(\rr^2 \setminus \{w\})$ of the first homotopy group of $\rr^2 \setminus \{w\}$ where $\wz \gz$ is a closed curve such that $\wz \gz$ and $\gz$ are homotopic in $\rr^2 \setminus \{w\}$ and $[\wz \gz]$ represents the equivalent class of $\wz \gz$.
Note  that $\gamma$ is homotopic to $\eta$ in $\rr^2 \setminus \{w\}$ if and only if they have
the
 same winding number with respect to $w$.

Now we focus on the geodesic segments joining two antipodal points.

\begin{lem} \label{antigeo}
  Let $p=(\rho, \tz), q = (\rho, \tz+\pi) \in  \mathcal C_\az^2$ be two antipodal points.
    For  each  $ k \in [0,\frac{1}{2\sin\az}-\frac12) \cap \nn$, there exists
a pair $\Gamma_k^\pm$ of smooth geodesic segments joining $p$ and $q$ and having winding numbers $\pm k $ with respect to the apex $o$.
 Moreover,
  \begin{equation}\label{gp1}
     \angle ( -p, \dot \Gz^\pm_k(0))= \frac{\pi-(2k-1)\pi\sin\az}{2\sin\alpha}
  \end{equation}
  and
  \begin{equation}\label{gp2}
    \ell_{\mathcal C_\az^2}(\Gz)  =
     2\rho \sin\frac{(2k-1)\pi\sin\az}{2}
  \end{equation}




%
\end{lem}

\begin{proof}
Note that the adjacent wedge domains (without the origin) in Figure \ref{gammapq} can be seen as a covering map of $ \mathcal C_\az^2\setminus \{o\}$ which is locally isometric. Thus all the smooth geodesic segments joining $p$ and $q$ are Euclidean line segments in the (union of) adjacent wedge domains. In Figure \ref{gammapq}, there are four adjacent wedge domains and then there are three pairs of geodesic segments joining $\bar p$ and $\bar q$, among which the red one is the length-minimising one.
   \vspace*{1pt}
\begin{figure}[h]
\centering
\includegraphics[width=8cm]{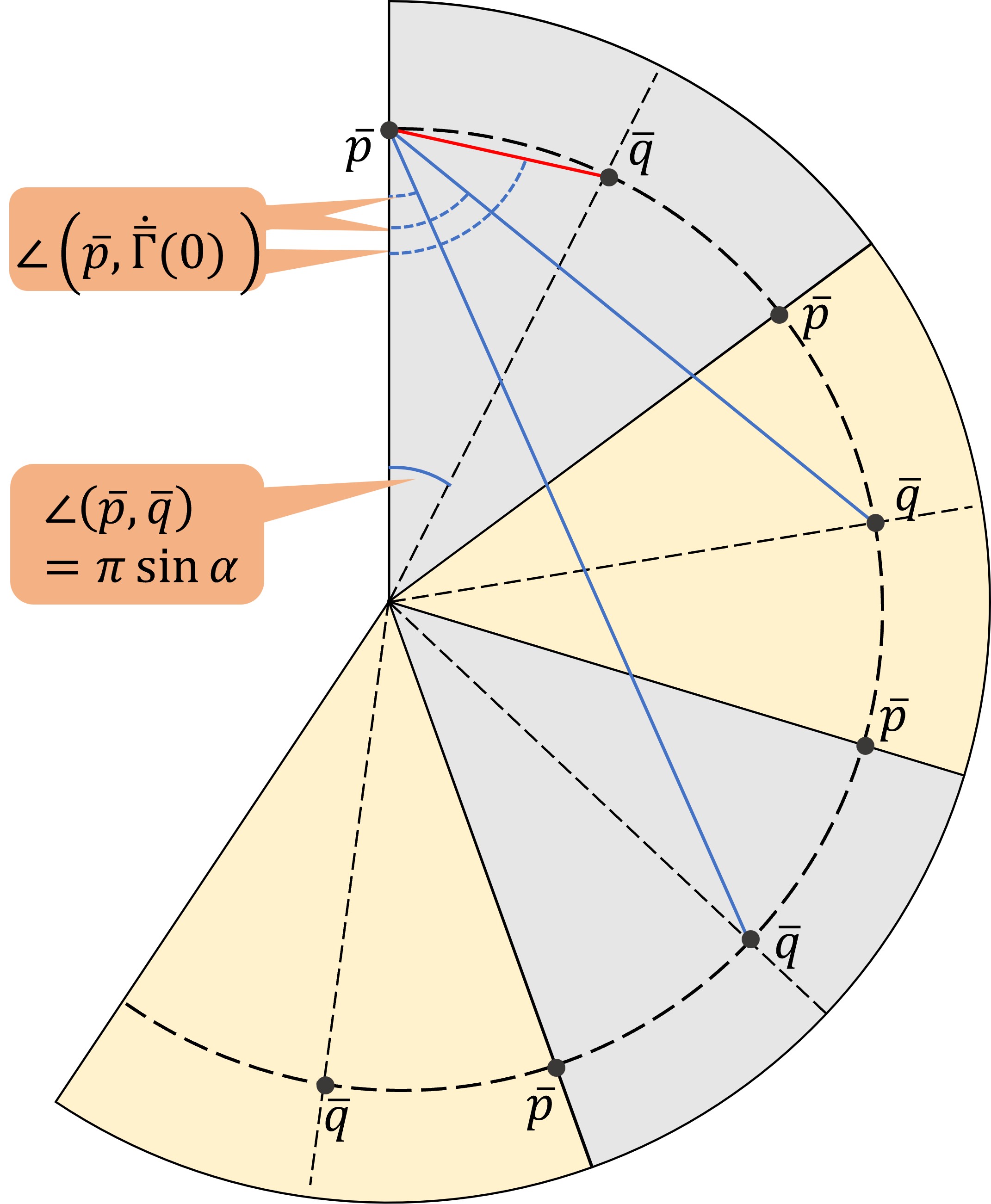}
\caption  {An illustration of $\Gz \in \Gz_{p,q}$.}
\label{gammapq}
\end{figure}
From Figure \ref{gammapq}, one could see that for each $k \in \nn$ such that $(2k-1)\sin\az \le 1$, there exists a pair $\bar \Gamma^\pm_k$ of geodesic segments joining $\bar p$ and $\bar q$, which
 leads to   geodesic segments $ \Gamma^\pm_k$  joining $ p$ and $  q$ and having
 winding number $k$.

 From Figure \ref{gammapq},  is clear that
the Euclidean angle between
 $ \bar p $ and the line segment $\bar  \Gamma^\pm_k$ (and also $\dot {\bar
  \Gamma}^ \pm_k(0)$) is
 $$\pm \frac{\pi-(2k-1)\pi\sin\az}{2} $$
 and
 the Euclidean length is
  $$2\rho \cos\{\pm \frac{\pi-(2k-1)\pi\sin\az}{2}\}=  2\rho \sin\frac{(2k-1)\pi\sin\az}{2}.$$
 This leads to   \eqref{gp1} and \eqref{gp2}.
\end{proof}

\begin{rem} \label{pi/2} \rm
   The above proof also works for two nonantipodal points $a=(\rho,\tz)$ and $b = (\rho,\tz)$. It is not hard to check that for any geodesic segment $\gz:[0,1] \to \ccc_\az^2$ with $\gz(0)=1$ and $\gz(1)=b$ parametrised by a constant multiple of arc-length (including the piecewise smooth generatrix segment), we always have that
   $\gz|_{(0,1)}$ is contained in the ball $B(o,\rho)$ and
   $$  \angle(\dot\gz(0),a) > \frac{\pi}{2}, \quad \angle(\dot\gz(1),b) < \frac{\pi}{2}.$$
   The precise values of the above angles are not important. However, the fact that they do not equal to $\pi/2$ is useful in the proof of Lemma \ref{delta} in Section 5.2.
\end{rem}

\begin{lem} \label{ab}
   Let $a,b \in  \pa B(o,1)$ with $\angle(a,b)>0$ and $\eta \in [0,1]$. Then for any geodesic segment $\gz:[0,1] \to \ccc_\az^2$ of $\ccc_\az^2$ joining $a$ and $\eta b$ such that
   $$   \gz \cap B(o,\eta) = \emptyset, $$
   we have
   \begin{equation}\label{angleab0}
     \angle(\dot\gz(1),-b) \ge \angle(a,b).
   \end{equation}
\end{lem}

\begin{proof}
   We work in the wedge domain $W_{2\pi \sin\az}$. From Figure \ref{angleab},
    \vspace*{1pt}
\begin{figure}[h]
\centering
\includegraphics[width=12cm]{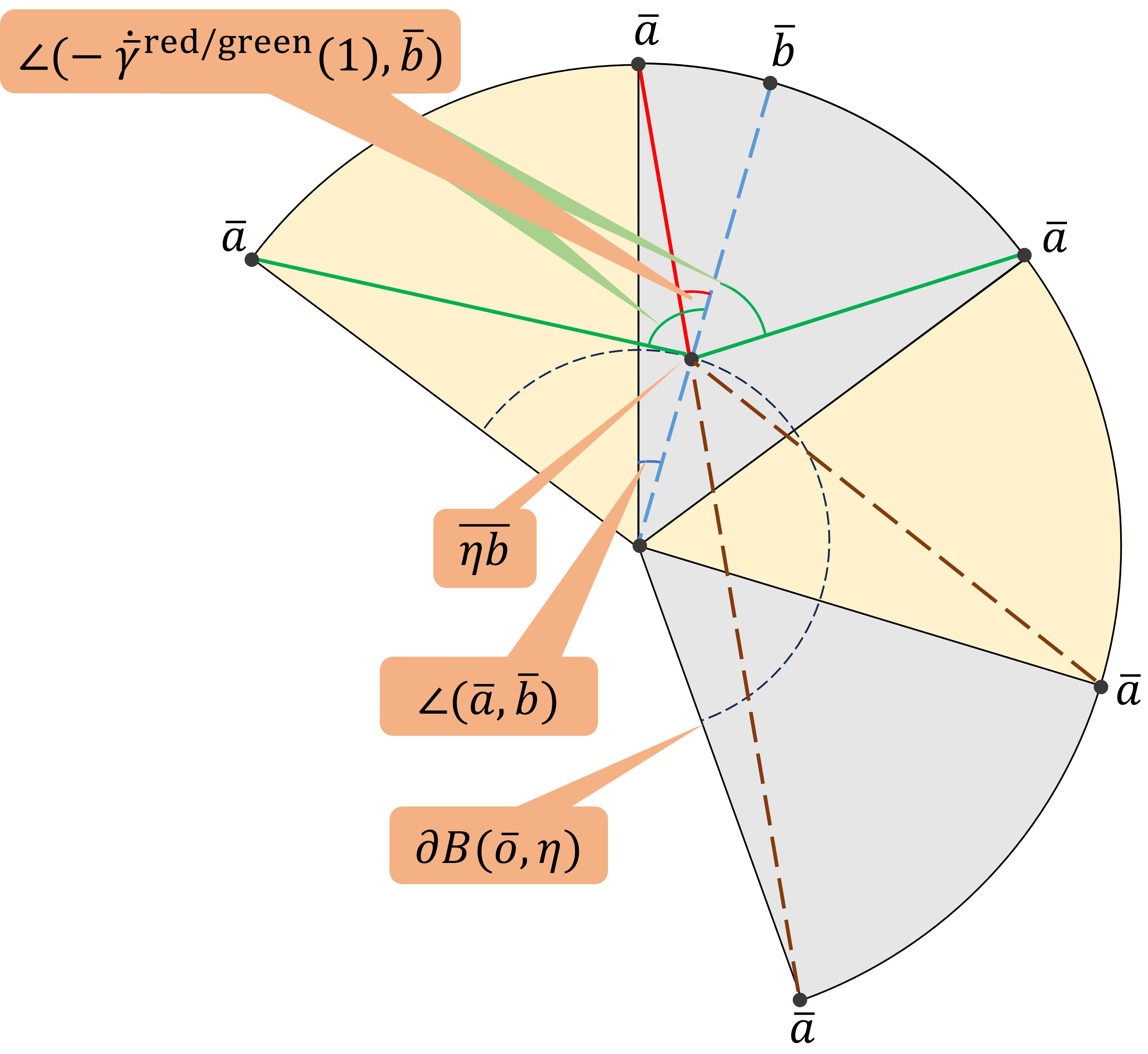}
\caption  {$\angle(-\dot{\bar \gz}(1),\bar{b})$ and $\angle(\bar{a},\bar{b})$.}
\label{angleab}
\end{figure}
we see that all the possible geodesic segment joining $\overline{\eta b}$ and $\bar a$ are the red line segment $\bar \gz^{\rm red}$ and two green line segments $\bar \gz^{\rm green}$ in the adjacent wedge domains without intersecting $B(\bar{o},\eta)$ where $\bar \gz^{\rm red}$ is the length minimising one (the brown dash line segments intersecting $B(\bar{o},\eta)$ are not satisfying the assumption of the lemma). Also from Figure \ref{angleab}, it is straightforward to see that for the length minimising line segment $\gz^{\rm red}$, one has
   $$ \angle(-\dot{\bar \gz}^{\rm red}(1),\bar{b}) \ge \angle(\bar a,\bar b).$$
   Moreover, note that for any non length minimising segment $\bar \gz^{\rm green}$,
   $$ \angle(-\dot{\bar \gz}^{\rm green},\bar{b}) \ge \angle(-\dot{\bar \gz}^{\rm red}(1),\bar b).$$
    Thus by \eqref{wed}, \eqref{angleab0} holds.
\end{proof}

\begin{rem} \label{geomin} \rm
Let $p=(\rho, \tz), q = (\rho, \tz+\pi) \in  \mathcal C_\az^2$ be two antipodal points.

 (i) Write
  $$\Gz_{p,q}=\left \{\Gamma_k^\pm : k\in [0,\frac{1}{2\sin\az}-\frac12) \cap \nn\right\} $$
  and
  $$ \mathscr{A}_\az:=\left\{ \frac{\pi-(2k-1)\pi\sin\az}{2} : k=1, \cdots, \lfloor\frac{1}{2\sin\az}-\frac12\rfloor  \right\}.$$
    If $k\ne j$, then 
$\Gz^\pm_k $ and $\Gz^\pm_j$ are not homopotic in $ \mathcal C_\az^2 \setminus\{o\}$.

  (ii)
 Observe that $\Gz^\pm_0
 $  are the length-minimising geodesic segments joining $p$ and $q$ and having
 winding number $0$ (in particular,  having no self-intersection).

 Given any $\xi\in \mathbb S^1$ with $\xi\perp G=G_p\cup G_q$,
 If $p\xi q$ are ordered clockwise, we set $L_\xi=\Gamma^+_0$;
  otherwise set  $L_\xi=\Gamma^-_0$.

  (iii) Recall $p =(\rho,\tz)$ and $q=(\rho,\tz^\star)$. It is easy to see $$ d_{\ccc_\az^2}(o,L_{\xi}) = \frac12 \rho \pi \sin\az.  $$
  Thus $\lim_{\rho \to \fz} d_{\ccc_\az^2}(o,L_{\xi}) = \fz$.
\end{rem}

Next, we focus on the geodesic lines of the standard cone $ \mathcal C_\az^2$.
Observe that for each $\az \in (0,\frac{\pi}{2})$, there is a unique integer
$$ K_\az \in (\frac1{2\sin\az}-\frac12,\frac1{2\sin\az}+\frac12]$$
such that
 $$0\le |2\pi K_\az \sin \az - \pi| \le \pi \sin \az.$$
 Write
 $${\bf K}_\alpha:= \frac{|2\pi K_\az \sin \az - \pi|}{\sin \az}\in[0, \pi].$$
\begin{lem}\label{line1}
  Let $\Gz: (-\fz,\fz)\to  \mathcal C_\az^2\setminus\{o\}$ be a geodesic line. Then
  $$
 \lim_{t \to +\fz}\angle(\Gz(-t),\Gz(t))=\lim_{t \to +\fz}\angle(\dot\Gz(-t), \dot\Gz(t))= \mathbf{K}_\az.$$
   \end{lem}

\begin{proof}
Let $ \Gamma$ be a geodesic line in $\mathcal C^2_\alpha$, which corresponds to
the geodesic line $\bar \Gz$  in the adjacent wedge domains.
By the same reason as in the proof of Lemma \ref{antigeo},
we know that   $\bar \Gz$ is a straight line in the adjacent wedge domains.
 See Figure \ref{kaz} where the red line is $\bar \Gz$ and we also draw $\bar \Gz$ in the upper grey wedge domain (the red dash lines together with the part of red line in the upper grey wedge domain).
\begin{figure}[h]
\centering
\includegraphics[width=8cm]{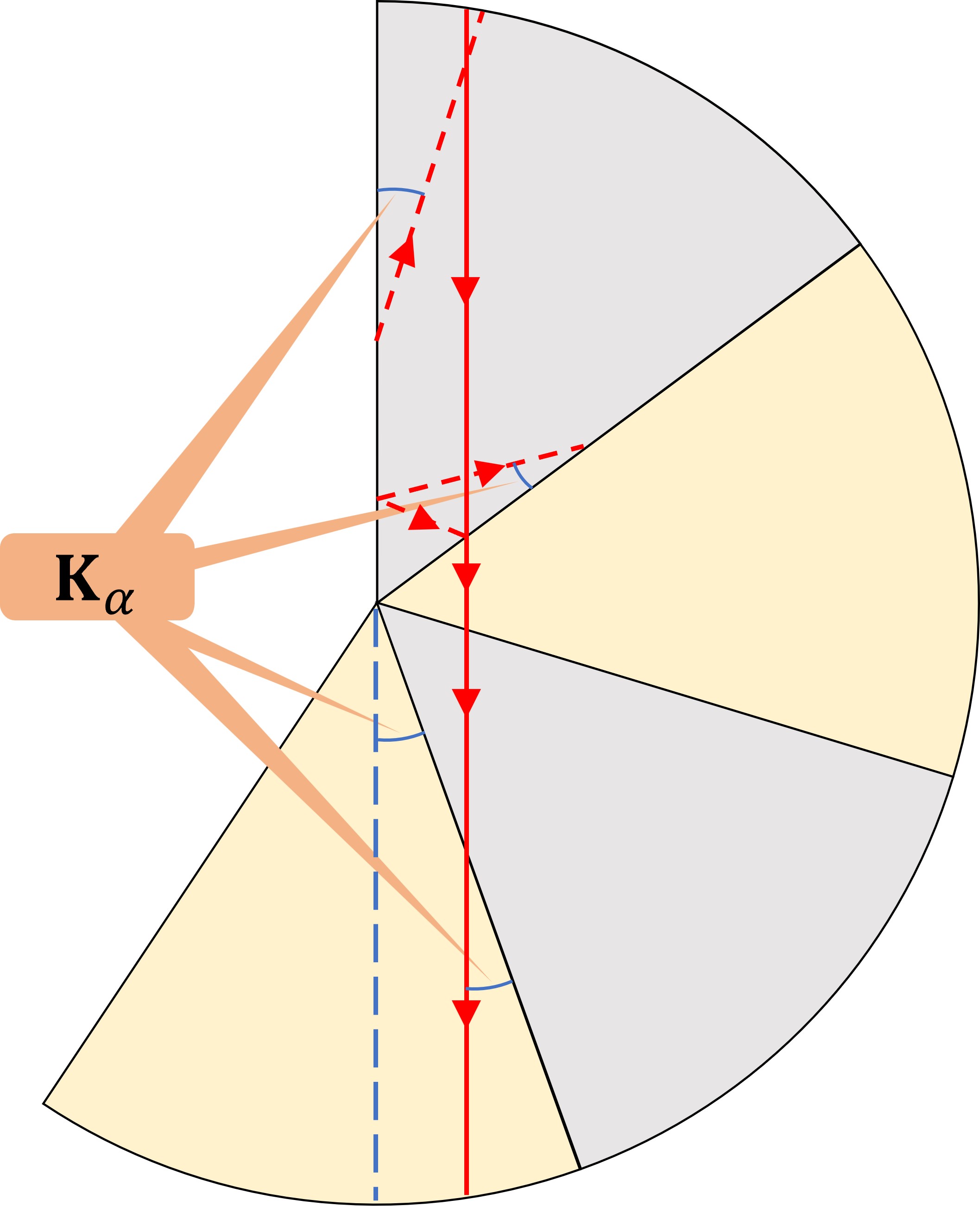}
\caption  {An illustration of the geodesic line $\Gz$.}
\label{kaz}
\end{figure}
From Figure \ref{kaz}, in the wedge domain  $W_{2\pi \sin \az}$ one see that the angles
$$\lim_{t \to +\fz}\angle(  \bar \Gz(-t), \bar \Gz(t) )\ \mbox{ and} \ \lim_{t \to +\fz} \angle (\dot{\bar \Gz}(-t), \dot{\bar \Gz}(t) )\ \mbox{ exist and coincide}$$
and  are given by
$  |2\pi K_\az \sin \az - \pi|.$
  Therefore,  in $\mathcal C^2_\alpha$, it follows that
$$\lim_{t \to +\fz}\angle(\Gz(-t),\Gz(t))=\lim_{t \to +\fz}\angle(\dot\Gz(-t), \dot\Gz(t))={\bf K}_\alpha$$
 as desired.
 \end{proof}

\begin{rem} \rm
   We remark that one could further investigate the relation of ${\bf K}_\alpha$ with the winding number of the geodesic line $\Gz$. However, since this is not needed in the following proofs and in higher dimensions, winding number for a curve is not defined, we will not make further investigation. The key point for the following proofs is that, Lemma \ref{line1} still holds for geodesic lines in higher dimensions cones $\ccc_\az^n$ with same ${\bf K}_\alpha$. See Lemma \ref{line2}.
\end{rem}

\subsection{  Geometry of $\mathcal C_\az^n$ with $n \ge 3$} \label{s22}

In this subsection, we work with the cone  $\mathcal C_\az^n$ with $n \ge 3$.
We start with the following lemma.

\begin{lem} \label{totalg}
(i) Given  any  $2$-dimensional subspace $P\subset\mathbb R^n$,
 $\mathcal C^2_{\alpha,P}:=(P, g_{C_n^\az}|_P)$ is isometric to the $2$-dimensional cone $ \mathcal C_\az^2$.
 Moreover, $\mathcal C^2_{\alpha,P}\setminus\{o\}$ is a  totally geodesic submanifold of $\mathcal C_\az^n\setminus\{o\}$, that is,   any geodesic in   $\mathcal C^2_{\alpha,P}\setminus \{o\}$ is also a geodesic in $\mathcal C_\az^n\setminus\{o\}$.

 (ii)   Any geodesic $\gamma$  in  $\mathcal C_\az^n\setminus\{0\}$ must be a
geodesic in $\mathcal C^2_{\alpha,P}\setminus\{0\}$
for  some 2-dimensional subspace $P$ of
$\mathbb R^n$.  Moreover, if $\gamma$ is contained in some generatrix, then $P$
is any plane containing $\gamma$; otherwise,
$P=\Span\{\gamma(t), \dot \gamma(t)\}$ for any possible $t$.
\end{lem}

 \begin{proof}
(i)  It is obvious that
 $\mathcal C^2_{\alpha,P}$ is isometric to  $ \mathcal C_\az^2$.
Note that   $\mathcal C^2_{\alpha,P} \setminus \{o\}$
is the fixed set  of the reflection  $I_{P}$ with respect to  $P$,
which is defined  by
$ I_P(x_P+x_{P^\perp})=x_P-x_{P^\perp}$ for any $ x_P+x_{P^\perp}\in P\oplus P^\perp=\mathbb R^n$
and  obviously is an isometry of   $(\mathcal C_\az^n \setminus \{o\},g_{\mathcal C_\az^n})$.
 By \cite[Theorem 1.10.15]{k95},
 $\mathcal C^2_{\alpha,P}\setminus\{o\}$ is a  totally geodesic submanifold of $\mathcal C_\az^n\setminus\{o\}$.

 (ii)  Assume that
  $\gamma$ is contained in some generatrix $G_a$ of $\mathcal C^n_\alpha$.
  Then $G_a$ is a generatrix  in $\mathcal C^2_{\alpha,P}$ for any 2-dimensional subspace $P$ containing this generatrix.
  Thus  $\gamma$ is a geodesic in $\mathcal C^2_{\alpha,P}$.

If $\gamma$ is not contained in any generatrix,
we know that $\gamma(t_0)$ and $\dot\gamma(t_0)$ is linearly independent for any $t_0$.
Without loss of generality, we may assume that $t_0=0$.
Write $P=\Span\{\gamma(0),\dot \gamma(0)\}$ and
denote by   $\exp^{\mathcal C_{\az,P }^2 }$
 the exponential map on   $ \mathcal C_{\az,P }^2$.
  The geodesic flow $t\mapsto \eta(t)= \exp_x^{ \mathcal C^2_{\alpha,P}}(tv)$
   is not a generatrix, hence  is well defined for all $t \in \rr$
   and  does not passing through $o$ for any $t \in \rr$.
   Since  $ \mathcal C_{\az,P }^2\setminus\{o\}$ is totally geodesic submanifold of $ \mathcal C_{\az}^n\setminus\{o\}$,
  $\eta$  is also a geodesic of $\mathcal C_\az^n$ with $ \eta(0)=x$ and $\eta^\prime(0)=v$.
  By the uniqueness of geodesics, $\gamma $ must be part of $\eta$ and hence $\gamma$ is a geodesic in $ \mathcal C^2_{\alpha,P}\setminus\{o\}$.
  Consequentely,  $P=\Span\{\gamma(t),\gamma^\prime(t)\}$ for all possible $t$.
 \end{proof}

 Below
  we always fix any pair of antipodal points $p=(\rho,\theta)$ and $q=(\rho,\tz^\star)$ in $\mathcal C_\az^n\setminus\{o\}$.
Denote by $\Gz_{p,q}$ all geodesic segments
   joining $p$ and $q$ in $\mathcal C^n_\alpha$.
   With the help of Lemma \ref{totalg},
  we  are able to fully characterize $\Gz_{p,q}$ and $\mathcal L_{p,q}$.

   To this end,  we first
decompose   $\mathcal C_\az^n$  as the uion of a family of $2$-dimensional cones.
Denote by $G_p$ and $G_q  $   the pair of two generatrices $p$ and $q$ lie in.
Then
 $G:=G_p\cup G_q
  $ is a line passing through the origin $o$.
Denote by $\mathbb S^{n-2}_G $  the unit sphere of the $(n-1)$-dimensional subspace
$G^{\perp}$ of  $\mathbb R^n$, that is, unit vectors in  $\mathbb R^n$ which   are
perpendicular to $G$.
   For each $\xi \in \mathbb S^{n-2}_G$,  write  $P_\xi$  as
   the  $2$-dimensional subspace of $\rr^n$ spanned by $ \xi $ and $ G$, that is,
 $P_\xi:=G\oplus \mathbb R\xi
  $ 
and set
  $$ \mathcal C_{\az,\xi}^2:= ({P_\xi}, g_{C_n^\az}|_{P_\xi}). $$
Note that  $\mathcal C^2_{\az,\xi} $ is isometric to the $2$-dimensional cone $ \mathcal C_\az^2$, and by Lemma \ref{totalg},
$\mathcal C^2_{\az,\xi} \setminus\{o\}$   is a totally geodesic submanifold of $\mathcal C_\az^n \setminus \{o\}$.
 Obvisouly,
  $$  \mathcal C_\az^n = \bigcup_{\xi
 \in \mathbb S^{n-2}_G} \mathcal C_{\az,\xi}^2  $$
  and
  $$  \mathcal  C_{\az,\xi}^2  \cap \mathcal C_{\az,\zeta}^2 = G   \quad \mbox{ if } \xi \ne \pm\zeta.  $$

   For each $\xi\in  \mathbb S^{n-2}_G$,
  write $\Gz_{p,q}(\xi)$ as the  collection  of all geodesic segments
   joining $p$ and $q$ in $\mathcal C^2_{\alpha,\xi} \setminus\{o\}$.
   As a direct consequence of Lemma \ref{totalg}, one has
   \begin{equation*}
     \Gz_{p,q}=\bigcup_{\xi\in  \mathbb S^{n-2}_G}\Gz_{p,q}(\xi)
   \end{equation*}
   and
   $$  \Gz_{p,q}(\xi
   )\cap \Gz_{p,q}(\zeta) =\emptyset  \quad \mbox{ if } \xi \ne \pm\zeta .$$
 Recall from Lemma \ref{antigeo} that
   for each $\Gz \in \Gz_{p,q}$, the $\ccc_\az^n$-length of $\Gz$ satisfies
 \begin{equation}\label{nantigeo}
    \ell_{\ccc_\az^n} (\Gz) = \ell_{\mathcal C_\az^2}(\Gz) \in \left\{
     2\rho \sin\frac{(2k-1)\pi\sin\az}{2}:  k \in [0,\frac{1}{2\sin\az}-\frac12 ) \cap \nn \right\}.
   \end{equation}

   Moreover, for any $\xi \in \mathbb S^{n-2}_G $,
by Remark 2.5 (iii), there are  two  length-minimising geodesic segments $\Gamma^\pm_{0,\xi}$    joining $p$ and $q$, which are contained in
 $\mathcal C_{\az,\xi}^{2}$.
 We write $L_\xi$ the one which has  the same orientation as
 $p,\xi,q$ (intersecting the generatrix $\rr_+\xi$)   and $L_{-\xi}$ the one which has  the same orientation as
 $p,-\xi,q$ (intersecting the generatrix $\rr_+(-\xi)$).
 Recall that $$
  \ell_{\mathcal C_{\az }^n}(L _{\pm\xi})=
  \ell_{\mathcal C_{\az,\xi}^2 }(L _\xi^\pm)= 2\rho \sin(\frac{\pi}2 \sin\az)=d_{\mathcal C^n_\alpha}(p,q).$$
  We  reparametrize $  L_\xi$ to get $L^\star_\xi:[0,1]\to \mathcal C^n_\alpha$ with constant speed
$2\rho \sin(\frac{\pi}2 \sin\az)$. By abuse of notation, we write $L_\xi^\star$ as $L_\xi$

  Write $ \mathcal L_{p,q}$  as  the collection of all
 length minimizing geodesic segments joining $p,q$ in $\mathcal C^n_\alpha$  reparameterized to
 have constant speed $d_{\mathcal C^n_\alpha}(p,q)$. 
 Equip $\cl _{p,q}$ with $C^0([0,1])$-distance.
\begin{lem}\label{nmingeo}
 The map $\tau:\xi\mapsto L_\xi$ 
 is a homeomorphism from
$\mathbb S^{n-2}_{G}$   to $\cl_{p,q}$. In particular, $\cl_{p,q}=\{L_\xi: \xi \in  \mathbb S^{n-2}_G\}$ and
  $\cl_{p,q}  $ is homeomorphic to the standard sphere $\mathbb S^{n-2}$.
\end{lem}

 \begin{proof}
 Our construction of $\tau$ implies that  $\tau$ is  injective. As a consequence of Lemma \ref{totalg},  $\tau$ is surjective.
By the continuous dependence of exponential map $\exp_p(t\xi)$ in both $t \in \rr$ and $\xi \in \mathcal{T}_p \ccc_{\az}^n$, we know $\tau$ is continuous. Thus $\tau$ is a homeomorphism.
  \end{proof}

  \begin{rem} \rm
    Indeed, for any $\xi,\zeta \in \mathcal{S}_G^{n-2}$, one could check that the $C^0$-distance $$\|L_\xi-L_\zeta\|_{C^0([0,1])}= \rho\pi \sin \az \sin (\angle(\xi,\zeta)\sin \az),   $$
    which equals to the $\ccc_{\az}^n$-distance between the middle points of $L_\xi$ and $L_\zeta$ and is comparable to $|\xi-\zeta|$. This quantity is not important in the following proofs. However, the fact that $\tau$ is a homeomorphism is crucial for constructing of the homotopy class in Section 4.
  \end{rem}

For each $\xi \in \mathbb S^{n-2}_G$, write $ L_\xi([0,1]) \subset \rr^{n}$ as the image of $L_\xi: [0,1] \to \ccc_\az^n$($=\rr^{n}$ as sets).
Then
 $$ L_\xi([0,1])
 \cap L_\zeta([0,1]) = \{p,q\} \quad \forall \zeta \ne\zeta\in \mathbb S^{n-2}.  $$
 Set
  $$\cl^T_{p,q}:=\bigcup_{\xi \in \mathbb S^{n-2}_G } L_\xi([0,1]) \subset \rr^n.$$
and equip it with   the topology induced by the Euclidean topology.
We have the following.
\begin{lem}\label{nmingeo2}
 $\cl^T_{p,q}$   is homeomorphic to $\mathbb S^{n-1}$.
In particular, the apex $o$ belongs to the bounded component of $\rr^n \setminus \cl^T_{p,q}$.
\end{lem}

\begin{proof}

For each $\xi\in \mathbb \mathbb S^{n-1}_G$,
  $L_\xi$ and $L_{-\xi}$ are embedded curves contained in $C_{\az,\xi}^2 $ with winding number zero with respect to $o$.
  This implies that
$\mathbb S^1_{\xi}=L_\xi([0,1]) \cup L_\zeta([0,1])$
  is topologically an $\mathbb S^1$ and  $o$ bolongs to the bouded component of $P_\xi\setminus  \mathbb S^1_\xi$.
   By the rotational symmetry of $\mathcal C_{\az}^n$, as $\xi$ varies in $\mathbb S^{n-2}$,
   we know that $\mathbb S^1_{\xi}$ rotates in $\rr^n$ with respect to the axis $G $, thus
  $\cl^T_{p,q}$ forms an $(n-1)$-dimensional subset in $\rr^n$ homeomorphic to $\mathbb S^{n-1}$.
\end{proof}

%
%

The next lemma is parallel to Lemma \ref{line1}.
\begin{lem}\label{line2}
  Let $\Gz: (-\fz,\fz)\to \mathcal C^n_\az\setminus\{o\}$ be a geodesic line. Then,
  $$
  \mathbf{K}_\az:=\lim_{t \to +\fz}\angle(\Gz(-t),\Gz(t))=\lim_{t \to +\fz}\angle(\dot\Gz(-t), \dot\Gz(t))= \frac{|2\pi K_\az \sin \az - \pi|}{\sin \az}$$
   where $K_\az$ is the unique integer such that $\mathbf{K}_\az \in [0,\pi]$.
   Moreover,
   $$ \lim_{t \to +\fz} \angle(\Gz(-t),-\dot\Gz(-t)) = \lim_{t \to \fz} \angle(\Gz(t),\dot\Gz(t)) = 0. $$
\end{lem}

Directly from Lemma \ref{line2}, we have the following corollary.

\begin{cor}\label{line3}
  Let $\{\Gz^N: [0,1]\to \overline {B(o,N)} \}_{N \in \nn, N \ge 2}$ be a sequence of geodesic segments of $\mathcal C_\az^n$ parametrised proportional to arc-length such that
  $$ \Gz^N(0)\cup \Gz^N(1) \subset \pa B(o, N) $$
  and
  \begin{equation}\label{tangent}
    \min\{r>0: \Gz^N \cap \pa B(o, r) \ne \emptyset\} =1, \quad \forall N \ge 2.
  \end{equation}
   Then,
   \begin{equation}\label{second}
     \lim_{N \to +\fz}\angle(\Gz^N(0),\Gz^N(1))=\mathbf{K}_\az \mbox{ and }\lim_{N \to +\fz}\angle(\dot\Gz^N(0), \dot\Gz^N(1))=\mathbf{K}_\az.
   \end{equation}
\end{cor}

\begin{proof}

Reparametrize  $ \Gamma^N$ to the geodesic segment
$ \Gamma^N_\star:[t_N^-,t^+_N]$ under arc-length of $\ccc_\az^n$ with $ \Gamma^N_\star (0)\in  \pa B(o, 1)$.  Note that $t^+_N$, $-t^-_N\ge N-1$.
By abuse of notation we write $\Gamma^N_\star$ as $\Gamma^N$.
By condition \eqref{tangent}, we know that
$$ \dot \Gamma^N (0) \mbox{ is tangent to } \pa B(o, 1), \quad \forall N \ge 2. $$
By the rotational symmetry of $\ccc_\az^n$, we know
there  is  a  rotation  $\car^N \in O(n)$ which sends
$   \Gz^N(0)   $ to
    $   \Gz^{2}(0)  $
    and  the tangential
    $ \dot \Gz^{N }(0)$ to     $\dot \Gz^{2}(0)$ for each $N \ge 3$.
    By the uniqueness of geodesics starting at $x_0$ of a given initial velocity, we know that
  $\Gz^{2}$ is the restriction of $\car^N (\Gz^{N})|_{[t^-_2,t^+_2]}$.
  Moreover,  for any $j\le N-1$,
$  \car^j(\Gz^j) $ is the restriction of $\car^N(\Gz^{N})|_{[t^-_j,t^+_j]}$.

 Therefore $\Gz:= \Gz^2 \cup [\cup_{N \ge 3} \car^N(\Gz^N)]$ is a geodesic line of $\mathcal C_\az^n$.
    Applying Lemma \ref{line2}, we have
    $$
\lim_{N \to +\fz}\angle(\car^N(\Gz^N(0)),\car^N(\Gz^N(1)))=\lim_{N \to +\fz}\angle(\Gz(t_N^-), \Gz(t_N^+))=\mathbf{K}_\az.$$
Since $\car^N$ is an isometry of $\mathcal C_\az^n$, we have
 $$
\lim_{N \to +\fz}\angle(\Gz^N(t_N^-),\Gz^N(t^+_N))=\lim_{N \to +\fz}\angle(\car^N(\Gz^N(0)),\car^N(\Gz^N(1)))=\mathbf{K}_\az.$$
The second equality in \eqref{second} can be similarly proven. The proof is complete.
\end{proof}

\section{Asymptotically conical manifolds}\label{s23}


Let $(M, g, o) = (\rr^n,g,o)$   be  a complete $C^{2}$  non-compact manifold with reference point $o$.
Assume  that $M$ is an
\emph{asymptotically conical manifold} with asymptotical angle  $\az \in (0,\pi/2)$
and  asymptotic decay rate $\mu>0$   if
  the difference    $A :=g  -g_{\mathcal C_\az^n} $
 is a $C^{2}$ symmetric $(0,2)$-tensor  aways from  $o$,
  and has  asymptotic decay rate $\mu$.   Here, recall that $g_{\mathcal C_\az^n}$ is the Riemannian metric on the standard cone $\mathcal C_\az^n \setminus  \{o\}$ with the open  angle $\alpha$. Write
 the  (0,2)-tensor  $A$ as
  \begin{equation}\label{apolar}
 A(\rho, \tz)
  =   a_{\rho\rho}(\rho, \tz)\,d\rho\otimes\,d \rho + \rho^2 \sum_{i,j=1}^{n-1}a_{\tz_i\tz_j}(\rho, \tz)\,d \tz_i \otimes\,d \tz_j + 2  \sum_{i}^{n-1}\rho a_{\rho\tz_i}(\rho, \tz)\,d\rho\otimes\,d\tz_i
 \end{equation}
 under the standard polar coordinate $ (\rho, \tz) $ of $\rr^n$, that is,
$$(\rho, \tz) =(\rho, \tz_1,\cdots,\tz_{n-1})
\in(0,\infty) \times [-\pi,\pi)\times[-\pi/2,\pi/2)^{n-2}.$$
  We  say $A$
  has the asymptotic decay rate $\mu>0$ if
for any $a\in \{a_{\rho\rho},  a_{\rho\tz_j} ,a_{\tz_i\tz_j}\}$ and
  any multi-index  $I=(I_\rho,I_{\tz_1},\cdots,I_{\tz_{n-1}})$ with length $ |I|\le 2$,  one has $$ |\pa^I a(\rho, \tz)|= O(\rho^{-\mu- I_{\rho}})\ \mbox{as $ \rho \to \infty$.} $$
Note that, under the  Cartesian coordinate $x= (x_1,\cdots,x_n)$      the metric $g$ and also the  difference $A$ are written as
\begin{equation}\label{gcar}
   g (x) =\sum_{i,j=1}^n g_{ij}(x) \, dx_i \otimes dx_j = g_{\mathcal C_\az^n}(x) + A(x)
\end{equation}
and
\begin{equation}\label{acar}
  A(x)= \sum_{i,j=1}^{n}a_{ij}(x) \, dx_i \otimes dx_j, \quad \forall x=(x_1,\cdots,x_n)\in \rr^n\setminus\{o\}.
\end{equation}
The asymptotical decay rate $\mu$ of $A $ means that for multi-index $I=(I_1,\cdots,I_{n})$ with length $ \le2$, we have
\begin{equation*}
 |\partial^I_x a_{ij}(x)|= O (|x|^{-\mu-|I|}) \ \mbox{as $x\to\infty$}.
\end{equation*}

\begin{rem}\rm
(i) We remark that
 $M$ is always assumed to be homeomorphic  to $\mathbb R^n$ in this paper.
 One could easily check that   $M$ is always quasi-isometric to the Euclidean space $\mathbb R^n$.   Since quasi-isometric spaces has the  same ideal boundary, we know $\pa_\fz M=\pa_\fz \rr^n = \mathbb S^{n-1}$ and we could identify $\pa_\fz M$ as the unit sphere with the round metric. We refer to \cite{bs07} for this fact.

(ii)  Below,  \begin{equation}\label{afz}
  \|A\|_\fz:=   \max_{1\le i,j \le n}\|a_{ij}\|_{L^\infty(\mathbb R^n)}<\infty. 
 \end{equation}
Consequently,  the Euclidean ball $B(y,r)$ and $g$-geodesic ball $B_g(y,r)$ are always comparable for any $y\in M$ and $r>0$.

 \end{rem}

For  any $\ez\in(0,1/2)$,
set $$R_\ez:=\min\left\{R>0 :  \max_{1\le i,j,k \le n}\{|a_{ij}(x)|\} \le \ez  \quad \forall x \in \rr^{n} \setminus B(o,R ).\right\}
$$
Then 
\begin{equation}\label{rez}
  R_\ez \ge C\ez^{-1/\mu}
\end{equation}
Set
\begin{equation}\label{mez}
  M_{\le \ez}:=  \rr^{n} \setminus B(o,R_\ez).
\end{equation}
One   has
 $$  \max_{1\le i,j,k \le n}\{|a_{ij}(x)|\} \le \ez, \quad \forall x \in M_{\le \ez}. $$
In particular,
\begin{equation}\label{compare2a}
\mbox{$(1-C\ez)|\xi|_{g_{\mathcal C_\az^n}} \le |\xi|_g \le (1+C\ez)|\xi|_{g_{\mathcal C_\az^n}} $
 for all  $x \in M_{\le \ez}$ and $\xi \in \mathcal{T}_xM$ }
 \end{equation}
 and
  \begin{equation}\label{compare2} (1-C \ez) \ell_{\mathcal C_\az^n}(\gz) \le \ell_{g}(\gz) \le  (1+C \ez) \ell_{\mathcal C_\az^n}(\gz)\ \mbox{for any curve $\gamma\subset M_{\le \epsilon}$}    \end{equation}
  for some absolute constant $C >0$ depending only on $n$  and $\az$.

%

We also need the fact  from \cite[Lemma 2.15]{bfm22}.
\begin{lem}
\label{bfm}
 One has
  $$  \lim_{|x| \to \fz}\frac{d_g(o,x)}{|x|} =  \lim_{d_g(o,x) \to \fz}\frac{d_g(o,x)}{|x|} =1. $$
\end{lem}

 For $\lz>0$,   recall the scaling map $\mathcal{D}_\lz$
\begin{equation}\label{slz}
   \mathcal{D}_\lz(x) = \lz x, \quad x \in\rr^n
\end{equation}
and
consider the rescaled  metric
  \begin{equation}\label{scale}
    g^{(\lz)}(x) = \lz^{-2} \mathcal{D}_{\lz}^\ast g(x), \quad x \in \rr^n , \ j \in \nn.
  \end{equation}
We denote by $M^{(\lz)}=(\mathbb R^n , g^{(\lz)})$.
The following comes from the definition directly; see for example \cite[Lemma 2.14]{bfm22}.

\begin{lem}  \label{convg}
 As $\lz \to \fz$, the Riemannian manifolds $ M^{(\lz)}\setminus\{o\}$ converges
(locally uniformly in the $C^{2}$-topology) to the standard cone $(\mathcal C_\az^n\setminus \{o\}, g_{\mathcal C_\az^n})$.
\end{lem}

  For any $\lz>1$, the region $M^{(\lz)}_{\le \ez}$ can be similarly defined as in \eqref{mez}.
For all $\ez \in(0,1/2)$, it holds that
\begin{align} \label{(lz)}
 M^{(\lz)}_{\le \ez} = \mathcal{D}_{\lz^{-1}}(M_{\le \ez})
  \end{align}

\subsection{A key convergence lemma under properness of  exponential map }

Denote by $\exp^M$  the exponential map  of  $M$.
We call $\exp^M $   is   proper  at  a point $x\in M$ if
   the preimage
 $(\exp_x^M )^{-1}(K) 
 $
of   any compact set $K\subset M$ is  compact in  $\mathcal \mathcal{T}_xM$.

\begin{lem}\label{inj}
  The exponential map $\exp ^M$ is proper at each point if and only if for each $x \in M$ and $r>0$, there exists $T_{x,r}>0$ such that  $\exp_x^M(t\xi)\notin \overline{B(0,r)}$
  for any $\xi\in \mathcal{T}_xM$ and $ t>T_{x,r}/|\xi|$.
\end{lem}

\begin{proof}

Assume that the exponential map $\exp^M $  is proper  at each  point $x\in M$.
Thus the preimage  $(\exp_x^M )^{-1}(\overline{B(0,r)})$ is compact, and hence contained in some ball in  $\mathcal{T}_xM$ with radius $T_{x,r}$. Therefore, for any $\xi\in \mathcal{S}_xM$
and for $t>\mathcal{T}_{x,r}$, we know   $t\xi\notin (\exp_x^M )^{-1}(\overline{B(0,r)})$, which is equivalent to $\exp_x(t\xi) \notin\overline{B(0,r)}$.

Conversely,  for each $x\in M$ and $r>0$, assume  that  there exists $T_{x,r}>0$ such that $\exp_x^M(t\xi)\notin \overline{B_g(0,r)}$
  for any $\xi\in \mathcal{T}_xM$ and $t>T_{x,r}/|\xi|$.
  Then
  $$(\exp_x^{M})^{-1}(\overline {B(o,r)})\subset 
   \{t\xi: \xi\in \mathcal{S}_xM, 0\le  t\le T_{x,r} \}.$$
Thus $ (\exp_x^{M})^{-1}(\overline {B(o,r)})$ is bounded and hence compact.
 \end{proof}

We need the following result on the uniform bound of the length of a sequence of geodesic rays.

\begin{lem} \label{geoflow}
Assume that
 the exponential map  is   proper  at  each point.
  Let $\{\gz_j:[0,\fz)\to M\}_{j \in \nn}$ be a sequence of geodesic rays of $(M,g)$ parameterised by arc-length of $g$ such that
  $$  (\gz_j(0),\dot \gz_j(0)) \in \pa B(o,r) \times \mathcal{S}_{\gz_j(0)}M  $$
  and
  \begin{equation}\label{initial}
     \lim_{j  \to \fz} (\gz_j(0),\dot \gz_j(0)) = (p,\xi)
  \end{equation}
  for some $r>0$, $p \in \pa B(o,r)$ and $\xi \in \mathcal{S}_{p}M$.
  Define
  $$  T_j:= \max\{t>0: |\gz_j(t)| =r \}. $$
  Then
  \begin{equation}\label{dj}
    \sup_{j}T_j= \sup_{j}\ell_g(\gz_j|_{[0,T_j]})<\infty.
  \end{equation}
\end{lem}

\begin{proof}
Since  $ \gz_j(t)=\exp^M _{\gamma_j(0)}(t\dot \gz_j(0))$, by Lemma \ref{inj}, we know
  $T_j<\fz$ for all $j \in \nn$.
  Since $\gz_j$ is arc-length parametrised, we have $\ell_g(\gz_j|_{[0,t]})=t$ for all $t>0$. Thus, it suffices to show that    $ \sup_{j}T_j <\infty.$
   We consider the following two cases separately.

  \medskip
  \emph{Case 1. There exists $D_\ast>0$ such that $\gz_j|_{[0,T_j]} \subset B(o,D_\ast)$ for all $j \in \nn$.}
  We argue by contradiction.
  Assume there exists a subsequence $\nn_1 \subset \nn$ such that
  $$ \limsup_{\nn_1 \ni j \to \fz} T_j =\fz. $$

  Consider the geodesic ray $ \gz(t)=\exp^M _{x}(t\xi)$.
   For $r>0$, let
     \begin{equation}\label{taud0}
      \tau = \max\{ t>0: |\gz(t)| =D_\ast+1\}.
    \end{equation}
   Since $\gamma_j \to \gamma$ locally unfiormly,
   for sufficiently large $j$ we have
  $$ |\gamma_j(t)-\gamma(t)|\le \frac1{10}\quad\forall t<\tau  .$$
  For  sufficiently large $j$, one has $T_j>D_\ast+1$ and hence
   $$\gz_j|_{[0,\tau ]} \subset \gz_j|_{[0,T_j]} \subset B(o,D_\ast).$$
   From this we deduce that
    $$ \gamma|_{[0,\tau_{D_\ast+1}]} \subset B(o,D_\ast+ \frac{1}{10}),  $$
    in particular,
  $$ \gamma(\tau_{D_\ast+1}) \in B_g(o,D_\ast+ \frac{1}{10}).  $$
  This contradicts with $\gamma(\tau_{D_\ast+1})\in \partial B_g(o,D_\ast+ 1)$.

  \medskip
  \emph{Case 2.   There exists a subsequence $\mathbb N_2$
  and an increasing sequence $\{D_j\}_{j \in \nn_2}$ with $\lim_{\mathbb N_2\ni j \to \fz} D_j = +\fz$ such that
  $$\max\{|\gz_j(t)|: t \in [0,T_j]\} =D_j, \quad \forall j \in \nn_2.$$  }
   Up to some subsequence and we may assume that $\mathbb N_2=\mathbb N$.
   For each $j \in \nn $, let
  $$s_j=\max\{0<s<T_j: |\gz_j(s)|_g=D_j\}.$$
  We further consider two subcases.

  \medskip
  \emph{Subcase 2-1. There exists $N>0$ such that $T_j \le N D_j$ for all  $j \in \nn$.}
  Define $\bar \gz_j: [0,1] \to M$:
  $$ \bar \gz_j(t) := \mathcal{D}_{D_j^{-1}} \circ \gz_j(T_j t)  $$
  where we recall $\mathcal{D}_{\lz}: x \mapsto \lz x$ is the scaling map. Thus we know
  \begin{equation}\label{Bo1}
    \bar \gz_j \subset \overline{B(o,1)}, \quad  \bar \gz_j(T_j s_j) \in \pa B(o,1)
  \end{equation}
  and
  $$ |\dot{\bar \gz}_j(t)|_{g^{(D_j)}}= \frac{T_j}{D_j}|\dot \gz_j(t)|_{g} \equiv \frac{T_j}{D_j} , \quad \forall t \in [0,1], \ j \in \nn_2.  $$
  In particular,
  \begin{equation}\label{td}
    \ell_{g^{(D_j)}}(\bar \gz_j) = \frac{T_j}{D_j} \le N.
  \end{equation}
  Since $\lim_{j \to \fz} D_j  =\fz$, we know $g^{(D_j)}$ converges to $g_{\mathcal C_\az^n}$ and
  $$\lim_{j \to \fz}\bar \gz_j(0) = \lim_{j \to \fz} \frac{\gz_j(0)}{D_j} = \lim_{j \to \fz}\bar \gz_j(1) = \lim_{j \to \fz} \frac{\gz_j(t_j)}{D_j} = o.$$
Fom Arzela-Ascolli Theorem  it follows that, up to subsequences,
  $\{\bar \gz_j\}_{j \in \nn }$ converges in $C^2([0,1],M)$ to a geodesic segment $\gz_\fz: [0,T] \to \mathcal C_\az^n$ of the standard cone $\mathcal C_\az^n$ with $\gz_\fz(0)=\gz_\fz(1)=o$. Hence $\gz_\fz$ is a geodesic loop passing through the apex $o$. Moreover, by \eqref{Bo1}, we know
  $\gz_\fz \subset \overline{B(o,1)}$ with some point lies in $\pa B(o,1)$.
  This is impossible, since all the geodesics  in $\mathcal C_\az^n$  passing through $o$ are the generatrices and there is no geodesic loop of the standard cone passing through the apex $o$.

\medskip
  \emph{Subcase 2-2.   $\limsup_{  j \to \fz} T_j /D_j =\fz$.}
  Up to subsequence, we may assume that $\lim_{  j \to \fz} T_j /D_j =\fz$.
  Define $\hat \gz_j: [0,T_j /D_j] \to M$:
  $$ \hat \gz_j(t) := \mathcal{D}_{D_j^{-1}} \circ \gz_j(D_j t).  $$
   We have
  \begin{equation}\label{Bo6}
    \hat \gz_j \subset \overline{B(o,1)}, \quad  \hat \gz_j(T_j s_j) \in \pa B(o,1)
  \end{equation}
  and
  $$ |\dot{\hat \gz}_j(t)|_{g^{(D_j)}}= |\dot \gz_j(t)|_{g} \equiv 1, \quad \forall t \in [0,1], \ j \in \nn.  $$
  In particular, for each $N>0$,
  \begin{equation*}
    \ell_{g^{(D_j)}}(\hat \gz_j)|_{[0,N]} = N.
  \end{equation*}
  Since $\lim_{j \to \fz} D_j  =\fz$, we know $g^{(D_j)}$ converges to $g_{\mathcal C_\az^n}$ and
  $$\lim_{j \to \fz}\hat \gz_j(0) = \lim_{j \to \fz} \frac{\gz_j(0)}{D_j} = o.$$
  For $N >0$ sufficiently large, again by Lemma \ref{convg}, we know that, up to subsequence,
  $\{\hat \gz_j|_{[0,N]}\}_{j \in \nn}$ converges in $C^2([0,N],M)$ to a geodesic segment $\gz_\fz^N: [0,N] \to \mathcal C_\az^n$ of the standard cone $\mathcal C_\az^n$ with $\gz_\fz(0)=o$.
  By \eqref{Bo6},   $  \gz_\fz^N \subset \overline{B(o,1)}.  $
  Since $\ell_{g^{(D_j)}}(\hat \gz_j|_{[0,N]}) =N$, we know
  $$\ell_{\mathcal C_\az^n}(\gz_\fz^N)= \lim_{j \in \nn_3} \ell_{g^{(D_j)}}(\hat \gz_j|_{[0,N]})= N.$$
  This implies for every $N$ sufficiently large, there exists a geodesic segment $\gz_\fz^N$ of the standard cone $\mathcal C_\az^n$ contained in the ball $\overline{B(o,1)}$ of length $N$ starting from $o$. This is impossible, either, since all the geodesics passing through $o$ are the generatrices and any generatrix $G$ satisfies  $\ell_\std(G\cap \overline{B(o,1)}) =1$.

  As a consequence, both Subcase 2-1 and Subcase 2-2 could not happen, which implies Case 2 can not happen.  Thus Case 1 holds,  where we already showed $\sup_{j\in\nn}T_j < \infty$.
 The proof is complete.
\end{proof}

\begin{rem} \label{gen} \rm
In Section 5.1, we also need a slightly general version of Lemma \ref{geoflow}. Note that if all the exponential maps of $(M,g)$
  are proper, then for any sequence $\{\lz_j\}_{j \to\fz}$ with $\lim_{j \to \fz} \lz_j = \lz \in [1,\fz]$, the exponential maps of each $(M,g^{(\lz_j)})$ are also proper. If we only change the assumption that each $\gz_j$ is a geodesic ray of $(M,g^{(\lz_j)})$ in Lemma \ref{geoflow}, by a parallel proof, we also have the uniform length bound
  $$  \sup_{j}\ell_{g^{(\lz_j)}}(\gz_j|_{[0,T_j]}) < \sup_{j}\ell_{g^{(\lz)}}(\gz_j|_{[0,T_j]}) <\infty $$
  where in the case that $\lz =\fz$, $g^{(\lz)} = g_{\ccc_\az^n}$.
\end{rem}

Finally, we have further discussion for the properness assumption of
 the exponential map.

 \begin{rem}\rm
If a manifold $M$ has positive sectional curvature,   then  the exponential map  is always proper  at each point (see \cite[Theorem 4]{gm69}).
However the  positive sectional curvature  is not ncessary to
get properness of  the exponential map  at each point as indicated by the following
Example 3.6,
where  we construct   an asymptotically conical manifold $M$ with vanishing sectional curvature at some parts and the exponential map being proper at each point.
%


\begin{eg} \label{egnon} \rm
  Define a a $C^2$ Riemannian metric on $\rr^n$ by
  $$g(x):= \tilde g (x) + \chi_{\rr^n \setminus B(o,1)}(x) g_{\mathcal C_\az^n}(x), \quad x \in \rr^n$$
  where $\az= \pi/4$, $\chi$ is the indicator function and $\tilde g$ is a $C^2$ Riemannian metric on $B(o,1)$ such that
  \begin{enumerate}
      \item[(i)] $\tilde g$ has positive sectional curvature;
      \item[(ii)] the Riemannian metric $\tilde g|_{B(o,1-\ez)}$ on $B(o,1-\ez)$ has a $C^2$-extension to a Riemannian metric $\tilde g_\ez$ on $\rr^n$ such that $\tilde g_\ez$ has positive sectional curvature for all $\ez \ll 1$.
    \end{enumerate}
    First, $g$ can be easily realized by a rotational symmetric metric. Consider a $C^2$ function $f: [0,\fz) \to [0,\fz)$ such that $f(0)=f'(0)=0$, $f'(1)=1$, $f(s)=s+ f(1)-1$ for all $s \ge 1$ and $f''(s)>0$ for all $s \in [0,1)$. Then consider the graph $G(f)$ of $f$ as a curve in the plane spanned by $x_1$ and $x_{n+1}$ in $\rr^{n+1}$ and take $M$ as the $n$-dimensional manifold by revolution of the curve $G(f)$ with respect to the axis $x_{n+1}$.

    In addition, noticing that if $x \in \rr^n \setminus B(o,1)$, then any sectional curvature at $x$ containing the generatrix direction $\frac{x}{|x|}$ vanishes, since this is just the sectional curvature of the standard cone containing the generatrix direction.

    We check that the exponential map at each point is proper.
    By Lemma \ref{inj}, it suffices to show every geodesic ray $\Gz_{p,v}$ will escape to infinity as $t \to \fz$. Assume $\Gz_{p,v} \subset \rr^n \setminus B(o,1)$. Then $\Gz_{p,v}$ is a geodesic ray of the standard cone $\mathcal C_\az^n$, we know  $\Gz_{p,v}$ will escape to infinity. Now assume there exists a geodesic ray $\gz$ such that
    \begin{equation}\label{b01}
      \gz|_{[T,\fz)} \subset B(o,1).
    \end{equation}
    We first show that for any $\ez$ sufficiently small,
    \begin{equation}\label{b02}
      \gz|_{[T,\fz)} \nsubseteq B(o,1-\ez).
    \end{equation}
    By contradiction, assume for some $\ez_0>0$
    \begin{equation}\label{b03}
      \gz|_{[T,\fz)} \subset B(o,1-\ez_0).
    \end{equation}
    Then $\gz|_{[T,\fz)}$ is also a geodesic ray of the metric $\tilde g_{\ez_0}$, which has positive sectional curvature on $\rr^n$. In particular, $\ell_g(\gz \cap B(o,1-\ez_0)) =\fz$. By \cite[Lemma 6]{gm69}, we know this is impossible. Thus \eqref{b03} is false and \eqref{b02} holds.

    Combining \eqref{b01} and \eqref{b02}, we deduce that there exists a sequence of points $\{t_j\}_{j \in \nn}$ such that
    $$ \lim_{j \in \nn} (\gz(t_j), \dot \gz(t_j)) = (q,w) \in \pa B(o,1) \times \mathcal{S}_qM$$
    and
    $$ \angle (q,w) \le \frac{\pi}{2}.$$
    Consider the geodesic ray $\Gz_{q,w}$ starting from $p$ with velocity $w$. We know $\Gz_{q,w}$ is a geodesic ray of the standard cone $\mathcal C_\az^n$. Thus there exists $s_0$ such that
    $ \Gz_{q,w}(s_0) \in \pa B(o,2)$.
    This implies
    $$  d_g(\Gz_{q,w}(s_0), \Gz_{q,w}(0)) = d_\std(\Gz_{q,w}(t_0),q) = d_\std(\pa B(o,2),\pa B(o,1)) = 1.   $$
    On the other hand, letting $(q_j,w_j):=(\gz(t_j), \dot \gz(t_j))$ for $j \in \nn$, consider the geodesic rays $\Gz_{q_j,w_j}$. By the continuity dependence of the initial data of the geodesic flow of $(M,g)$, we know there exists $s_1> s_0$ such that for all $j> J(s_1)$ and all $s \in [0,s_1]$,
    $$  d_g(\Gz_{q,w}(s), \Gz_{q_j,w_j}(s)) < \frac12.   $$
    As a result $\Gz_{q,w}(s_0) \notin B(o,1)$ for all $j > J(s_1)$. However, $\Gz_{q,w} \subset \gz \subset B(o,1)$. This implies that \eqref{b01} can not hold. Thus all geodesic ray of $M$ will ultimately escape $B(o,1)$ and then becomes a geodesic ray of standard cone $\mathcal C_\az^n$ which will escape to infinity. By Lemma \ref{inj}, we know the exponential map of $M$ at every point is proper.
\end{eg}

 \end{rem}

 \begin{rem} \rm
In general, on  a manifold  $M$ diffeomorphic to $\rr^n$ with non-negative sectional curvature, it is not clear whether the exponential map   is proper at each point.
The behavior of geodesic rays on manifold $M$ with non-negative sectional curvature may be complicated. See \cite[Theorem 5.1]{cg71}.
 \end{rem}

\section{ Existence of min-max geodesic segments}
\label{s4}

Let $(M, g, o) = (\rr^n,g,o)$   be  an asymptotically conical manifold with asymptotical angle  $\az \in (0,\pi/2)$
and  asymptotic decay rate $\mu>0$ as in Section 3.
%
Given  any pair of antipodal points $p,q \in M$,
we plan to find some geodesic segments joining them  and close to $o$ uniformly in $p,q$
by using min-max method.

 Firstly we need to construct a useful homotopy class  based on the geometry of the standard cone $\mathcal C_\az^n$ given in Section 2.
Then we give min-max value and bound it from above and below.
Moreover, we introduce the truncated energy and consider its gradient flow.
By using such flow as the deformation map, we
  build up a mountain-pass theorem for geodesic segment.

Without loss  of generality,  we may assume that $p=(\rho,0,\cdots,0)$ and $ q=(\rho,\pi,0,\cdots,0)$ in the polar coordinate. Thus $p,q \in \Span \{e_1\}$ and
 $G=G_p\cup G_q=\mathbb Re_1$ where we recall $G_p$ and $G_q$ are the generatrices.
Recall that $\cl _{p,q}$ is the collection of  all length-minimising geodesic segments
in $\mathcal C^n_\alpha$ joining
$p,q$ paramatrized with constant speed
$$ |\dot\gz(t)| \equiv \ell_{\mathcal C^n_\alpha}(\gz)=d_{\mathcal C^n_\alpha}(p,q)=2\rho \sin(\frac{\pi}{2}\sin\az), \quad \forall t \in [0,1], \ \gz \in \cl _{p,q}.$$
and  $\cl^T _{p,q}$ is the  union of  images $ \gz([0,1])$ of all $\gz\in
 \cl _{p,q}$. Moreover,  set
 $$  \ccc_{pq}:=\{\gz \in W^{1,2}([0,1],M) : \gz(0)=p, \ \gz(1) =q, \ t\in[0,1]\}. $$

\subsection{ Construction of  homotopy class}


In dimension $n=2$,  recall from Remark \ref{geomin} (ii) that there are  exactly two elements $\Gz^\pm_0$ in $\cl_{p,q}$. 
 We define the class of homotopy
 $$ \cm_{pq}:=\{H \in C^0([0,1],\ccc_{pq}) : H(0)=\Gz^\pm_0, \ H(1) =\Gz^\mp_0\}. $$

  In dimension $n\ge3$, recall that $\mathbb S^{n-2}_G$ is  the unit sphere of $(\Span\{e_1\})^\perp=\Span \{e_2, \cdots, e_{n}\}$, and  is written as $\mathbb S^{n-2}$ for simplicity. By the discussion under \eqref{nantigeo}, for each $\xi \in \mathbb S^{n-2}$, there is a unique length-minimising geodesic segment   $L_\xi$   in $\mathcal C^n_\alpha$ joining $p$ and $q$ such that
   $$L_\xi([0,1])
    \subset 
    \Span\{e_1,\xi
   \}  \mbox{ intersecting the generatrix $\rr_+ \xi$}.$$
  Consider the reflection $I$ of $\rr^{n}$ with respect to $(\Span\{e_{n}\})^\perp$, i.e.
  \begin{equation}\label{ix}
    I(x)=I(x_1,x_2,\cdots,x_{n-1},x_n):= (x_1,x_2,\cdots,x_{n-1},-x_{n})\
    \forall x\in \mathbb R^n.
  \end{equation}
 Note that   $I(p)=p$ and $I(q)=q$.
  Set
 \begin{equation}\label{homotopy}
   \cm_{pq}:=\{H \in C^0(\mathbb S^{n-2} \times [0,1],\ccc_{pq}) : H(\xi,0)=L_\xi, \ H(\xi,1) = L_{I(\xi)},  \ \forall \xi \in \mathbb S^{n-2}\}.
 \end{equation}

The following basic
property is important for later use.

 \begin{lem}\label{lem4.1}
 (i) We  have  $
   I(L_\xi) = L_{I(\xi)}$ for all $\xi \in \mathbb S^{n-2} $ when $ n\ge3$
   and  $ I(\Gz^\pm_0) = \Gz^\mp_0 $ when $ n=2$.
 Hence $I(\cl_{p,q})=\cl_{p,q}$  and $I(\cl^T_{p,q})=\cl^T_{p,q}.$
Moreover, the mapping $I$ reverses the orientation of $\cl^T_{p,q}$, i.e., $\cl^T_{p,q}$ and $I(\cl^T_{p,q})$ have different orientations.

 (ii) The class $\cm_{pq}$ is not empty.
   For any $H \in \cm$, there exists $(\xi,s,t) \in \mathbb S^{n-2} \times [0,1]^2$ such that $H(\xi,s)(t) =o$.
 \end{lem}

 \begin{proof}
 (i) Let $\mathbb S^{n-3} \subset \mathbb S^{n-2}$ be the $(n-3)$-dimensional unit sphere of span$\{e_2, \cdots,e_{n-1}\}$.
 Then
 $$\mbox{$I(\xi) = \xi$ for all $\xi \in \mathbb S^{n-3}$ and $I(\xi) \ne  \pm \xi$ for all $\xi \in \mathbb S^{n-2} \setminus \mathbb S^{n-3}$.}$$
 Moreover, for any $\xi \in \mathbb S^{n-2}$, noticing that
 $$ I(\Span\{e_1\} \oplus \rr_+ \xi) = \Span\{e_1\} \oplus\rr_+ I(\xi), $$
 and recalling that $L_\xi \subset \Span\{e_1\} \oplus \rr_+ \xi$,
 we know
 $$I(L_\xi) \subset \Span\{e_1\} \oplus\rr_+ I(\xi).$$
 Since $L_{I(\xi)}$ is the unique length-minimising geodesic segment joining $p$ and $q$ such that
  $L_{I(\xi)}$ intersecting the generatrix $\rr_+ I(\bz)$,
 we deduce that
 \begin{equation}\label{refl}
   I(L_\xi) = L_{I(\xi)}, \quad \forall \xi \in \mathbb S^{n-2}.
 \end{equation}

Recall from Lemma \ref{nmingeo}(ii) that $\cl^T_{p,q}= \cup_{\xi \in \mathbb S^{n-2}} L_\xi
([0,1])$ is topologically an $(n-1)$-sphere in $\mathcal C_\az^n$. By \eqref{refl}, we know
$$I(\cl^T_{p,q})=\cl^T_{p,q}.$$
%
The mapping $I$ reverses the orientation of $\cl^T_{p,q}$, i.e., $\cl^T_{p,q}$ and $I(\cl^T_{p,q})$ have different orientations.
 Indeed, since the determinant of the differential det$DI$ of $I$ is $-1$ at all $x \in \rr^{n}$, $I$ is an orientation reversing map acting on any subsets of $\rr^{n}$ not contained in the fixed point set $\{x_n = 0\}$ of $I$. Since $\cl^T_{p,q}$ is not contained in the fixed point set of $I$, we get the desired conclusion.

 Observe that for any $H \in \cm_{pq}$ and $\xi\in \mathbb S^{n-2}$, $H(\xi,\cdot)$ is a homotopy from $L_\xi$ to $L_{I(\xi)}$.
 Since $M$ is simply-connected, we know the homotopy $H(\xi,\cdot)$ exists and can be chosen continuously depending on $\xi \in \mathbb S^{n-2}$. Thus $\cm_{pq}$ is not empty.


 (ii)  We argue by contradiction. Assume there an $H \in \cm$ such that
   $o \notin H(\bz,s)$ for all $(\bz,s) \in \mathbb S^{n-2} \times [0,1]$.
   Then $H(\mathbb S^{n-2},\cdot)$ is a homotopy from $\cl_{p,q}^T$ to $I(\cl_{p,q}^T)$ in $M \setminus \{o\}$.

   Note that the $(n-1)$-th homotopy group $\pi_{n-1}(M \setminus \{o\}) = \zz$ and $\cl_{p,q}^T$ belongs to a nontrivial element $e$ in $\pi_{n-1}(M \setminus \{o\}) = \zz$ (since by Lemma \ref{nmingeo2}, the origin $o$ is contained in the bounded component of $\rr^n\setminus \cl_{p,q}^T$).

   Thus the reversed orientation $I(\cl_{p,q}^T)$ belongs to the element $-e$ in $\pi_{n-1}(M \setminus \{o\}) = \zz$. This shows that $\cl_{p,q}^T$ and $I(\cl_{p,q}^T)$ cannot be homotopic, which leads to a contradiction.
 \end{proof}

 In the following of this section,  we  only focus on dimensions $n \ge 3$  while
  the dimension $n =2$ can  be treated in  a similar but easier way.


%
%

When $n\ge3$, we construct an example to show that
if  $H \notin \cm_{pq}$,  it may happen that $o$ is not contained in
the image of $H(\xi,s)([0,1])$ for any $(\xi,s)\in \mathbb S^{n-1}\times[0,1]$.
For the sake of simplicity, we only consider the case $n=3$ and for $n>3$, one could adapt the example correspondingly.
\begin{eg}\label{o}
   \rm When $n=3$,  the $(n-2)$-dimensional unit sphere $\mathbb S^{n-2}$ becomes
 the unit circle $\mathbb S^1$ of the $2$-dimsensinoal subspace $\Span\{e_2,e_3\} \subset \rr^3$.
   Consider the parametrization $\xi(\tz)=(0,\cos\tz,\sin\tz)$ with $\tz\in [0,2\pi)$ of $\mathbb S^1$.
  Note that
$$\xi(\tz+\pi)=  (0,\cos (\tz+\pi),  \sin (\tz+\pi))=(0,-\cos \tz, -\sin \tz)  $$
and hence
$L_{\xi(\tz+\pi)}\ne L_{I(\xi(\tz))}$ for all $\tz \in \rr$, where we recall from \eqref{ix} that
   $  I(\xi (\tz)) = I((0,\cos\tz,\sin\tz))=  (0,\cos \tz, -\sin \tz).  $
   Define $H$ by
   $$ H(\xi,s)= L_{\xi(\cdot+\pi s)} \in \cl_{p,q}, \quad (\xi,s) \in [0,2\pi) \times [0,1]. $$
Obviously $H \in C^0([0,2\pi) \times [0,1], \ccc)$, and for all $\xi \in \mathbb S^1$,
   $$ H(\xi,0)  = L_{\xi} \mbox{ and } H(\xi,1) = L_{\xi(\cdot+\pi)} \ne L_{I(\xi)},$$
which implies $H \notin \cm$.
   Moreover,  for each $\xi \in \mathbb S^1$ and each $s \in [0,1]$, $H(\xi,s)=L_{\xi(\cdot+\pi s)} $ is a length-minimising geodesic segment in $\mathcal C^n_\alpha$ which will never pass through $o$.
\end{eg}

\subsection{ Min-max value and lower/upper bound}

Define the min-max value
$$  \Lz_{pq}:= \inf_{H \in \cm_{pq}} \max_{(\xi,s) \in \mathbb S^{n-2} \times [0,1]} \ce(H(\xi,s))$$
where
$$  \ce(\gz):= \int_{0}^1 |\dot \gz(t)|_g^2 \, dt$$
is the standard energy functional of a curve $\gz$.
 The Euler-Lagrange equation of $ \mathcal E$ is geodesic equation,
whose solutions in $\ccc_{pq}$ consists of   geodesics with respect to $g$ parametrized by a constant multiple of the arc-length, are called as
the critical points  Crit$(\ce)$ for $\ce$.

Let $\ez >0$ be sufficiently small such that
\begin{equation}\label{std1}
 \sin^2(\frac{\pi}2 \sin\az ) < \frac{[1-5C_\star\ez]^2}{[1+5C_\star\ez]^2},
\end{equation}
where $C_\star$ denotes the  constant from \eqref{compare2}.
By Lemma \ref{bfm}, there exists $0<\tilde \kz<\ez$ such that
\begin{equation}\label{std2}
  (1-\ez)|x-o| \le d_g(o,x) \le (1+\ez)|x-o|, \quad \forall x \in M_{\le \tilde \kz}
\end{equation}
where we recall the region $M_{\le \tilde \kz}$ from \eqref{mez}.
Moreover,  
thanks to Remark \eqref{geomin} (iii),  we may choose  $ \kz<\tilde \kz$ so small such that for any pair of antipodal points $p,q \in M_{\le \kz}$,
 $$  \mathcal L^T_{p,q} \subset M_{\le \tilde\kz},$$
 and hence, by \eqref{std2},
 \begin{equation}\label{std2-a}
  (1-\ez)|\gz(t)-o| \le d_g(o,\gz(t)) \le (1+\ez)|\gz(t)-o|, \quad \forall \gz \in \mathcal L_{pq}\ t\in[0,1]
\end{equation}
 By Remark \ref{compare2}, this implies that for all $\gz \in \cl_{p,q}$, we have
\begin{equation}\label{std3}
 \ell_g(\gz) \le (1+C_\star\ez) \ell_{\mathcal C_\az^n}(\gz) = (1+C_\star\ez) 2\rho \sin(\frac{\pi}{2}\sin\az).
\end{equation}

Below  we always let $\ez$ and $\kz$ be as above and assume
the pair of antipodal points $p,q \in M_{\le \kz}$.
\begin{lem} \label{lbddlz}  One has
   $$(1-\ez)^2 4\rho^2\le \Lambda_{pq} \le (1+C_\star\ez)^2 4\rho^2.$$
\end{lem}

\begin{proof}
We  first  show $(1-\ez)^2 4\rho^2\le \Lambda_{pq}$.
  Thanks to Lemma \ref{lem4.1} (ii),
  for any homotopy $H \in \cm$, there exists $(\xi _o,s_o,t_o)\in \mathbb S^{n-2} \times [0,1]^2$ such that $H(\xi _o,s_o,t_o)=o$.

  Let $\gz_{po}:[0,1] \to M$ be the length-minimising geodesic segments joining $p$ and $o$ and $\gz_{oq}:[0,1] \to M$ be the length-minimising geodesic segments joining $o$ and $q$ with respect to $g$.
  Using \eqref{std2}, we estimate
  \begin{align*}
     \ell_g(\gz_{po}) & = d_g(p,o) \ge (1-\ez)|p-o|
      \ge \rho(1- \ez)
  \end{align*}
  and similarly,
  \begin{align*}
     \ell_g(\gz_{oq}) &  \ge \rho(1- \ez).
  \end{align*}
  Since $o = H(\xi _o,s_o,t_o)$, we deduce that
  $$\ell_g(H(\xi _o,s_o,\cdot)|_{[0,t_o]}) \ge \ell_g(\gz_{po})\ge \rho(1- \ez) \mbox{ and }  \ell_g(H(\xi _o,s_o,\cdot)|_{[t_o,1]}) \ge \ell_g(\gz_{oq})\ge \rho(1- \ez), $$
  which implies
  $$ \ell_g(H(\xi _o,s_o,\cdot)) \ge \ell_g(\gz_{poq}) \ge 2\rho(1- \ez) $$
  where $\gz_{poq}= \gz_{po} \ast \gz_{oq}$ is the concatenation of $\gz_{po}$ and $\gz_{oq}$. Using Cauchy-Schwartz inequality, we deduce that
  $$ \frac{\ce(H(\xi _o,s_o,\cdot))}{\rho^2} \ge \frac{\ell_g(H(\xi _o,s_o,\cdot))^2}{\rho^2} \ge \frac{\ell_g(\gz_{poq})^2}{\rho^2} \ge 4 \frac{[\rho(1- \ez)]^2}{\rho^2}> 4(1-\ez)^2. $$
  By the arbitrariness of $H$, we arrive at the desired result.

To  see $\Lambda_{pq} \le (1+C_\star\ez)^2 4\rho^2$, it suffices to construct a homotopy $H \in \cm$ such that
  \begin{equation}\label{upbd3}
    \max_{(\xi ,s) \in \mathbb S^{n-2} \times [0,1]} \ce(H(\xi ,s))\le (1+C_\star
    \ez)^2 4\rho^2.
  \end{equation}
  Recalling that for each $\xi   \in \mathbb S^{n-1}$, $L_\xi  \in \cl_{p,q}$ is a length-minimising geodesic of $\mathcal C_\az^n$ joining $p$ and $q$ which lies in the half plane $\Span\{e_1\}\oplus \rr_+\xi  \subset \mathcal C_\az^n$, we will choose $H(\xi ,s)\in \ccc_{pq}$ to be a curve such that
\begin{itemize}
  \item $H(\xi ,s) \subset  \Span\{e_1\}\oplus \rr_+\xi  $ for all $s \in [0,\frac12]$;
  \item $H(\xi ,\frac12)$ is the unique piecewise geodesic segment consisting of two generatrices passing through $o$ joining $p$ and $q$.
  \item $H(\xi ,s) \subset  \Span\{e_1\}\oplus \rr_+I(\xi ) $ for all $s \in [\frac12,1]$.
\end{itemize}
We note that for all $s \in [0,1]$, $H(\xi ,s)$ is always a curve in either the half plane $\Span\{e_1\}\oplus \rr_+\xi  $  or the half plane $\Span\{e_1\}\oplus \rr_+I(\xi )$.
This enables us to
  define $H(\xi ,\cdot)$ similar to the one in \cite[Lemma 20]{cd19}. Then it is obvious to check $H$ is continuous and $H\in \cm$ for each $\xi  \in \mathbb S^{n-2}$. Also the estimate \eqref{upbd3} is obtained similarly as in \cite[Lemma 20]{cd19}. We omit the details.
\end{proof}

\subsection{  Truncated energy functional and its gradient flow}

Below  we always let $\ez$ and $\kz$ be as in Section 4.2 and assume
  that the pair of antipodal points $p,q \in M_{\le \kz}$.

Let $\eta \in C^\fz( [0,\fz),[0,1])$ be a cut-off function such that
$$ \eta(x)=0, \quad \forall x \in [0,(1+5C_\star\ez)^2 4\rho^2\sin^2(\frac{\pi}2 \sin\az )], $$
$$ \eta(x)=x, \quad \forall x \in [(1-5C_\star\ez)^2 4\rho^2, \fz), $$
and for all $x \in ((1+5C_\star\ez)^2 4\rho^2\sin^2(\frac{\pi}2 \sin\az ), (1-5C_\star\ez)^2 4\rho^2]$,
$$0< \eta'(x) \le \frac{3}{4\rho^2[(1-5C_\star\ez)^2-(1+5C_\star\ez)^2 \sin^2(\frac{\pi}2 \sin\az )]}.$$

Define the truncated energy functional 
$$ \ce_\ast(\gz):= \eta \circ \ce(\gz), \quad \forall \gz \in  W^{1,2}([0,1],M). $$

Here  recall that $\ccc_{pq}$ has a natural Hilbert manifold structure (see for example \cite[Section 2.1]{d17}).
 For each $\gz \in \ccc_{pq}$, define the tangent space $\mathcal{T}_\gz \ccc_{pq}$ of $\ccc_{pq}$ at $\gz$ as
 $$ \mathcal{T}_\gz\ccc_{pq}:=\{X \in W^{1,2}([0,1],\mathcal{T}_\gz M): X(t) \in \mathcal{T}_{\gz(t)}M \mbox{ for all $t \in [0,1]$ and } X(0)=X(1)=0\}.$$
 Thus $ \mathcal{T}_\gz\ccc_{pq}$ consists of $W^{1,2}$ sections of the tangent bundle $\mathcal{T}M$ supported on $\gz$ and vanishing at the endpoints.
 Equip with $\ccc_{pq}$ the Riemannian metric
 $$ g^{\ccc_{pq}}: \mathcal{T}_{\gz}\ccc_{pq} \times \mathcal{T}_{\gz}\ccc_{pq} \to \rr, \ g^{\ccc_{pq}}(X_1,X_2)=\int_0^1 [g(X_1(t),X_2(t))+ g(\dot X_1(t), \dot X_2(t))] \, dt $$
 where $\dot X(t) = \nabla_{\dot \gz(t)} X(t)$ for any $X \in \mathcal{T}_{\gz}\ccc_{pq}$. Recalling that $g^{\ccc_{pq}}$ is a $C^2$-Riemannian metric, we know
 $(\ccc_{pq}, g^{\ccc_{pq}})$ is a Hilbert manifold with a complete $C^2$-Riemannian metric.

 Since $({\ccc_{pq}},g^{\ccc_{pq}})$ is a $C^2$-Hilbert manifold, we know $\ce$ and $\ce_\ast$ are both $C^2$-function defined on ${\ccc_{pq}}$.
 Define  $\nabla \ce_\ast$ is defined in the way such that
 $$g^{\ccc_{pq}}(\nabla \ce_\ast(\gz), V ) := ( d\ce_\ast(\gz), V) = \eta'(\ce_\ast(\gz))\int_0^1 g( \dot\gz,  V ) \, dt \quad \forall V \in \mathcal{T}_{\gz} \ccc_{pq}. $$
Given any initial value  $\gz \in \ccc_{pq}$, we consider the gradient flow of $\ce_\ast$ :
\begin{equation}\label{gflow}
 \left \{ \begin{array}{l} \frac{d}{d\tau}\gz_\tau = -\nabla \ce_\ast(\gz_\tau),
  \\
  \gz_0=\gz.
\end{array}
\right.
\end{equation}
Since  $\ce_\ast$ is of class $C^2$, the flow $\gz_\tau$ is well-defined for all $\gz \in \ccc_{pq}$ and $\tau \ge 0$.
 For more details, we refer to \cite[Theorem 1.11.1]{d17}.

Note that
$$  \frac{d}{d\tau}\ce_\ast (\gz_\tau) = (d \ce_\ast(\gz_\tau), \frac{d}{d\tau} (\gz_\tau))= - \eta'(\ce_\ast(\gz))\|\nabla \ce_\ast(\gz_\tau)\|_{g^{\ccc_{pq}}}^2. $$
Thus for $\tau'' \ge \tau' \ge 0$, we have
\begin{equation}\label{flowtau}
  \ce_\ast(\gz_{\tau''}) - \ce_\ast(\gz_{\tau'}) = \int_{\tau'}^{\tau''} \frac{d}{d\tau}\ce_\ast (\gz_\tau) \, d\tau= - \eta'(\ce_\ast(\gz))\int_{\tau'}^{\tau''}  \|\nabla \ce_\ast(\gz_\tau)\|_{g^{\ccc_{pq}}}^2 \, d\tau .
\end{equation}
Moreover,
\begin{itemize}
  \item if $\ce (\gz) < (1+5C_\star\ez)^2 4\rho^2\sin^2(\frac{\pi}2 \sin\az )$, then in a small neighbourhood $U$ of $\gz$, $\ce_\ast|_{U} = \nabla\ce_\ast|_{U} \equiv 0$, which implies that
      $$ \gz_\tau \equiv \gz_0=\gz, \quad \forall \tau>0;$$
  \item if $\ce (\gz)> (1-5C_\star\ez)^2 4\rho^2$ and $\gz$ is not a critical point of $\ce$, by \eqref{flowtau}, whenever $\ce (\gz_\tau)> (1-5C_\star\ez)^2 4\rho^2$, we have
      $\eta'(\ce_\ast(\gz_\tau)) =1$ in \eqref{flowtau} and consequently,
$$\mbox{the maps}\ \tau \mapsto \ce_\ast(\gz_\tau) = \ce(\gz_\tau) \ \mbox{is strictly decreasing at $\tau$.}$$
\end{itemize}

Below we write
$$\ccc_{pq}^{\ge (1-5C_\star\ez)^2 4\rho^2}:= \{\gz\in{\ccc_{pq}}: \ce(\gz) \ge (1-5C_\star\ez)^2 4\rho^2\}.$$
Consider   the restriction  $\ce_\ast|_{\ccc_{pq}^{\ge (1-5C_\star\ez)^2 4\rho^2}}$
of $\ce_\ast$ in the  set $ \ccc_{pq}^{\ge (1-5C_\star\ez)^2 4\rho^2}$.
We have the following property.
\begin{cor} \label{psast}
The restriction
  $\ce_\ast|_{\ccc_{pq}^{\ge (1-5C_\star\ez)^2 4\rho^2}}$ satisfies
  the Palais-Smale
condition.
\end{cor}
Here   recall the definition of   the Palais-Smale
condition; see for example, \cite{d17}.
\begin{defn}
\rm
Suppose that $\ch$ is a Hilbert manifold with a complete $C^2$-Riemannian metric. Then a $C^2$ function $f : \ch \to [0, \fz)$ is said to satisfy
the {\it Palais-Smale condition} if whenever $\{\xi_i\}$ is a sequence in $M$ such that
\begin{enumerate}
  \item[(i)] $\{f(\xi_i)\}$ is bounded,
  \item[(ii)] $\{\nabla f(\xi_i)\}$ is not bounded away from zero,
\end{enumerate}
then $\{\xi_i\}$ possesses a subsequence which converges to a critical point of $f$.
\end{defn}

Corollary \ref{psast} follows from
  the following theorem; see for example \cite[Theorem 2.25]{d17}.
\begin{thm} \label{psthm}
  $\ce$ satisfies Palais-Smale condition.
\end{thm}

\subsection{A mountain pass theorem for geodesic segments}
\label{s42}

Below  we always let $\ez$ and $\kz$ be as in Section 4.2 and assume
the pari of  antipodal points $p,q \in M_{\le \kz}$.
We have the following
 mountain pass theorem for geodesic segments.
For any set $A$, denote by $\Seq (A)$ the set of sequences $\{ a_i\}_{i\in\nn}\subset A$.
A sequence $\{H_j\}_{j \in \nn} \subset \Seq(\cm_{pq})$ is called to be
minimizing if
$$ \max_{(\xi,s) \in \mathbb S^{n-2} \times [0,1]} \ce(H_j(\xi,s))  \to \Lambda_{pq} \quad \text{as } j \to \fz.$$
Given a mminimzing sequence $ \{H_j\}_{j \in \nn}$ we further call a sequence $\{H_j(\xi_j,s_j)\}_{j \in \nn} $ as min-max if
$$ \ce(H_j(\xi_j,s_j))  \to  \Lambda_{pq} \quad \text{as } j \to \fz. $$

\begin{prop} \label{p14}
 There exists a geodesic segments $\Gz \in {\ccc_{pq}}$ such that
 $$  \ce(\Gz) = \Lambda_{pq}
 \ \mbox{and}\ \nabla \ce(\Gz)=0.$$
\end{prop}

\begin{proof}  We prove by  contradiction. Assume the set
  \begin{equation}\label{counter0}
    \{\gz \in {\ccc_{pq}}: \ce(\gz) = \Lambda_{pq}, \ \nabla\ce(\gz)=0   \} = \emptyset.
  \end{equation}
  We will conculde some contradiction by the following 3 steps.

{\it Step 1.}  Fix any min-max sequence
  \begin{equation*}
   \left \{\gz_j= H_j(\xi_j,s_j)\right\}_{j \in \nn}\subset \Seq({\ccc_{pq}}),
  \end{equation*}
  consequently,
  \begin{equation*}
    \gz_j \in \ccc_{pq} \mbox{ with } \lim_{j \to \fz} \ce_\ast(\gz_j) =\lim_{j \to \fz} \ce(\gz_j)= \Lambda_{pq}.
  \end{equation*}
  We moreover require that
   \begin{equation}\label{seq1}
    \ce(H_j(\xi_j,s_j))= \max_{(\xi,s)\in S^{n-2}\times [0,1]}\ce(H_j(\xi,s)), \quad \forall j \in \nn.
  \end{equation}
Since $\ce$ satisfies the Palais-Smale
condition, we know
\begin{equation}\label{minmax1}
  \liminf_{j \to \fz}\|\nabla\ce(\gz_j)\|_{\ccc_{pq}} = \dz >0.
\end{equation}
  Otherwise, the Palais-Smale
condition would imply that, up to subsequence, there exists a $W^{1,2}$-limit $\gz_\ast$ of $\{\gz_j\}$ which is a critical point of $\ce$. In particular,
$$ \ce(\gz_\ast) = \lim_{j \to \fz} \ce(\gz_j) = \Lz_{pq}  \mbox{ and } \nabla\ce(\gz_\ast)=0,$$
which contradicts to the counter assumption \eqref{counter0}.

  Since $\Lambda_{pq} \ge (1-\ez)^2 4\rho^2 $, we can assume $\ce(\gz_j) \ge (1-2\ez)^2 4\rho^2$ for all $j \in \nn$. As a result,
  $$ \{\gz_j\}_{j \in \nn} \subset \ccc_{pq}^{\ge (1-2\ez)^2 4\rho^2} \subset \ccc_{pq} ^{\ge (1-5C_\star\ez)^2 4\rho^2}$$
  By the definition of $\ce_\ast$, we know
  \begin{equation}\label{counter}
    \liminf_{j \to \fz}\|\nabla\ce_\ast(\gz_j)\|_X =\liminf_{j \to \fz}\|\nabla\ce(\gz_j)\|_X= \dz >0.
  \end{equation}

 {\it Step 2.}
%
For each $j \in \nn$ and $(\xi ,s) \in \mathbb S^{n-2} \times [0,1]$, consider the gradient flow $\Phi_{j,\xi ,s}$ of $\ce_\ast$ starting from $H_j(\xi ,s)$, i.e.
  \begin{equation}\label{gflow2}
 \left \{ \begin{array}{l}\displaystyle \frac{d}{d\tau}\Phi_{j,\xi ,s}(\tau) = -\nabla \ce_\ast(\Phi_{j,\xi ,s}(\tau)),
  \\
  \displaystyle
  \Phi_{j,\xi ,s}(0)=H_j(\xi ,s).
\end{array}
\right.
\end{equation}
  Fix $\wz \tau>0$ and define
  $$ \wz H_j (\xi ,s) : = \Phi_{j,\xi ,s}(\wz \tau), \quad \forall j \in \nn \mbox{ and }(\xi ,s) \in \mathbb S^{n-2} \times [0,1].$$
  Choose
  $(\wz \xi _j, \wz s_j) \in \mathbb S^{n-2}\times [0,1]$  such that
 $$\ce(\wz H_j(\wz\xi_j ,\wz s_j)) = \max_{(\xi ,s)\in \mathbb S^{n-2}\times [0,1]}\ce(\wz H_j(\xi ,s)).$$
We claim that
  \begin{equation}\label{claim}
  \mbox{$\{\wz H_j\}_{j \in \nn}\subset \Seq(\cm_{pq})$ is also a minimising sequence,}
 \end{equation}
 and moreover,
 \begin{equation}\label{claim2}
 \mbox{both $\{ \wz H_j(\wz \xi _j, \wz s_j)\}$ and $\{H_j(\wz \xi _j, \wz s_j)\}$ are min-max sequences}.
  \end{equation}

 To prove this claim, thanks to $H_j \in \cm_{pq}$ for all $j \in \nn$, one has
  $$H_j(\mathbb S^{n-2},0)\cup H_j(\mathbb S^{n-2},1) \subset \cl_{p,q}.$$
   Recalling from \eqref{std3} that   $$\ell_g(\gz)  \le (1+C_\star \ez) 2\rho \sin(\frac{\pi}{2}\sin\az)\quad\mbox{$\forall \gz \in \cl_{p,q}$},$$
   we obtain
  \begin{equation}\label{elpq}
    \ce(\gz) \le (1+C_\star\ez)^2 4\rho^2 \sin^2(\frac{\pi}{2}\sin\az)\quad\forall \gz \in \cl_{p,q}.
  \end{equation}
 This implies that $\ce_\ast$ and $\nabla \ce_\ast$ vanishes at any small neighbourhood of each element in $H_j(\mathbb S^{n-2},0)\cup H_j(\mathbb S^{n-2},1)$.
   Thus for all $j \in \nn$ and $\xi  \in \mathbb S^{n-2}$,
  $$\wz H_j (\xi ,0)  = \Phi_{j,\xi ,0}(\wz \tau)=  H_j (\xi ,0) \mbox{ and } \wz H_j (\xi ,1)  = \Phi_{j,\xi ,1}(\wz \tau)=  H_j (\xi ,1).$$
By the continuity of $\Phi_{j,\xi ,s}$ with respect to $(\xi ,s)$, we deduce that $\{\wz H_j\}_{j \in \nn}\subset \Seq(\cm_{pq})$. 
The definition of $ \Lambda_{pq}$ then gives
$$\sup_{(\xi,s)\in \mathbb S^{n-2}\times[0,1]}\ce(\widetilde H_j(\wz \xi , \wz s))\ge \Lambda_{pq}, \quad \forall j \in \nn,$$
hence
 \begin{equation}\label{ge}\liminf_{j\to\infty}\sup_{(\xi,s)\in \mathbb S^{n-2}\times[0,1]}\ce(\widetilde H_j(\wz \xi , \wz s)) \ge\Lambda_{pq}.
  \end{equation}

On the other hand,
 since the energy is decreasing along the flow $\tau\mapsto\Phi_{j,\wz \xi _j, \wz s_j}(\tau)$, it holds that
$$
 \ce(\wz H_j(\wz \xi _j, \wz s_j)) = \ce(\Phi_{j,\wz \xi _j, \wz s_j}(\wz \tau)) \le \ce(\Phi_{j,\wz \xi _j, \wz s_j}(0))\le \ce(H_j(\wz \xi _j, \wz s_j)).  $$
Thanks to \eqref{seq1} we have
$$\ce(H_j(\wz \xi _j, \wz s_j)) \le \ce(H_j( \xi _j,  s_j)) $$
and hence
 \begin{equation}\label{le}
   \limsup_{j \to \fz} \ce(\wz H_j(\wz \xi _j, \wz s_j)) \le   \limsup_{j \to \fz} \ce(H_j(\wz \xi _j, \wz s_j)) \le \Lambda_{pq}.
 \end{equation}
 Combining \eqref{le} and \eqref{ge}, we conclude
 \begin{equation*}
   \limsup_{j \to \fz} \ce(\wz H_j(\wz \xi _j, \wz s_j))=   \limsup_{j \to \fz} \ce(H_j(\wz \xi _j, \wz s_j)) = \Lambda_{pq}.
 \end{equation*}
 This gives the claims \eqref{claim} and \eqref{claim2}
as desired.

\medskip
  {\it Step 3.}  
Recall from Lemma \ref{lbddlz} that $\Lambda_{pq} \ge (1-\ez)^2 4\rho^2$.
 Without loss of generality, we could assume $$\ce(H_j(\wz \xi _j, \wz s_j)) \ge (1-2\ez)^2 4\rho^2\quad \mbox{for all $j \in \nn$}.$$
  Thus
 $$ \|\nabla\ce_\ast(H_j(\wz \xi _j, \wz s_j))\|_{\mathcal C_{pq}} = \|\nabla\ce(H_j(\wz \xi _j, \wz s_j))\|_{\mathcal C_{pq}}  \quad \forall j\in \nn. $$
 Similar to the argument to get \eqref{minmax1} and \eqref{counter}, we know
 $$  \liminf_{j \to \fz}\|\nabla\ce_\ast(H_j(\wz \xi _j, \wz s_j))\|_{\mathcal C_{pq}} = \dz'  $$
 for some $\dz'>0$.
 Choose $\wz\dz \in (0, \dz')$ and let
 \begin{equation}\label{tauj}
    \tau_j:= \inf \{ \tau\in(0,\wz\tau): \|\nabla\ce_\ast(\Phi_{j,\wz \xi _j, \wz s_j})(\tau)\|_{\mathcal C_{pq}} < \wz\dz   \}
 \end{equation}
 where we make the convention that $\tau_j =\wz\tau$ if the set in the RHS of \eqref{tauj} is empty.

 We consider the following two cases of $\liminf_{j \to \fz} \{\tau_j\}$ separately to conclude  contradiction.

 \medskip
 \emph{Case 1. There exists $\wz \dz \in (0,\dz')$ such that $\liminf_{j \to \fz} \{\tau_j\}= \tau_\fz >0$.} This means that there exists a subsequence $\wz \nn \subset \nn$ such that for each $j \in \wz \nn$,
 $$ \|\nabla\ce_\ast(\Phi_{j,\wz \xi _j, \wz s_j}(\tau))\|_X \ge \wz\dz, \quad \forall \tau \in [0,\tau_j]. $$
 As a result, for each $j \in \wz \nn$,
 \begin{align*}
   \ce(\wz \gz_j) & = \ce_\ast(\wz \gz_j)= \ce_\ast(\wz H_j(\wz \xi _j, \wz s_j)) =\ce_\ast(\Phi_{j,\wz \xi _j, \wz s_j}(\wz \tau))\\
    & \le \ce_\ast(\Phi_{j,\wz \xi _j, \wz s_j}(\tau_j)) \\
    & \le \ce_\ast(\Phi_{j,\wz \xi _j, \wz s_j}(0)) - \tau_j \wz\dz \\
    & =\ce_\ast(H_j(\wz \xi _j, \wz s_j)) - \tau_j \wz\dz =\ce(\Phi_{j,\wz \xi _j, \wz s_j}) - \tau_j \wz\dz .
 \end{align*}
 Hence
 $$  \Lambda_{pq} = \liminf_{j \in \wz \nn} \ce(\wz \gz_j) \le \liminf_{j \in \wz \nn} \ce(\Phi_{j,\wz \xi _j, \wz s_j}) - \tau_j \wz\dz \le  \Lambda_{pq} - \tau_\fz \wz\dz, $$
 which is impossible. Thus this case cannot happen.

 \medskip
 \emph{Case 2. For all $\wz \dz \in (0,\dz')$, $\liminf_{j \to \fz} \tau_j =0$.}
 In this case, we can find a subsequence $\wz \nn \subset \nn$ such that for $j \in \wz \nn$,
 \begin{equation}\label{ps}
   \|\nabla\ce_\ast(\Phi_{j,\wz \xi _j, \wz s_j}(\tau_j))\|_X < \frac1j.
 \end{equation}
 Recall that
 $$\wz H_j(\wz \xi _j, \wz s_j)= \Phi_{j,\wz \xi _j, \wz s_j}(\wz\tau), \quad  H_j(\wz \xi _j, \wz s_j)= \Phi_{j,\wz \xi _j, \wz s_j}(0), \quad \forall j \in \wz \nn$$
 and
 notice that
 $$\ce(\wz H_j(\wz \xi _j, \wz s_j)) = \ce_\ast(\wz H_j(\wz \xi _j, \wz s_j)), \quad \ce( H_j(\wz \xi _j, \wz s_j)) = \ce_\ast( H_j(\wz \xi _j, \wz s_j)), \quad \forall j \in \wz \nn. $$
 Since both $\{\wz H_j(\wz \xi _j, \wz s_j)\}_{j \in \nn}$ and $\{ H_j(\wz \xi _j, \wz s_j)\}_{j \in \nn}$ are min-max sequences, by the monotonicity of $\ce_\ast$ along the flow $\Phi$ and $0 \le  \tau_j \le \wz \tau$, we deduce that
 $$ \Lambda_{pq} = \lim_{\wz \nn \ni j \to \fz} \ce_\ast(\Phi_{j,\wz \xi _j, \wz s_j}(\wz\tau)) \le \lim_{\wz \nn \ni j \to \fz}\ce_\ast(\Phi_{j,\wz \xi _j, \wz s_j}(\tau_j))\le \lim_{\wz \nn \ni j \to \fz}\ce_\ast(\Phi_{j,\wz \xi _j, \wz s_j}(0))=\Lambda_{pq}. $$
 As a consequence, $\{\Phi_{j,\wz \xi _j, \wz s_j}(\tau_j)\}_{j \in \wz \nn}$ is also a min-max sequence. Using \eqref{ps} and the Palais-Smale condition, we know that
 $\{\Phi_{j,\wz \xi _j, \wz s_j}(\tau_j)\}_{j \in \wz \nn}$ has a limit $\wz \gz \in \ccc$ which is a critical point of $\ce$
 and moreover $\ce(\wz \gz)=\Lambda_{pq}$.
 This violates the counter assumption \eqref{counter0}.

The proof is complete.
 \end{proof}

 \begin{prop} \label{index}
 If $$\Gz \in \{\gz \in \ccc_{pq}: \ce(\gz) = \Lambda_{pq}, \ \nabla\ce(\gz)=0   \}$$
 is obtained by Proposition \ref{p14},
 then
  $\Gz$ has the Morse index $\le n-1$.

 \end{prop}

 \begin{proof}

    This is a direct application of \cite[Theorem 4 in page 53]{n91}. In the notation of \cite{n91}, the group $G$ is the trivial group in our case and
  one can verify that the homotopy class $\cm_{pq}$ in \eqref{homotopy} satisfies \cite[Definition (2) in page 52]{n91} by taking $X=\ccc_{pq}$, $B=\cl_{p,q}$, $\mathscr{C}=\mathbb S^{n-2} \times [0,1]$, $$D_0=[\mathbb S^{n-2}\times\{0\}] \cup [\mathbb S^{n-2}\times\{1\}],$$
  and $\sz:D_0 \to B$,
  $$\sz|_{\mathbb S^{n-2}\times\{0\}}= Id, \quad \sz|_{\mathbb S^{n-2}\times\{1\}}=I|_{\mathbb S^{n-2}}$$
   (where $I$ is as in \eqref{homotopy}) therein.
 Moreover, we can apply \cite[Theorem 4]{n91} by taking $c=\Lambda_{pq}$, $$K_{\Lambda_{pq}}=\{\gz \in \ccc_{pq}: \ce(\gz) = \Lambda_{pq}, \ \nabla\ce(\gz)=0   \},   \
 \phi=\ce_\ast|_{\ccc_{pq}^{\ge (1-5C_\star\ez)^2 4\rho^2}} \ \mbox{and} \ F=\ccc_{pq}^{\ge (1-5C_\star\ez)^2 4\rho^2}$$ therein. The output is a geodesic segment $\Gz \in \ccc_{pq}$ with Morse index $\le n-1$.
\end{proof}

\begin{rem} \label{flowast} \rm
  We remark that in the above proof, it is necessary to use the truncated energy functional $\ce_\ast$ instead of $\ce$ to define $\wz H_j$. This is because we need the homotopy $\wz H_j$ belong to the homotopy class $\cm_{pq}$ as shown in \eqref{claim} in the proof. One could check that by using the gradient flow of $\ce$ instead of the flow in \eqref{gflow2}, it
  is not possible to conclude the claim \eqref{claim} in the proof.
  This is one of the main difference from the paper \cite{cd19} for the $2$-dimensional case.
\end{rem}

\section{Existence of min-max geodesic lines}

%

In  this section, we assume that
 $(M,g,o)=(\mathbb R^n,g,o)$ be an asymptotically conical manifold with opening angle
 $$\az \in (0,\frac\pi 2) \setminus \{\arcsin \frac1{2k+1} \}_{k \in \nn}$$
 and asymptotic rate $\mu>0$.
We assume $\ez$ and $\kz$ satisfy \eqref{std1}, \eqref{std2} and \eqref{std3} throughout this section.
In the polar coordinate of $\rr^n$, given any $a=(1,\theta)\in S^{n-1}$, denote by $a^\star=(1,\theta^\star)$ where we recall $\theta^\star$ is the antipodal direction of $\tz$.
Write the line $G=G_a\cup G_{a^\star}$.

Given any $\rho>R_\kz$
write
$$\mbox{$p=\rho a=(\rho,\theta)$ and $q=\rho a^\star=(\rho,\theta^\star)$}.$$
Thus $p,q \in  M_{\le \kz}$.
We apply Proposition \ref{p14} to get at least one geodesic segment $\Gz_\rho \in \ccc_{pq}$ such that
\begin{equation}\label{length0}
  \ce(\Gz_\rho) \ge (1-\ez)^2 4\rho^2.
\end{equation}
In addition, Proposition \ref{index} implies that $\Gz_\rho$ has Morse index $\le n-1$.

\subsection{Uniform estimate for distances between $o$  and minmax geodesic segments}

\begin{lem}\label{unif}
We have
$$C_{\rm minmax}:=\sup_{\rho>R_\kz}\dist(\Gz_\rho,o)<\infty.$$
\end{lem}

\begin{proof}
  We argue by contradiction.
  Assume that
  $$\sup_{\rho>R_\kz}\dist(\Gz_\rho,o)=\infty.$$
  Then there exists $\{\rho_j\}_{j\in\nn}$ such that $\rho_j\to\infty$
 and
 $ \lim_{j\to\infty}\dist(\Gz_{\rho_j},o)=\infty.$
  Write
  \begin{equation}\label{hat}
    \hat \rho_j:=\dist(\Gz_{\rho_j},o)= \sup \{r> 0 : B(o,r) \cap \Gz_{\rho_j} = \emptyset \}.
  \end{equation}
    Note that $\hat\rho_j \le \rho_j$ and
  \begin{equation}\label{ass1}
    \liminf_{j \to \fz}\hat \rho_j = +\fz.
  \end{equation}
  To conclude some contradiction, we consider  the following two cases:
  $$\liminf_{j\to\infty} \frac{\hat \rho_j}{\rho_j}>0
  \quad\mbox{and}\quad \liminf_{j\to\infty} \frac{\hat \rho_j}{\rho_j}=0.$$

  \medskip
  \emph{Case 1. $\liminf_{j\to\infty}  {\hat \rho_j}/{\rho_j}>0$.} There exists $D \ge 1$  such that
  \begin{equation}\label{case1}
    \hat\rho_j \ge \frac1{D} \rho_j \quad  \forall j \in \nn.
  \end{equation}
  For each $j \in \nn$, consider the curve
  $$ \wz \Gz_{\rho_j}:= \mathcal{D}_{1/\rho_j}\circ \Gz_{\rho_j},$$
  where we recall the scaling map  from \eqref{slz}.
  Since
  \begin{equation*}
      \mathcal{D}_{1/\rho_j}(p_j)=(\frac{\rho_j}{\rho_j}, \tz)=a\mbox{ and }   \mathcal{D}_{1/\rho_j}(q_j)=(\frac{\rho_j}{\rho_j}, \tz^\star)=a^\star,
  \end{equation*}
 we know that $\wz \Gz_{\rho_j}$ is a geodesic segment in $M^{(1/\rho_j)}= (M, g^{(1/\rho_j)})$ joining $a,a^\star$.
  The assumption \eqref{case1} implies that
  $$  \widetilde \Gz_{\rho_j} \subset \overline{B(o,1)} \setminus B(o,1/D).$$
   Thanks to Lemma \ref{convg}, there exists a subsequence $\nn_1 \subset \nn$ such that
  $ \{\wz \Gz_{\rho_j}\}_{j \in\nn_1}$ converges in $C^2$-norm to some geodesic segment $\Gz$  in $\mathcal C_\az^n$  joining $a,a^\star$
  and
  $$ \Gz \subset \overline{B(o,1)} \setminus B(o,1/D).$$
  Since $\Gz$ does not passing through the origin $o$ of $\mathcal C_\az^n$, from \eqref{nantigeo} (and also \eqref{gp2} in Lemma \ref{antigeo}) it follows that
  $\Gz$ is $ \Gz_k^\pm$ as given therein, and hence
 \begin{equation*}
    \ell_{\mathcal C_\az^n}(\Gz)  \in\left \{ 2 \sin\frac{(2k-1)\pi\sin\az}{2}: k=1, \cdots, \lfloor\frac{1}{2\sin\az}-\frac12\rfloor\right  \} .
  \end{equation*}
By our choice of $\ez$ as in \eqref{std1},
  we know that
   \begin{equation} \label{k13}
    \ell_{\mathcal C_\az^n}(\Gz) \le \max_{k=1, \cdots, \lfloor\frac{1}{2\sin\az}-\frac12\rfloor  }\left \{2 \sin\frac{(2k-1)\pi\sin\az}{2} \right\} \le 2(1-2\ez).
  \end{equation}

  On the other hand, recalling \eqref{length0}, we know that
  $$  \ell_{g^{1/\rho_j}}(\Gz_{\rho_j}) = \frac1{\rho_j} \ell_{g}(\Gz_{\rho_j}) \ge 2(1-\ez), \quad  j \in \nn.   $$
  Since the metric $g^{(1/\rho_j)}$ converges to $g_{\mathcal C_\az^n}$ in $C^2$-norm, we deduce that
  $$  \ell_{\mathcal C_\az^n}(\Gz) = \lim_{\nn_1 \ni j \to \fz} \ell_{\mathcal C_\az^n}(\Gz_{\rho_j})= \lim_{\nn_1 \ni j \to \fz}\ell_{g^{(1/\rho_j)}}(\Gz_{\rho_j})\ge 2(1-\ez). $$
  This contradicts with \eqref{k13}. 

\medskip
\emph{Case 2. $ \limsup_{j\to\infty} {\hat \rho_j}/{\rho_j}=0$.}
There exists a subsequence $\nn_2 \subset \nn$ such that
\begin{equation}\label{case2}
   \lim_{\nn_2 \ni j \to \fz} \frac{\hat\rho_j}{\rho_j} =0.
\end{equation}
Thus for any $N\ge 2$, there exists $m_N$ such that
\begin{equation}\label{case22}
  \hat\rho_j < \frac1N \rho_j , \quad \forall j \in \nn_2 \mbox{ and } j\ge m_N.
\end{equation}
  For each $j \in \nn_2$ and $ j\ge m_N$, by the definition of $\hat\rho_j$ in \eqref{hat},
   there exists $\hat t_{\rho_j} \in (0,1)$ such that $$|\Gz_{\rho_j}(\hat t_{\rho_j})|= \hat \rho_j.$$
Define
  $$   t_{\rho_j,N}^-:= \max \{t \in [0,\hat t_{\rho_j}]:|\Gz_{\rho_j}(t)|= N\hat\rho_j \}, \quad  t_{\rho_j,N}^+:= \min \{t \in [\hat t_{\rho_j},1]:|\Gz_{\rho_j}(t)|=N\hat\rho_j\}. $$
Since
$$ \hat\rho_j <\frac1N \rho_j \mbox{ and
 } |\Gz_{\rho_j}(0)|=|p_j|= |q_j|=|\Gz_{\rho_j}(1)|=\rho_j,$$
one has
  $$0<t_{\rho_j,N}^-<\hat t_{\rho_j}<t_{\rho_j,N}^+<1.$$
  Moreover, one has
  \begin{equation}\label{nball}
    \Gz_{\rho_j}([t_{\rho_j,N}^-,t_{\rho_j,N}^+]) \subset \overline{B(o,N\hat\rho_j)}, \quad \forall j \in \nn_2 \mbox{ and } j\ge m_N.
  \end{equation}
 Observe that
  our assumption
  $\sin\az \notin \{1/(2k+1)\}_{k \in \nn}$
implies
$(2k+1)\pi\sin\alpha \ne \pi$ for any $k\in\nn$. In particular, letting $k=K_\az$ and $k=K_\az-1$ respectively, we have
 $$ 2\pi K_\alpha\sin\alpha - \pi \ne \pm\pi\sin\alpha.$$
 Hence recalling ${\bf K}_\alpha \in [0,\pi]$ from Lemma \ref{line2} (also Lemma \ref{line1}), we have
 \begin{equation}\label{lepi}
   {\bf K}_\alpha=\frac{|2\pi K_\alpha \sin \az - \pi|}{\sin \az}<\pi.
 \end{equation}

  We will conclude the contradiction by showing the following two claims and thus complete the proof of Lemma \ref{unif}.

\medskip
\noindent{\it Claim I.}  There is a large ${N_0}\in\nn$ and a subsequence $\nn^\ast \subset \nn_2$ so that
\begin{align}\label{angleest1}
  \angle(-\dot\Gz_{\rho_j}(t^-_{\rho_j,N_0}), \dot \Gz_{\rho_j} (t^+_{\rho_j,N_0}))
 \in  [\mathbf{K}_\az -\frac1{10}(\pi-\mathbf{K}_\az),\mathbf{K}_\az +\frac1{10}(\pi-\mathbf{K}_\az)], \quad j\in\nn^\ast.
\end{align}

\noindent {\it Claim II.}
There is a subsequence $\nn^{\ast\ast}\subset\mathbb N^\ast$ so that
\begin{align}\label{angleest2} \lim_{\nn^{\ast\ast}\ni j\to\infty}\angle(-\dot\Gz_{\rho_j}(t^-_{\rho_j,N_0}) , \dot \Gz_{\rho_j} (t^+_{\rho_j,N_0})) =\pi.
\end{align}

Below, we prove the two claims.

\medskip {\it Proof of Claim I.}
For any $N \ge 2$, write
  $ \Gz_{\rho_j,N}:[0,1]\to M$ as the reparametrization of
  $\Gz_{\rho_j}|_{[t_{\rho_j,N}^-,t_{\rho_j,N}^+]}$ with a constant multiple of arc-length,
  that is,
   $$ \Gz_{\rho_j,N}(t):= \Gz_{\rho_j}
    ((t_{\rho_j,N}^+ - t_{\rho_j,N}^-) t + t_{\rho_j,N}^-) \quad t \in [0,1].$$
   For each $N\ge 2$ and $j \in \nn_2$ with $j \ge m_N$,  define
  $$ \Gz_{\rho_j,N}^\ast:= \mathcal{D}_{1/\hat\rho_j} \circ \Gz_{\rho_j,N}$$
  Since $\mathcal{D}_{1/\hat\rho_j}$ is conformal, one has
  \begin{equation}\label{conformal}
    \angle(-\dot\Gz_{\rho_j}(t^-_{\rho_j,N_0}), \dot \Gz_{\rho_j} (t^+_{\rho_j,N_0}))= \angle(-\dot\Gz_{\rho_j,N}(0),\dot\Gz_{\rho_j,N}(1))=
    \angle(-\dot\Gz^\ast_{\rho_j,N}(0),\dot\Gz^\ast_{\rho_j,N}(1)).
  \end{equation}
  To get the claim \eqref{angleest1}, it then suffices to find a large ${N_0}\in\nn$ and subsequence $\nn^\ast \subset \nn_2$ so that
\begin{align}\label{angleest} \angle(-\dot\Gz_{\rho_j,N_0}(0),\dot\Gz_{\rho_j,N_0}(1)) \in [\mathbf{K}_\az -\frac1{10}(\pi-\mathbf{K}_\az),\mathbf{K}_\az +\frac1{10}(\pi-\mathbf{K}_\az)]\quad j\in\nn^\ast.
\end{align}

 We prove \eqref{angleest} as below. See Figure \ref{ngamma} for an illustration of the following proof.
   \vspace*{1pt}
\begin{figure}[h]
\centering
\includegraphics[width=11cm]{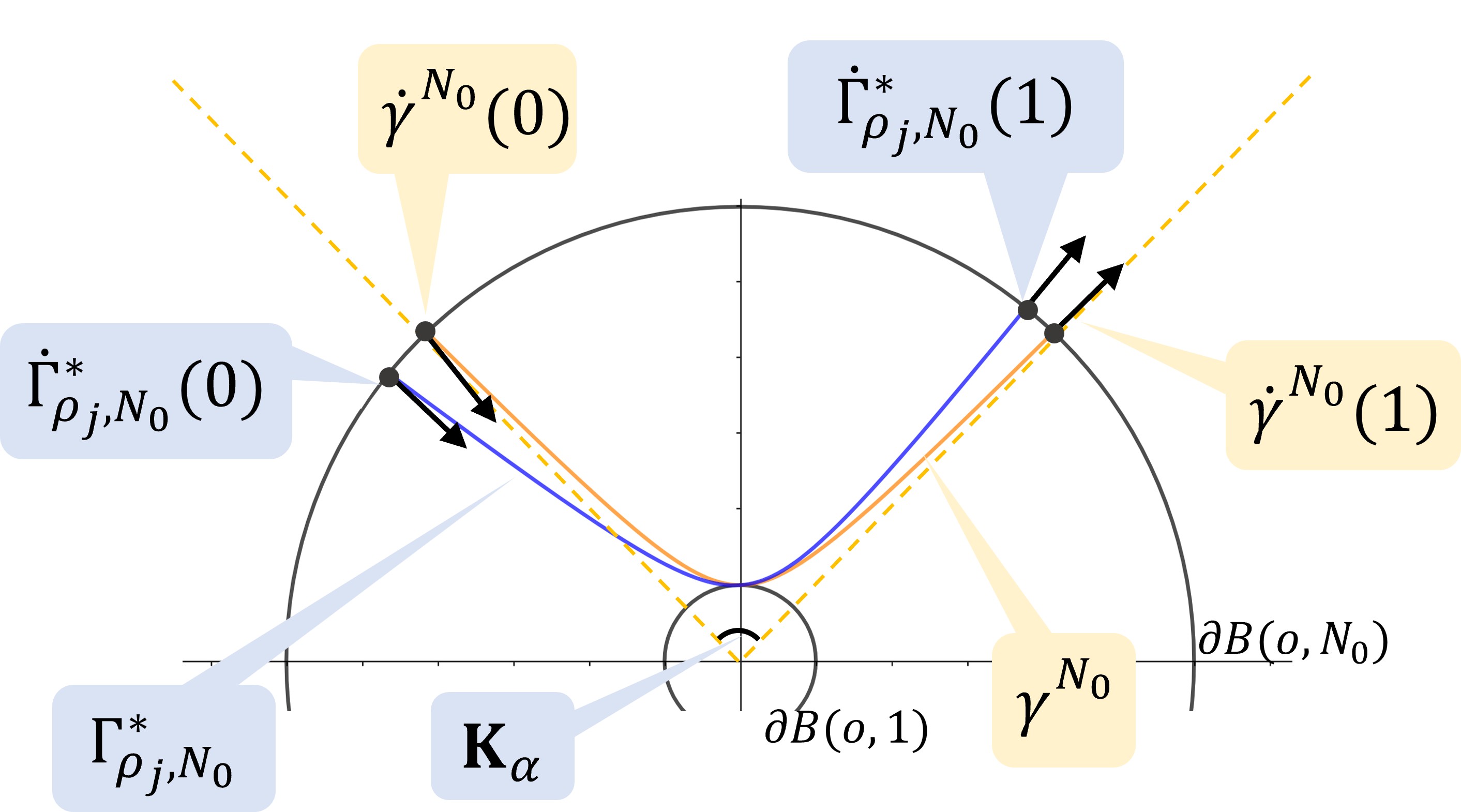}
\caption  {The estimate of angle $\angle(-\dot\Gz^\ast_{\rho_j}(0),\dot\Gz^\ast_{\rho_j}(1))$.}
\label{ngamma}
\end{figure}
Note  that $\Gz_{\rho_j,N}^\ast$ is a geodesic segment
 in $M^{(1/\hat\rho_j)}=(M, g^{(1/\hat\rho_j)})$
parameterised with a constant multiple of arc-length and by \eqref{nball},
\begin{equation}\label{cor-1}
   \Gz_{\rho_j,N}^\ast([0,1]) \subset \overline{B(o,N)} \setminus B(o,1)
\end{equation}
  with
\begin{equation}\label{cor-2}
    \Gz_{\rho_j,N}^\ast(\{0,1\})\subset \pa B(o,N) \ \mbox{and}\
   \Gz_{\rho_j,N}^\ast([0,1])\cap \partial B(0,1) \ne\emptyset.
\end{equation}
 There exists a subsequence $\nn_3^{(N)} \subset \{j\in \nn_2 : j\ge m_N\}$ such that
 $\{\Gz_{\rho_j,N}^\ast\}_{j \in \nn_3^{(N)}}$ converges in the $C^2$-norm to a geodesic segment $\gz^N:[0,1]\to\mathcal C_\az^n$ with a constant multiple of arc-length parmetrization and \eqref{cor-1} and \eqref{cor-2} imply that
 $$  \gz^N \subset \overline{B(o,N)} \setminus B(o,1) \mbox{ with } \gz^N(0),\gz^{N}(1) \in \pa B(o,N) \
 \mbox{and}\    \gz^N([0,1])\cap \partial B(0,1) \ne\emptyset.$$
We apply Corollary \ref{line3} to the family $\{\gz^{N}\}_{N \in \nn}$ to deduce that
 $$ \lim_{N \to \fz} \angle(-\dot\gz^{N}(0),\dot\gz^{N}(1)) = \mathbf{K}_\az< \pi$$
where for the last inequality we recall \eqref{lepi}.

 We fix $N_0$ large enough such that
 $$  \angle(-\dot\gz^{N_0}(0),\dot\gz^{N_0}(1)) \in [\mathbf{K}_\az -\frac1{20}(\pi-\mathbf{K}_\az),\mathbf{K}_\az +\frac1{20}(\pi-\mathbf{K}_\az)].$$
 Since $\{\Gz_{\rho_j,N_0}^\ast\}_{j \in \nn_3 ^{(N_0)}}$ converges in the $C^2$-norm to $\gz^{N_0}$, for all sufficiently large $j \in \nn_3^{(N_0)}$, denoted by $j\in \nn^\ast$, one has
 $$ \angle(-\dot\Gz^\ast_{\rho_j,N_0}(0),\dot\Gz^\ast_{\rho_j,N_0}(1)) \in [\mathbf{K}_\az -\frac1{10}(\pi-\mathbf{K}_\az),\mathbf{K}_\az +\frac1{10}(\pi-\mathbf{K}_\az)].$$

For $j\in\nn^\ast$, recalling the second equality in \eqref{conformal}, we arrive at
\eqref{angleest}, thus conclude Claim I as desired.

\medskip {\it Proof of Claim II.} For each $j \in \nn^\ast$,
   write   $ \Gz^+_{\rho_j,N_0}:[0,1]\to M$ as the reparametrization of
  $\Gz_{\rho_j}|_{[t_{\rho_j,N_0}^+,1]}$ with a constant multiple of arc-length,
  that is,
   $$ \Gz_{\rho_j,N_0}^+(t):= \Gz_{\rho_j}
    ((1-t_{\rho_j,N_0}^+ ) t + t_{\rho_j,N_0}^+), \quad t \in [0,1].$$
   Also,  write     $ \Gz^-_{\rho_j,N_0}:[0,1]\to M$ as the reparametrization of
  $\Gz_{\rho_j}|_{[0,t_{\rho_j,N_0}^-]}$ with a constant multiple of arc-length,
  that is,
   $$ \Gz_{\rho_j,N_0}^-(t):= \Gz_{\rho_j}
    (  t_{\rho_j,N_0}^- t ), \quad t \in [0,1].$$
   Observing that
   $$ \Gz_{\rho_j}(t^-_{\rho_j,N_0}) = \Gz^-_{\rho_j,N_0}(1) \mbox{ and }   \Gz_{\rho_j}(t^+_{\rho_j,N_0}) = \Gz^+_{\rho_j,N_0}(0),$$
   we have
    $$\angle(-\dot\Gz_{\rho_j}(t^-_{\rho_j,N_0}), \dot\Gz_{\rho_j} (t^+_{\rho_j,N_0}))=
     \angle (\dot \Gz^-_{\rho_j,N_0}(1), \dot \Gz^-_{\rho_j,N_0}(0)).$$
    To see \eqref{angleest2} in Clain II, it then suffices to find
     a subsequence $\nn^{\ast\ast}\subset\mathbb N^\ast$ so that
     \begin{align}\label{angleest3}
     \lim_{\nn^{\ast\ast}\ni j\to\infty}
    \angle (-\dot \Gz^-_{\rho_j,N_0}(1), \dot \Gz^+_{\rho_j,N_0}(0))=\pi.
    \end{align}
To see \eqref{angleest3},
 for $j \in \nn^\ast$, set
$$  \Gz_{j}^\pm:= \mathcal{D}_{\rho_j^{-1}}\circ \Gz_{\rho_j,N_0}^\pm.
$$
We first note that, for all $j \in \nn^\ast$,
$$  \Gz_{j}^-(0)= \mathcal{D}_{\rho_j^{-1}}\circ \Gz^-_{\rho_j,N_0}(0) = \mathcal{D}_{\rho_j^{-1}}\circ \Gz_{\rho_j}(0) = \mathcal{D}_{\rho_j^{-1}}(p_j) =a,$$
$$   \Gz_{j}^+(1)= \mathcal{D}_{\rho_j^{-1}}\circ \Gz^+_{\rho_j,N_0}(1) = \mathcal{D}_{\rho_j^{-1}}\circ \Gz_{\rho_j}(1) = \mathcal{D}_{\rho_j^{-1}}(q_j)=a^\star,$$
$$  \Gz_{j}^-(1)= \mathcal{D}_{\rho_j^{-1}}\circ \Gz^-_{\rho_j,N_0}(1)= \mathcal{D}_{\rho_j^{-1}}\circ \Gz_{\rho_j}(t_{\rho_j,N_0}^-) \in \pa B(o,\frac{\hat \rho_j N_0}{\rho_j}) $$
and
$$ \Gz_{j}^+(0)= \mathcal{D}_{\rho_j^{-1}}\circ \Gz^+_{\rho_j,N_0}(0)= \mathcal{D}_{\rho_j^{-1}}\circ \Gz_{\rho_j}(t_{\rho_j,N_0}^+)\in \pa B(o,\frac{\hat\rho_j N_0}{\rho_j}).$$
Recalling \eqref{case2} and \eqref{case22}, we know
$$ \lim_{\nn_2 \supset \nn^\ast \ni j \to \fz} \frac{\hat\rho_j N_0}{\rho_j} = 0.$$
As a consequence,
$$ \lim_{j \to \fz} \Gz_{j}^-(1) = \lim_{j \to \fz} \Gz_{j}^+(0) = o$$
and there exists a subsequence $\nn^{\ast\ast} \subset \nn^\ast$ such that
$\{\Gz_{j}^-\}_{j \in \nn^{\ast\ast}}$ (resp. $\{\Gz_{j}^+\}_{j \in \nn^{\ast\ast}}$) converges in $C^2$ to the geodesic segment in $\mathcal C_\az^n$ joining $a$ and $o$ (resp. $o$ and $a^\star$).
  See Figure \ref{ngamma2} for an illustration for the following proof.
  \vspace*{1pt}
\begin{figure}[h]
\centering
\includegraphics[width=11cm]{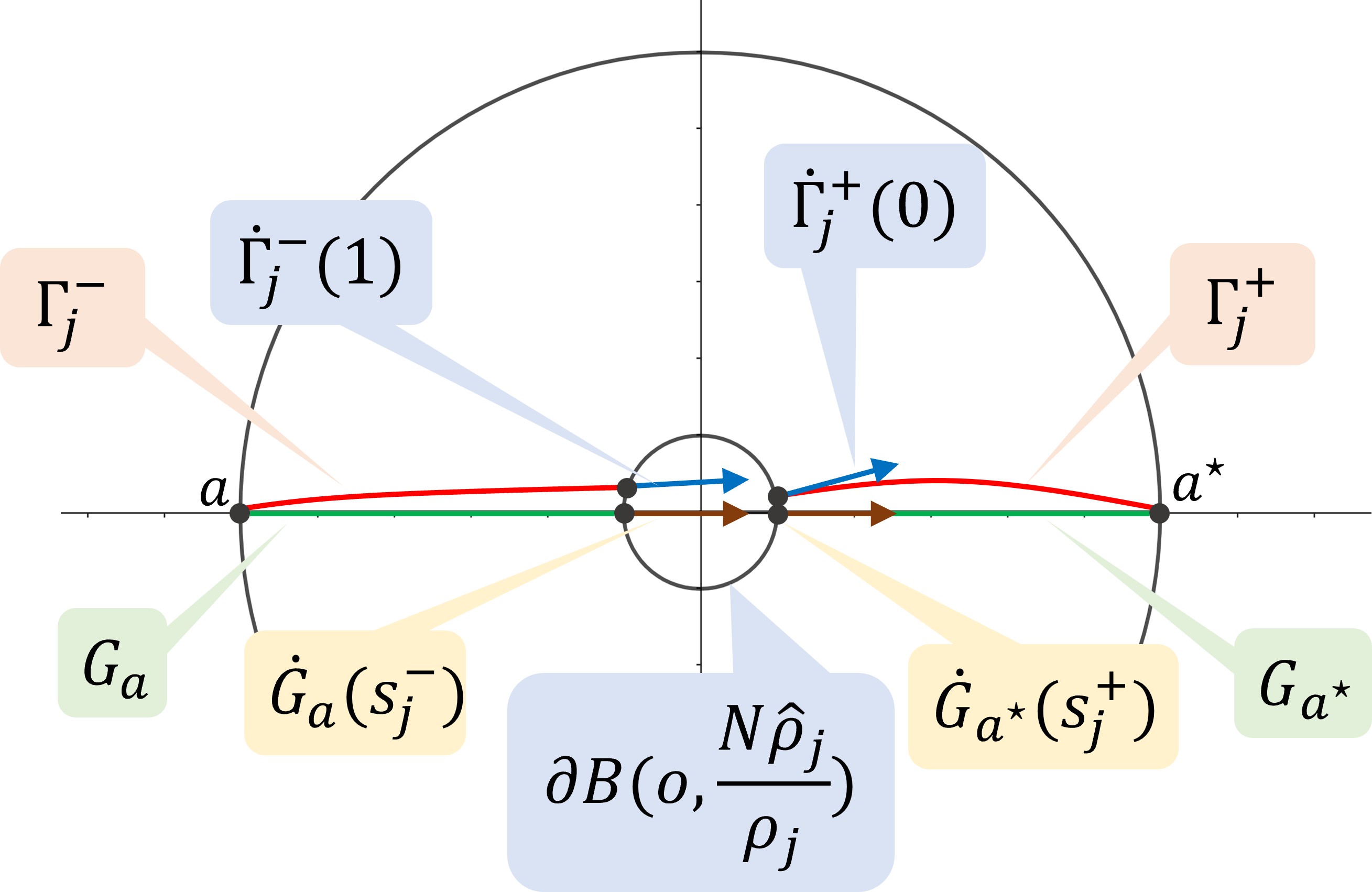}
\caption  {The estimate of angle $\angle(-\dot \Gz_j^-(1),p)$ and $\angle(\dot \Gz_j^+(0),q)$.}
\label{ngamma2}
\end{figure}
Note that the only geodesic segment in $\mathcal C_\az^n$ joining $a$ and $o$
is the part of the generatrix $G_a$ starting at $a$ and terminating at $o$.
Parameterising $G_a$ as
$$G_a:[0,1] \to \mathcal C_\az^n, \quad G_a(s) = (1-s)a.$$
Observe that
\begin{equation}\label{gp0}
  \angle (G_a(0),  G_a(s) )= \angle (G_a(s), -\dot G_a(s) )= 0  \quad \forall s \in [0,1].
\end{equation}
Let $s_j^- \in [0,1]$ be the point such that
$$  G_a(s_j^-) \in \pa B(o,\frac{\hat\rho_j N_0}{\rho_j}).  $$
Since $\{\Gz_{j}^-\}_{j \in \nn^{\ast\ast}}$  converge in $C^2$ to $G_a$, we deduce that
$$   |\dot \Gz_{j}^-(1)-\dot G_a(s_j^-)| \to 0, \quad \nn^{\ast\ast} \ni j \to \fz.  $$
Combining \eqref{gp0}, we have
 \begin{equation}\label{gpa}
  \angle(-\dot \Gz_{j}^-(1),a)= \angle(-\dot \Gz_{j}^-(1),G_a(0)) = \angle(-\dot \Gz_{j}^-(1),G_a(s_j^-)) \to 0, \quad \mbox{as $\nn^{\ast\ast} \ni  j \to \fz.$}
 \end{equation}
Similarly, let
$$G_{a^\star}:[0,1] \to \mathcal C_\az^n, \quad G_{a^\star}(s) = s a^\star$$
and $s_j^+ \in [0,1]$ such that
$$  G_a(s_j^+) \in \pa B(o,\frac{\hat\rho_j N_0}{\rho_j}).  $$
We have
\begin{equation}\label{gqa}
  \angle(\dot \Gz_{j}^+(0),a^\star)= \angle(\dot \Gz_{j}^+(0),G_{a^\star}(1)) = \angle(\dot \Gz_{j}^+(0),G_{a^\star}(s_j^+)) \to 0, \quad \mbox{as $\nn^{\ast\ast} \ni  j \to \fz.$}
 \end{equation}

 Recalling that $\mathcal{D}_{\rho_j^{-1}}$ is conformal for each $j$, we deduce from \eqref{gpa} and \eqref{gqa} that
 \begin{equation}\label{anglepi}
  \angle(-\dot \Gz^-_{\rho_j,N_0}(1),p_j)= \angle(-\dot \Gz_{j}^-(1),a)  \to 0,  \quad \mbox{as $\nn^{\ast\ast} \ni  j \to \fz.$}
   \end{equation}
  \begin{equation}\label{anglepi2}
 \angle(\dot \Gz^+_{\rho_j,N_0}(0),q_j)= \angle(\dot \Gz_{j}^+(0),a^\star)  \to 0, \quad \mbox{as $\nn^{\ast\ast} \ni  j \to \fz.$}
 \end{equation}
Since $a$ and $ a^\star$ lie in the same line $G=G_a\cup G_{a^\star}$, we have $\angle(a,a^\star)= \pi$. Thus we deduce from \eqref{anglepi} and \eqref{anglepi2} that
\begin{equation*}
  \angle(-\dot \Gz^-_{\rho_j,N_0}(1),\dot \Gz^+_{\rho_j,N_0}(0))  \to \pi,  \quad \mbox{as} \ \nn^{\ast\ast} \ni j \to \fz,
 \end{equation*}
 which gives \eqref{angleest3} and thus concludes Claim II as desired.


The proof is complete.
\end{proof}

\subsection{Non-twisting phenomenon (angles estimates) of min-max geodesic segments}

%

For  any  $r>C_\minm$ and $\rho>\max\{R_\kappa,r\}$, set
\begin{equation}\label{trhor}t^-_{\rho,r}:= \min \{t \in [0,\ell_{g}(\Gz_\rho)]:|\Gz_\rho(t)|= r \} \quad\mbox{and} \quad  t^+_{\rho,r}:= \max \{t \in [0,\ell_{g}(\Gz_\rho)]:|\Gz_\rho(t)|= r \}.
\end{equation}
Observe that the collection
    $$ \left\{  \frac{\Gz_{\rho}(t^\pm_{\rho,r})}{r} :\mbox{ $r>C_\minm$ and $\rho>\max\{R_\kappa,r\}$} \right\}\subset\mathbb S^{n-1}.
    $$
 Regards of the possible limits of this set when $r\to\infty$, we have the following result.

 \begin{lem} \label{nontwist}
For any  $\dz>0$, there exists $R_\dz\ge R_\kz$ such that
for all $  r>R_\dz$ and all $ \rho >r$,
\begin{equation}\label{ledz}
  \left |\frac{\Gz_{\rho } (t^-_{\rho,r})}{r} - a\right|+\left |\frac{\Gz_{\rho } (t^+_{\rho,r})}{r} - a^\star\right| \le \dz.
\end{equation}
 \end{lem}

 \begin{proof}
   We only show the first term in the left hand side of \eqref{ledz} is bounded by $\dz$; the proof for the second term is similar.
 We argue by contradiction. Under the counter assumption, we will find
a sequence $\{r_i\}_{i \in \nn}$ with $r_i \to \fz$ and
   a   subsequence  $\{\rho_{r_i}\}_{i\in\nn}$ with $\rho_{r_i}\ge r_i$
 such that
   \begin{equation}\label{b=}
     b:=\lim_{i \to \fz} \frac{\Gz_{\rho_{r_i}}(t^-_{\rho_{r_i},r_i})}{r_i} \ne a.
   \end{equation}
For simplicity, in the sequel
 of this
 proof, we write
$$\mbox{$ t^\pm_i= t^\pm_{\rho_{r_i},r_i}$ and
$\gz_i=\Gz_{\rho_{r_i}}$. }$$
   We will show the following two claims which contradict to each other, thus conclude the proof.

   \medskip
   \noindent
   \emph{Claim I.} There exists a subsequence $\nn^\ast \subset \nn$ such that
   $$ \lim_{\nn^\ast \ni i \to \fz} \angle(
\dot\gz_i(t^-_i) , -b) = 0.$$
   \emph{Claim II.}  There exists a subsequence $\nn^{\ast\ast} \subset \nn^\ast$ such that
   $$ \lim_{\nn^{\ast\ast} \ni i \to \fz} \angle(
\dot\gz_i(t^-_i), -b) \ge \angle(a,b)>0.$$

   \medskip
   \noindent
   \emph{Proof of Claim I.}
  For each $i \in  \nn$, recalling that $\gz_i \cap \overline{B(o, C_\minm) }\ne \emptyset$, we  find an $s_{i} \in (0,\ell_{g}(\gz_i))$ such that
   \begin{equation}\label{s=}
     |\gz_i(s_{ i})| \le C_\minm.
   \end{equation}
   By the definition of $t^\pm_{\rho,r}$, we see that
   $t^-_i < s_{i}< t^+_i
$. Let $X_{i} \in C^2([0,1],M)$ be the reparametrization of $\Gz_{ i}|_{[t^-_i
,s_{i}]}$ by a constant multiple of arc-length, that is
   $$ X_{i}(t):= \gz_i((s_{i}-t^-_i)t+ t^-_i
)  \quad\forall t \in [0,1]. $$
   Since $X_{i}$ is a geodesic segment of $(M,g)$, we know that $\mathcal{D}_{r_i^{-1}} \circ X_{i}$ is a geodesic segment of $M^{(r_i)}$.
   Thanks to \eqref{b=} and \eqref{s=}, one has
   $$  \lim_{i \to \fz} \mathcal{D}_{r_i^{-1}} \circ X_{i}(0) = \lim_{i \to \fz} \frac{\gz_i(t^-_i)}{r_i} = b \mbox{ and } \lim_{i \to \fz} \mathcal{D}_{r_i^{-1}} \circ X_{i}(1) = \lim_{i \to \fz} \frac{\gz_i(s_{i})}{r_i}=o.  $$
   Moreover, noting that
   $$  \ell_{g^{(r_i)}} (\mathcal{D}_{r_i^{-1}} \circ X_{i}) = \ell_{g^{(r_i)}} (\mathcal{D}_{r_i^{-1}} \circ \gz_i|_{[t^-_i
,s_{i}]}) \le \ell_{g^{(r_i)}} (\mathcal{D}_{r_i^{-1}} \circ \gz_i|_{[t^-_i,t^+_i]}),$$
   thanks to Lemma \ref{geoflow} and Remark \ref{gen}, we know
   $$ \sup_i \ell_{g^{(r_i)}} (\mathcal{D}_{r_i^{-1}} \circ X_{i}) < \fz.$$
   By the Arzela-Ascolli theorem,
there exists a subsequence $\nn^\ast \subset \nn$ such that $\{\mathcal{D}_{r_i^{-1}} \circ X_{i}\}_{i\in \nn^\ast}$ converges in $C^2$ to a geodesic segment of the standard cone $\ccc_\az^n$ joining $b$ and $o$, which is the unique generatrix segment $G_b: G_b(t)= (1-t)b, \ t \in [0,1]$. As a consequence, one has
   $$  \lim_{\nn^\ast \ni i \to \fz} \angle( (\mathcal{D}_{r_i^{-1}} \circ X_{i})'(0), \dot G_b(0))  = \lim_{\nn^\ast \ni i \to \fz} \angle ((\mathcal{D}_{r_i^{-1}} \circ X_{i})'(0), -b) =0.   $$
   Since $\mathcal{D}_{r_i^{-1}}$ is conformal for each $i$, we deduce that
   $$ \lim_{\nn^\ast \ni i \to \fz} \angle( \dot \gz_i(t^-_i) , -b) = \lim_{\nn^\ast \ni i \to \fz} \angle( \dot X_{i}(0), -b)  = \lim_{\nn^\ast \ni i \to \fz} \angle ((\mathcal{D}_{r_i^{-1}} \circ X_{i})'(0), -b) =0.  $$
   This shows Claim I.

   \medskip
   \noindent
   \emph{Proof of Claim II.} For each $i \in \nn^\ast$,
   let $Y_{i} \in C^2([0,1],M)$ be the reparametrization of $\gz_i|_
{[0,t^-_i]}$ by a constant multiple of arc-length, that is,
   $$ Y_{i}(t):= \gz_i(t^-_it) \quad\forall t \in [0,1] . $$
   Noting that $Y_{i}(0)= \gz_i(0) = \rho_{i}a$ and $ \rho_{i}> r_{i}$,
we   find a subsequence $\nn_1^\ast \subset \nn^\ast$ such that
   $$ \eta:= \lim_{\nn_1^\ast \ni i \to \fz}\frac{r_i}{\rho_{i}} \in [0,1] \mbox{ exists}.$$
   Recalling \eqref{b=}, one has
   $$  \lim_{\nn_1^\ast \ni i \to \fz} \mathcal{D}_{\rho_{i}^{-1}} \circ Y_{i}(0) = \lim_{\nn_1^\ast \ni i \to \fz} \frac{\rho_{i} a }{\rho_{i}} = a$$
and $$ \lim_{\nn_1^\ast \ni i \to \fz} \mathcal{D}_{\rho_{i}^{-1}} \circ Y_{i}(1) = \lim_{\nn_1^\ast \ni i \to \fz} \frac{\gz_i(t^-_i)}{\rho_i}   = \eta b.  $$
   Moreover, recalling $\ell_g(\gz_i) \le 2(1+\ez) \rho_{i}$, we deduce that
   $$ \ell_{g^{(\rho_{i})}}(\mathcal{D}_{\rho_{i}^{-1}} \circ Y_{i}) = \rho_{i}^{-1}\ell_g( Y_{i}) \le \rho_{i}^{-1} \ell_g( \gz_i)  \le  2(1+\ez).  $$
   By Arzela-Ascolli theorem,  there exists a subsequence $\nn^{\ast\ast} \subset \nn^\ast_1$ such that $\{\mathcal{D}_{\rho_{i}^{-1}} \circ Y_{i}\}_{i\in \nn^{\ast\ast}}$ converges in $C^2$ to a geodesic segment $\gz_\eta$ of the standard cone $\ccc_\az^n$ joining $a$ and $\eta b$.
   By the definition of $t^-_i$, we know that $Y_{i} \subset \rr^n \setminus B(o,r_i)$. Thus $\gz_\eta \subset \rr^n \setminus B(o,\eta)$ (if $\eta=0$, $\gz_\eta \subset \rr^n$).
   Moreover, one has
   $$  \lim_{\nn^{\ast\ast} \ni i \to \fz} \angle((\mathcal{D}_{\rho_{i}^{-1}} \circ Y_{i})'(1), \dot \gz_\eta(1))=0.   $$
   By Lemma \ref{ab}, we have $\angle(\dot \gz_\eta(1),-b) \ge \angle(a,b)$. Thus
   $$  \lim_{\nn^{\ast\ast} \ni i \to \fz} \angle((\mathcal{D}_{\rho_{i}^{-1}} \circ Y_{i})'(1), -b)\ge \angle(a,b).   $$
   Since $\mathcal{D}_{\rho_{i}^{-1}}$ is conformal for each $i$, we deduce that
   $$ \lim_{\nn^{\ast\ast} \ni i \to \fz} \angle(\dot\gz_i(t^-i
), -b)=\lim_{\nn^{\ast\ast} \ni i \to \fz} \angle((\mathcal{D}_{\rho_{i}^{-1}} \circ Y_{i})'(1), -b)\ge \angle(a,b),  $$
   which shows Claim II.

   The proof is complete.
 \end{proof}

 With the help of Lemma \ref{nontwist},  
 we investigate the sets
 $$\Delta_{\rho,r}:=\{t\in[0,\ell_{g}(\Gz_\rho)]: |\Gz_\rho (t)|=r\}\quad\mbox{for all  $
 r>C_\minm, \rho  >\max\{r,R_\kz\}$.}$$
We additionally need to consider the angle $\angle(\Gz_\rho (t), \dot \Gz_\rho (t))$ at each $ t\in\Delta_{\rho,r}$.

  For all possible  $ r $ and $\rho$ we further set
  $$ \Delta^\times_{\rho,r}:= \{t\in[0,\ell_{g}(\Gz_\rho)]: |\Gz_\rho (t)|=r, \ \angle(\Gz_\rho (t), \dot \Gz_\rho (t))= \frac{\pi}2\}.
$$
 Note that $\Delta^\times_{\rho,r}$ is  a closed subset of $\Delta_{\rho,r}$.
 For each $t \in \Delta^\times_{\rho,r}$, let $I_t$ be the connected component of $\Delta^\times_{\rho,r}$ that contains $t$.
 Recall that a closed connected subset of $\rr$ is either a single point or a closed interval.

 \medskip
$\bullet$ If $I_t=\{t\}$ is a single point, we have the following 3 types of $t$ (see Figure \ref{type1}):

  \vspace*{1pt}
\begin{figure}[h]
\centering
\includegraphics[width=16cm]{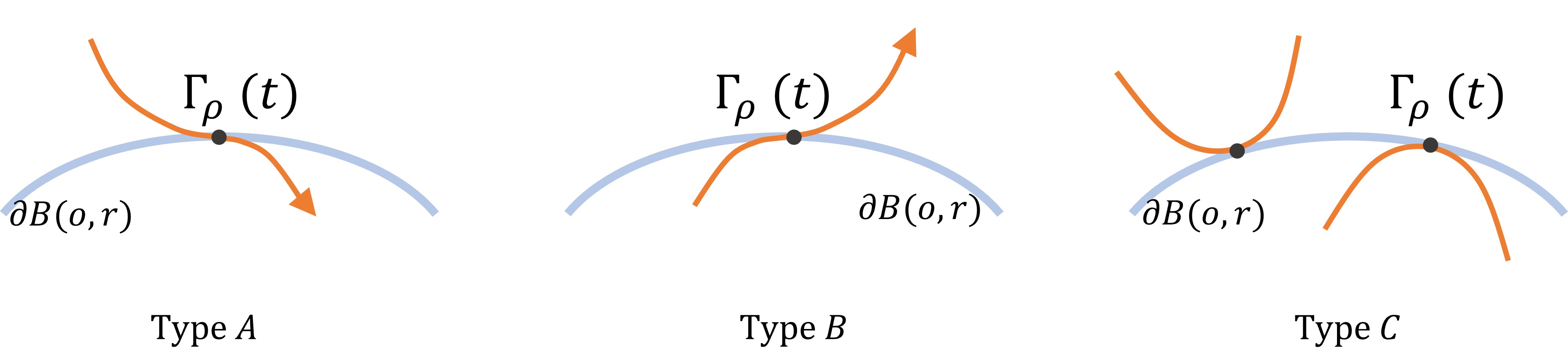}
\caption  {Types A, B and C.}
\label{type1}
\end{figure}
\begin{enumerate}
  \item[A.] $|\Gz_\rho (\cdot)|$ is strictly decreasing at $t$;
  \item[B.] $|\Gz_\rho (\cdot)|$ is strictly increasing at $t$;
  \item[C.] $|\Gz_\rho (\cdot)|$ is a local minimum/maximam at $t$.
\end{enumerate}

\medskip
$\bullet$
If $I_t=[t^-,t^+]$ is a closed interval, we also have the following 3 types of $t$ (see Figure \ref{type2}):
\vspace{6pt}
\begin{figure}[h]
\centering
\includegraphics[width=16cm]{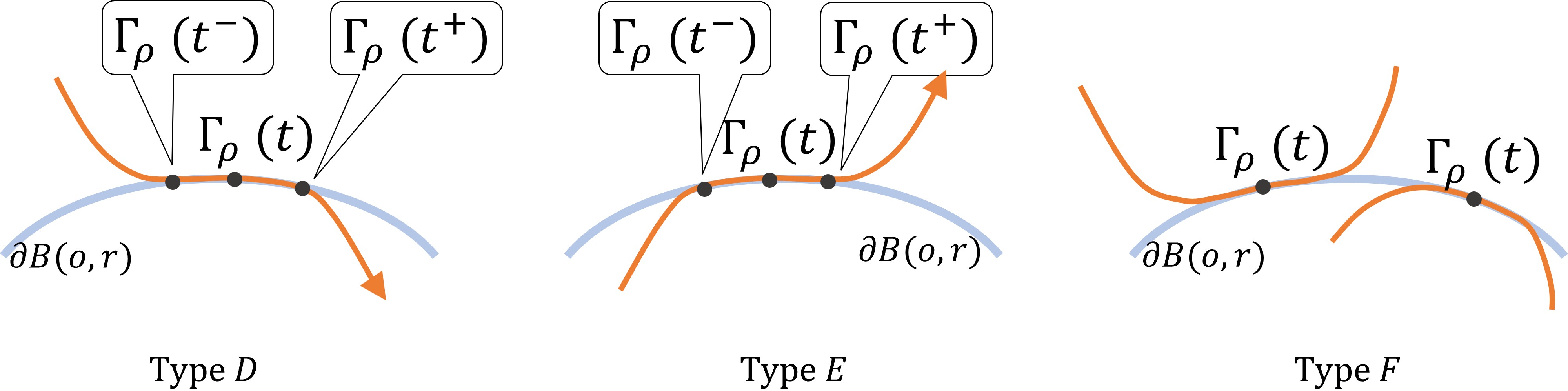}
\caption  {Types D, E and F.}
\label{type2}
\end{figure}

\begin{enumerate}
  \item[D.] $|\Gz_\rho (\cdot)|$ is decreasing at both $t^\pm$;
  \item[E.] $|\Gz_\rho (\cdot)|$ is increasing at both $t^\pm$;
  \item[F.] $t^\pm$ do not belong to D or E.
\end{enumerate}

Based on the above,    we partition $\Delta_{\rho ,r}$ into the following three sets:
 $$  \Delta_{\rho ,r}  =  \Delta^-_{\rho ,r} \cup \Delta^+_{\rho ,r} \cup \Delta^\circ_{\rho,r},$$
where
$$
  \Delta^-_{\rho ,r}: =  \{t \in [0,\ell_{g}(\Gz_\rho)]:|\Gz_\rho (t)|= r, \ \angle(\Gz_\rho (t), \dot \Gz_\rho (t)) > \frac{\pi}2 \}
   \ \bigcup \ \{t \in \Delta^\times_{\rho,r}: \mbox{ $t$ has type A or D}  \},
$$
$$
  \Delta^+_{\rho ,r}: = \{t \in [0,\ell_{g}(\Gz_\rho)]:|\Gz_\rho (t)|= r, \ \angle(\Gz_\rho (t), \dot \Gz_\rho (t)) < \frac{\pi}2 \}  \ \bigcup \ \{t \in \Delta^\times_{\rho,r}: \mbox{ $t$ has type B or E}  \},
$$
    and
    $$\Delta^\circ_{\rho,r}:= \{t \in \Delta^\times_{\rho,r}: \mbox{ $t$ has type C or F} \}. $$

    Geometrically, for any $t \in \Delta^-_{\rho ,r}$ (resp. $t \in \Delta^+_{\rho ,r}$), there exists a small neighbourhood $[t-\dz, t+\dz'])$ such that $\Gz_\rho |_{[t-\dz,t]}$ and $\Gz_\rho |_{[t,t+\dz']}$ lie outside and inside (resp. inside and outside) the ball $B(o,r)$ respectively. If in addition, $t$ has type $D$ or $E$, we know $\dz = t-t^-$ and $\dz' = t^+-t$.

    For the case $t \in \Delta^\circ_{\rho,r}$, note that the condition $|\Gz_\rho (\cdot)|$ is a local minimum/maximum at $t$ is equivalent to $\angle(\Gz_\rho (\cdot), \dot \Gz_\rho (\cdot))- \pi/2$ is either equal to $0$ or changing sign at $t$. Geometrically, it means that there exists a small neighbourhood $[t-\dz, t+\dz]$ such that $\Gz_\rho |_{[t-\dz,t+\dz]}$ lies in the same side of the ball $B(o,r)$. If moreover $t$ has type F, then
$$|\Gz_\rho ([t-\dz,t+\dz])| \equiv r \  \mbox{(i.e.}\quad
  \angle(\Gz_\rho (\cdot), \dot \Gz_\rho (\cdot))|_{[t-\dz,t+\dz]}\equiv \frac{\pi}2,\ \mbox{ or} \ \Gz_\rho |_{[t-\dz,t+\dz]} \subset \pa B(o,r) \mbox{)} $$


We have the following result.
    \begin{lem} \label{delta}
    \begin{enumerate}
      \item[(i)] We have $$ \max\{\sup_\rho  \diam \Gz_{\rho}(\Delta^-_{\rho ,r}), \ \sup_\rho  \diam \Gz_{\rho}(\Delta^+_{\rho ,r}),  \ \sup_\rho  \diam \Gz_{\rho}(\Delta^\circ_{\rho,r})\} = o(r)\quad \mbox{ as } r \to \fz. $$
      \item[(ii)] There exists ${\bf r}>\max\{2C_\minm,R_\kz\}$  such that
        $$\Delta^\circ_{\rho,r}=\emptyset \mbox{ whenever $\rho  > r\ge \bf r$}. $$
        \item[(iii)] For all $\rho>r>{\bf r}$,
         $$t_{\rho,r}^- \in \Delta^-_{\rho ,r} \mbox{ and } t_{\rho,r}^+ \in \Delta^+_{\rho ,r}.$$
         \item[(iv)] For all $\rho>r>{\bf r}$,
         $$ \Delta^-_{\rho ,r} \mbox{ and } \Delta^+_{\rho ,r} \mbox{ are connected sets}.$$
         \item[(v)] For any $\rho>r>{\bf r}$, let $t_\rho$ be any point in $[0,\ell_g(\Gz_\rho)]$ such that $\Gz_\rho(t_\rho) \in B(o,2C_\minm)$. Then,
         $$ \Delta^-_{\rho ,r} \subset [0, t_\rho) \mbox{ and } \Delta^+_{\rho ,r} \subset (t_\rho, \ell_g(\Gz_\rho)].$$
        Moreover, for any $t \in [0, t_\rho)$ and $t' \in (t_\rho, \ell_g(\Gz_\rho)]$ with $|\Gz_\rho(t)| = |\Gz_\rho(t')| = r$, one has
         $$  t\in \Delta^-_{\rho ,r} \mbox{ and }  t'\in \Delta^+_{\rho ,r}. $$
    \end{enumerate}
    \end{lem}

    \begin{proof}
   We first show (i). We only prove
\begin{equation}\label{rhoor}\sup_\rho \diam \Gz_{\rho}(\Delta^-_{\rho,r}) =o(r)
\quad\mbox{as $r\to\infty$;}
\end{equation}
  the proofs of the other two are similar.
We argue by contradiction. Assume that \eqref{rhoor} is not true.
There is a constant $c_0>0$ and a sequence $\{r_i\}_{i \in \nn}$ with $r_i \to \fz$ such that
$$ \frac1{r_i}\sup_\rho \diam \Gz_{\rho}(\Delta^-_{\rho,r_i}) \ge c_0\quad\forall i\in\nn.$$
We can further get a sequence $\{\rho_i\}$ with $\rho_i>\max\{r_i,R_\kz\}$ such that
$$ \frac1{r_i}  \diam \Gz_{\rho_i}(\Delta^-_{\rho_i,r_i}) \ge \frac{c_0}2\quad\forall i\in\nn.$$
This  allows us to choose $\{s^1_{\rho_i},  s^2_{\rho_i}\}\subset \Delta^-_{\rho_i,r_i}  $
with  $0<s^1_{\rho_i}< s^2_{\rho_i}<\ell_{g}(\Gz_{\rho_i})$
   such that
   \begin{equation}\label{distc}
     \dist(\Gz_{\rho_i}(s^1_{\rho_i}), \Gz_{\rho_i}(s^2_{\rho_i})) \ge \frac{c_0}4 r_i\quad\forall i\in\nn.
   \end{equation}

For each $i\in\nn$, let $X_{\rho_i} \in C^2([0,1],M)$ be the reparametrization of $\Gz_{\rho_i}|_{[s^1_{\rho_i}, s^2_{\rho_i}]}$ by a constant multiple of arc-length, that is
   $$ X_{\rho_i}(t):= \Gz_{\rho_i}((s^2_{\rho_i}-s^1_{\rho_i})t+ s^1_{\rho_i}) \quad\forall t \in [0,1]. $$
   Note  that
 $$\left\{  \frac{X_{\rho_i}(0)}{r_i},
      \frac{X_{\rho_i}(1)}{r_i}\right\}\subset\partial B(0,1).$$
      Up to some subsequence, we may assume
   $$
\lim_{i\to\infty} \frac{X_{\rho_i}(0)}{r_i}=b^0 \in \partial B(0,1) \quad\mbox{and} \quad
\lim_{i\to\infty} \frac{X_{\rho_i}(1)}{r_i}=b^1 \in \partial B(0,1) .$$
%
%
  By \eqref{distc},
  $$  \dist(b^0,b^1) =\lim_{i\to\infty}\dist( \frac{X_{\rho_i}(0)}{r_i},
   \frac{X_{\rho_i}(1)}{r_i}) \ge \frac{ c_0}4. $$
 Moreover, since $X_{\rho_i}$ is a geodesic segment of $(M,g)$, we know that $\mathcal{D}_{r_i^{-1}} \circ X_{\rho_i}$ is a geodesic segment of $M^{(r_i)}$.
   Thanks to Lemma \ref{geoflow} and Remark \ref{gen}, we know
   $$ \sup_i \ell_{g^{(r_i)}} (\mathcal{D}_{r_i^{-1}} \circ X_{\rho_i}) < \fz.$$
   By the Arzela-Ascolli theorem,  there exists a subsequence $\nn^\ast \subset \nn$ such that converges in $C^2$ to a geodesic segment $\gz: [0,1] \to \ccc_\az^n$ of the standard cone joining $b_0=\gz(0),b_1=\gz(1) $.
Since  $b^1 \ne b^0$ and  $b_0,b_1 \in\pa B(0,1)$, by Remark \ref{pi/2}, we have $ \gz\subset B(0,1)$ and
   $$ \angle(\gz(1), \dot \gz(1)) < \frac{\pi}{2}.$$
   Thus there exists $N$ sufficiently large such that
   $$  \angle(\mathcal{D}_{r_i^{-1}} \circ X_{\rho_i}(1), (\mathcal{D}_{r_i^{-1}} \circ X_{\rho_i})'(1)) < \frac{\pi}{2}  \quad\forall i \in \nn^\ast \cap [N,\fz). $$
   Since $\mathcal{D}_{r_i^{-1}}$ is conformal, we deduce that
   $$ \angle( \Gz_{\rho_i}(s^2_{\rho_i}), \dot \Gz_{\rho_i}(s^2_{\rho_i}))= \angle(\mathcal{D}_{r_i^{-1}} \circ X_{\rho_i}(1), (\mathcal{D}_{r_i^{-1}} \circ X_{\rho_i})'(1)) < \frac{\pi}{2} \quad\forall i \in \nn^\ast \cap [N,\fz). $$
   This contradicts to the fact that $s^2_{\rho_i} \in \Delta^-_{\rho_i,r_i}$, which gives (i).

   \medskip
   We show (ii).
   By contradiction,
 we can find
 a sequence $\{r_i\}_{i \in \nn}$ with $r_i \to \fz$  and $\{\rho_i\}$
 with $\rho_i>r_i$ such  that $\Delta_{\rho_i,r_i}^\circ\ne\emptyset$.
 Under this counter assumption, we claim that
\begin{equation} \label{contra}\diam(\Gz_{\rho_i}(\Delta_{\rho_i,r_i}
) )= o(r_i) \mbox{ as $i \to \fz$}.
\end{equation}
 Obviously, this  contracts with Lemma \ref{nontwist} which gives
 $$\liminf_{i\to\infty}\frac1{r_i}\diam(\Gz_{\rho_i}(\Delta_{\rho_i,r_i}))
 \ge
 \lim_{i\to\infty} \frac1{r_i}| \Gz_{\rho_i}(t^+_{\rho_i,r_i})- \Gz_{\rho_i}(t^-_{\rho_i,r_i}) | =|a-a^\star|>0.$$

 To see the claim \eqref{contra}, note that the counter assumption guarantees that $\Delta^\circ_{\rho_i,r_i}$ contains some point $s_i$. Then
 \begin{align*}
 \diam(\Gz_{\rho_i}(\Delta_{\rho_i,r_i}))& \le\sup\{
 \dist(\Gamma_{\rho_i}(s), \Gamma_{\rho_i}(t)) :s,t\in\Delta_{\rho_i,r_i}\}\\
& \le 2 \sup\{
 \dist(\Gamma_{\rho_i}(t), \Gamma_{\rho_i}(s_i)) :t\in\Delta_{\rho_i,r_i}\}\\
 &\le 2 \sup\{
 \dist(\Gamma_{\rho_i}(t), \Gamma_{\rho_i}(s_i)) :t\in\Delta_{\rho_i,r_i}^\pm\}+ 2\diam(\Gz_{\rho_i}( \Delta^\circ_{\rho_i,r_i})).
 \end{align*}
 Since (i) gives $\diam(\Gz_{\rho_i}( \Delta^\circ_{\rho_i,r_i}))=o(r_i)$,
 to get the claim \eqref{contra}, we only need to show the bound
 $$
  \sup\{
 \dist(\Gamma_{\rho_i}(t), \Gamma_{\rho_i}(s_i)) : t\in\Delta^\pm_{\rho_i,r_i}\}=o(r_i).$$
Recall that $s_i $ is of either type C or type F.
We consider  the two cases separately.

\medskip
 {\it Case that $s_{ i}$ is of type $C$}.
  It is clear  that
   \begin{equation}\label{or3}
     \mbox{ $s_i \in \overline{\Delta^+_{\rho_i,r_i}}$  and $s_{\rho_i} \in \overline{\Delta^-_{\rho_i,r_i}}$}.
   \end{equation}
   For any     $t \in \Delta^\pm_{\rho_i,r_i}$ 
   we deduce that
   \begin{align}\label{or4}
     \dist(\Gz_{\rho_i}(t),\Gz_{\rho_i}(s_i)) 
      & \le \diam(\Gz_{\rho_i}(\overline{\Delta^\pm_{\rho_i,r_i}}))
  = \diam(\Gz_{\rho_i}( \Delta^\pm_{\rho_i,r_i})),
   \end{align}
 thanks to (i),
which gives the desired bound $o(r_i)$ as $i\to \infty$.

\medskip
{\it Case that $s_{i}$ is of type $F$.}
Recall that $I_{s_{i}}=[s_{i}^-,s_{i}^+]$ is the connected component in $\Delta^\circ_{\rho_i,r_i}$.
         We observe that
   \begin{equation}\label{or2}
     \mbox{either $s_{i}^\pm \in \overline{\Delta^\pm_{\rho_i,r_i}}$ or $s_{ i}^\pm \in \overline{\Delta^\mp_{\rho_i,r_i}}$},
   \end{equation}
  where due to the type F of $s^\pm_i$,  we rule out the cases that $s_{ i}^\pm \in \overline{\Delta^+_{\rho_i,r_i}}$  and $s_{ i}^\pm \in \overline{\Delta^-_{\rho_i,r_i}}$.
Without loss of generality, we may assume     $s_{i}^\pm \in \overline{\Delta^\pm_{\rho_i,r_i}}$.

   For any   point $t \in \Delta^\pm_{\rho_i,r_i}$, we have
    \begin{align*}\label{or5}
   \dist(\Gz_{\rho_i}(t),\Gz_{\rho_i}(s _i))&\le
     \dist(\Gz_{\rho_i}(t),\Gz_{\rho_i}(s^\pm_i))    + \dist(\Gz_{\rho_i}(s_i),\Gz_{\rho_i}(s^\pm_i)) \\
     & \le \diam (\Gamma_{\rho_i}(\Delta^\pm_{\rho_i,r_i}) )
     +
     \diam (\Gamma_{\rho_i}(\Delta^\circ_{\rho_i,r_i}) ) ,
   \end{align*}
   thanks to (i),
which also gives the desired bound $o(r_i)$ as $i\to \infty$.
 This completes the proof (ii).

\medskip
   To see (iii), by the definition of $t_{\rho,r}^-$, we know $|\Gz_{\rho}(t)| > r$ for all $t \in [0,t_{j,r}^-)$. Thus $|\Gz_j(\cdot)|$ is strictly decreasing from left at $t_{j,r}^-$. This implies that $t_{j,r}^- \in \Delta^-_{j_i,r_i} \cup \Delta^\circ_{j_i,r_i}$. Since $r > {\bf r}$, by (ii), $\Delta^\circ_{j_i,r_i}$ is empty. We deduce that $t_{j,r}^- \in \Delta^-_{j_i,r_i}$. The proof of $t_{j,r}^+ \in \Delta^+_{j_i,r_i}$ is similar.

   \medskip
   To see (iv), we only prove $\Delta^-_{\rho,r}$ is connected and the proof for $\Delta^+_{\rho,r}$ is similar.
   We argue by contradiction. Assume that $
 \Delta^-_{\rho,r}$ is not connected  for some $\rho,r$.
   Note that the set $\Delta_{\rho,r}$ is   closed, and by (ii),
   $$\Delta_{\rho,r} = \Delta^-_{\rho,r} \cup \Delta^+_{\rho,r}\quad\mbox{and} \quad \Delta^-_{\rho,r} \cap \Delta^+_{\rho,r} = \emptyset.$$
Then  the
sets  $\Delta^\pm_{\rho,r}$ are closed.
By the assumption,    there must be two connected components
   $$\mbox{$I_1= [t_1,t_1'],I_2 =[t_2,t_2'] \subset \Delta^-_{\rho,r}$ with $t_1' < t_2$}.$$
   By the definition of $\Delta^-_{\rho,r}$, we know that $|\Gz_\rho(\cdot)|$ is strictly decreasing from left at $t_2$.
   Thus, there exists $t_3 \in (t_1',t_2)$ such that
   $$r':= |\Gz_\rho(t_3)| =\max_{t_3 \in [t_1',t_2]}|\Gz_\rho(t)| > |\Gz_\rho(t_1')| =  |\Gz_\rho(t_2)| =r. $$
   Thus, $t_3$ has type C or F and consequently $t_3 \in \Delta^\circ_{\rho,r'}$. This contradicts to (ii).

   \medskip
   To see (v), we only show $\Delta^-_{\rho,r} \subset [0,t_\rho)$ and the other one is similar.
   Since $|\Gz_\rho(t_\rho)| < 2C_\minm < r$, one has $t_\rho \notin \Delta^-_{\rho,r}$. By (iv), the connectedness of $\Delta^-_{\rho,r}$ implies that
   $$ \mbox{ either $\Delta^-_{\rho,r} \subset [0,t_\rho)$ or $\Delta^-_{\rho,r} \subset (t_\rho,\ell_g(\Gz_\rho)]$}.$$
   Obviously, $t^-_{\rho,r} \in [0,t_\rho)$. Combining the fact from (iii) that $t^-_{\rho,r} \in \Delta^-_{\rho,r}$, we get the desire result.

   Finally, if $t \in [0,t_\rho)$ with $|\Gz(t)|=r$, we first know $t \in \Delta_{\rho,r}$. Since $\Delta_{\rho,r} = \Delta^-_{\rho,r} \cup \Delta^+_{\rho,r}$ and $\Delta^+_{\rho,r}\cap [0,t_\rho) = \emptyset$, we conclude $t \in \Delta^-_{\rho,r}$.
   The same proof also works for $t'$.

   The proof of Lemma \ref{delta} is complete.
    \end{proof}

A direct consequence of Lemma \ref{nontwist} and Lemma \ref{delta} is the following.

\begin{cor}\label{closure}
One has
    $$\lim_{r \to \fz} \sup \left\{|a-\frac{x}{r}|: x \in \bigcup_{\rho > r}\Gz_\rho(\Delta^-_{\rho,r}) \right\} = 0, \mbox{ and } \lim_{r \to \fz} \sup \left\{|a^\star-\frac{y}{r}|: y \in \bigcup_{\rho > r}\Gz_\rho(\Delta^+_{\rho,r}) \right\} = 0.$$
\end{cor}

\begin{proof}
   We only show the first equality and the second one is similar.
   For any $r>\max\{R_\kz,{\bf r}\}$ and $\rho> r$, by Lemma \ref{delta} (iii), we know that $t^-_{\rho,r} \in \Delta^-_{\rho,r}$. For any $x \in \cup_{\rho > r}\Gz_\rho(\Delta^-_{\rho,r})$, we have $x \in \Gz_{\rho_x}(\Delta^-_{\rho_x,r})$ for some $\rho_x>r$.
    Using the triangle inequality, we deduce that
   $$ |ra-x| \le |ra-\Gz_{\rho_x}(t^-_{_{\rho_x},r})| + |\Gz_{\rho_x}(t^-_{\rho_x,r})-x| \le |ra-\Gz_{\rho_x}(t^-_{_{\rho_x},r})| + \diam \Gz_{\rho_x}(\Delta^-_{\rho_x,r}),$$
  thanks to Lemma \ref{nontwist} and Lemma \ref{delta} (i),
   $$ |a-\frac xr|   \le |a-\frac{\Gz_{\rho_x}(t^-_{_{\rho_x},r})} r| + \frac1r\diam \Gz_{\rho_x}(\Delta^-_{\rho_x,r}) \to 0 \mbox{ as $r\to\infty$.}$$
   The proof is complete.
\end{proof}

\subsection{Proof of Theorem \ref{ndimanti}}
In this subsection,
we prove Theorem \ref{ndimanti}.

\begin{proof}[Proof of Theorem \ref{ndimanti}]
Recall ${\bf r}> C_\minm$ in Lemma \ref{delta} (ii).
Let $\{\rho_j=2^j{\bf r}\}_{j\in\nn}$ and
 $\{\Gz_{\rho_j}\}_{j\in \nn}$ be the sequence of min-max geodesic segments as in the beginning of   Section 5.
Recall that each $\Gz_{\rho_j} $ is parameterised by the arc-length with respect to $g$.
 By Lemma \ref{unif},
$\Gz_{\rho_j} \cap \overline {B(o,C_\minm)}\ne\emptyset$ for all $j \in \nn$.
  For simple we write $ \Gamma_j$ as  $\Gamma_{\rho_j}$ and $T_j$ as its  length $\ell(\Gz_j)$.

 The idea to prove  Theorem \ref{ndimanti} is that:  For each $i \in \nn$, we will show that certain restrictions of $\{ \Gz_j\}_{j\in\nn}$ in $\overline {B(o,2^i{\bf r})}$ converge  (up to some subsequence) to some geodesic   segment $\Gz_i^\ast$.  The union of  $\Gz_i^\ast$ then gives the desired geodesic line $ \Gz$.

  We proceed by the following 4 steps.

     \medskip
   \emph{Step 1. A uniform length estimate for certain restriction of $\Gamma_j$.}

          For each $r\ge 2{\bf r}$, %
and
 each $j \ge r/2\bf r$,
thanks to Lemma \ref{unif}, we know that
   $ \Gz_j \cap \pa B(o,r) $
   contains at least two points.
   Write  $t_{j,r}^\pm$ as $ t_{\rho_j,r}^\pm$ given in \eqref{trhor}, that is,
   \begin{equation}\label{tjrpm}
      t_{j,r}^-:= \min \{t \in [0,T_j]:|\Gz_j(t)|= r \}, \quad  t_{j,r}^+:= \max \{t \in [0,T_j]:|\Gz_j(t)|=r\}.
   \end{equation}
Obviously,
 $\Gz_j \cap B(o,r) \subset \Gz_j |_{[t_{j,r}^-,t_{j,r}^+]}.$
 We claim that
   \begin{equation}\label{lengthb}
     \liminf_{j \to \fz}\ell_{g}(\Gz_j |_{[t_{j,r}^-,t_{j,r}^+]}) <\fz.
   \end{equation}

      To see \eqref{lengthb},
  set
   $$  p_{j,r}:= \Gz_j(t_{j,r}^-), \quad \xi_{j,r}:=\dot \Gz_j(t_{j,r}^-) $$
   and
   $$  q_{j,r}:= \Gz_j(t_{j,r}^+), \quad \zeta_{j,r}:=\dot \Gz_j(t_{j,r}^+). $$
   There is a subsequence $\nn_0 \subset \nn$ such that
   \begin{equation}\label{pjr}
     \lim_{\nn_0\ni j \to \fz}(p_{j,r}, \xi_{j,r} )=\lim_{\nn_0 \ni j \to \fz}(\Gz_j(t_{j,r}^-), \dot \Gz_j(t_{j,r}^-) ) = (p_{r}^\ast, \xi_r ^\ast) \in \pa B(o,r) \times \mathcal{S}_{p_r^\ast}M
   \end{equation}
    and
    \begin{equation}\label{qjr}
      \lim_{\nn_0 \ni j \to \fz}(q_{j,r}, \zeta_{j,r} )=\lim_{\nn_0 \ni j \to \fz}(\Gz_j(t_{j,r}^+), \dot \Gz_j(t_{j,r}^+) ) = (q_r
     ^\ast, \zeta_r
     ^\ast) \in \pa B(o,r) \times \mathcal{S}_{q_{r}^\ast}M.
    \end{equation}
   Consider geodesic rays $\Gz_{p_{j,r},\xi_{j,r}}$ of $(M,g)$ (i.e. $\Gz_{p_{j,r},\xi_{j,r}}(0) = p_{j,r}$ and $\dot\Gz_{p_{j,r},\xi_{j,r}}(0)=\xi_{j,r}$).  One has
    $$ \Gz_{p_{j,r},\xi_{j,r}}(t) =  \Gz_j(t+t_{j,r}^-) , \quad \forall t\in[0,T_j-t_{j,r}^-].  $$
             Thanks to \eqref{pjr}, we apply Lemma \ref{geoflow} to the sequence $\{ \Gz_{p_{j,r},\xi_{j,r}}\}_{j \in \nn_0}$ to deduce
           \begin{equation}\label{lengthb2}
     \liminf_{j \to \fz}\ell_{g}(\Gz_{j} |_{[t_{j,r}^+,t_{j,r}^-]}) =\liminf_{j \to \fz}\ell_{g}(\Gz_{p_{j,r},\xi_{j,r}} |_{[0,t_{j,r}^+-t_{j,r}^-]}) <\fz,
   \end{equation}
that is, \eqref{lengthb} holds.

   \medskip
   \emph{Step 2. The construction of increasing geodesic segments
      $\Gz_i^\ast$ with endpoints in $\partial B(o,2^i {\bf r})$.}

   Write
   $$r_i=2^i
    {\bf r} \mbox{ for all } i\ge 1.$$

   Letting $r=r_1$ in \eqref{lengthb}, thanks to \eqref{lengthb2}, one apply Arzela-Ascolli theorem to deduce that,
   there exists some subsequence $ \nn_1\subset \nn_0$ such that
   $$\{\Gz_{p_{j,r_1},\xi_{j,r_1}}|_{[0,t_{j,r_1}^+-t_{j,r_1}^-]}\}_{j\in \nn_1}  =
    \{\Gz_j(\cdot + t_{j,r_1}^-)|_{[0,t_{j,r_1}^+-t_{j,r_1}^-]} \}_{j\in \nn_1}$$
   converges in $C^1$ to a geodesic segment $\Gz^\ast_1:[0,T_1^\ast]\to M$ with
   \begin{equation*}
     (\Gz^\ast_1(0),\dot \Gz^\ast_1(0))=(p_{r_1}^\ast,\xi_{r_1}^\ast)  \mbox{ and } (\Gz^\ast_1(T_1^\ast),\dot\Gz^\ast_1(T_1^\ast))=(q_{r_1}^\ast,\zeta_{r_1}^\ast).
   \end{equation*}
  In other words,   $T_1^\ast=\lim_{j\to\infty}[t_{j,r_1}^+-t_{j,r_1}^-] $
  and
  $$
   \Gz^\ast_1(t)=\lim_{\nn_1\ni j\to\infty}\Gz_{p_{j,r_1},\xi_{j,r_1}}(t)= \lim_{j\to\infty}\Gz_j(t+t_{j,r_1}^-) $$
  for all $t\in[0,T^\ast_1]$ in $C^1$.
Since  $\Gz_j$ has
the Morse index $\le n-1$, we know that  $\Gz^\ast_1$ has the Morse index $\le n-1$.

In general, for  $i\ge2$, letting $r=r_i$ in \eqref{lengthb}, thanks to \eqref{lengthb2}, one  apply the Arzela-Ascolli theorem to deduce that,
   for some subsequence $ \nn_i\subset \nn_{i-1}$,
 $ t^{\pm}_{j,r_i}-  t^{\pm}_{j,r_ {i-1}}$ converges to $u^\pm_i$, and
   $$\{\Gz_{p_{j,r_i},\xi_{j,r_i}}|_{[0,t_{j,r_i}^+-t_{j,r_i}^-]}\}_{j\in \nn_i} = \{\Gz_j(\cdot + t_{j,r_i}^-)| _{[0,t_{j,r_i}^+-t_{j,r_i}^-]} \}_{j\in \nn_i}
$$
   converge in $C^1$ to a geodesic segment $\Gz^\ast_i:[0,T_i^\ast]\to M$ with
   \begin{equation}\label{Gz1}
     (\Gz^\ast_i(0),\dot \Gz^\ast_i(0))=(p_{r_i}^\ast,\xi_{r_i}^\ast)  \mbox{ and } (\Gz^\ast_i(T_i^\ast),\dot\Gz^\ast_i(T_i^\ast))=(q_{r_i}
     ^\ast,\zeta_{r_i}^\ast).
   \end{equation}
    In other words,   $T_i^\ast=\lim_{j\to\infty}[t_{j,r_i}^+-t_{j,r_i}^-]$
  and
  \begin{equation}\label{ri}
     \Gz^\ast_i(t)=\lim_{\nn_i\ni j\to\infty} \Gz_{p_{j,r_i},\xi_{j,r_i}}(t)= \lim_{j\to\infty}\Gz_j(t+t_{j,r_i}^-)
  \end{equation}
  for all $t\in[0,T^\ast_i]$  in $C^1$.
This also implies $\Gz^\ast_i
$ has the Morse index $\le n-1$.

Note that for each $i \in \nn$, $\nn_{i+1} \subset \nn_i$. Thus as subsets in $M$, $\Gz_i^\ast \subset \Gz_{i+1}^\ast$. In particular,
$$ \Gz_1^\ast \subset \Gz_{i}^\ast, \quad i \in \nn. $$
This implies that for each $i \in \nn$, there exists $0< u_i^-<u_i^+ < T_i^\ast$ (with $u_i^+-u_i^-=T_1^\ast$) such that
$$  \Gz_i^\ast(u_i^-) = \Gz_1^\ast(0), \mbox{ and } \Gz_i^\ast(u_i^+) = \Gz_1^\ast(T_1^\ast).   $$
We make the convention that $u_1^- = 0$ and $u_1^+ = T_1^\ast$. See Figure \ref{GZ} for an illustration.
\vspace{6pt}
\begin{figure}[h]
\centering
\includegraphics[width=16cm]{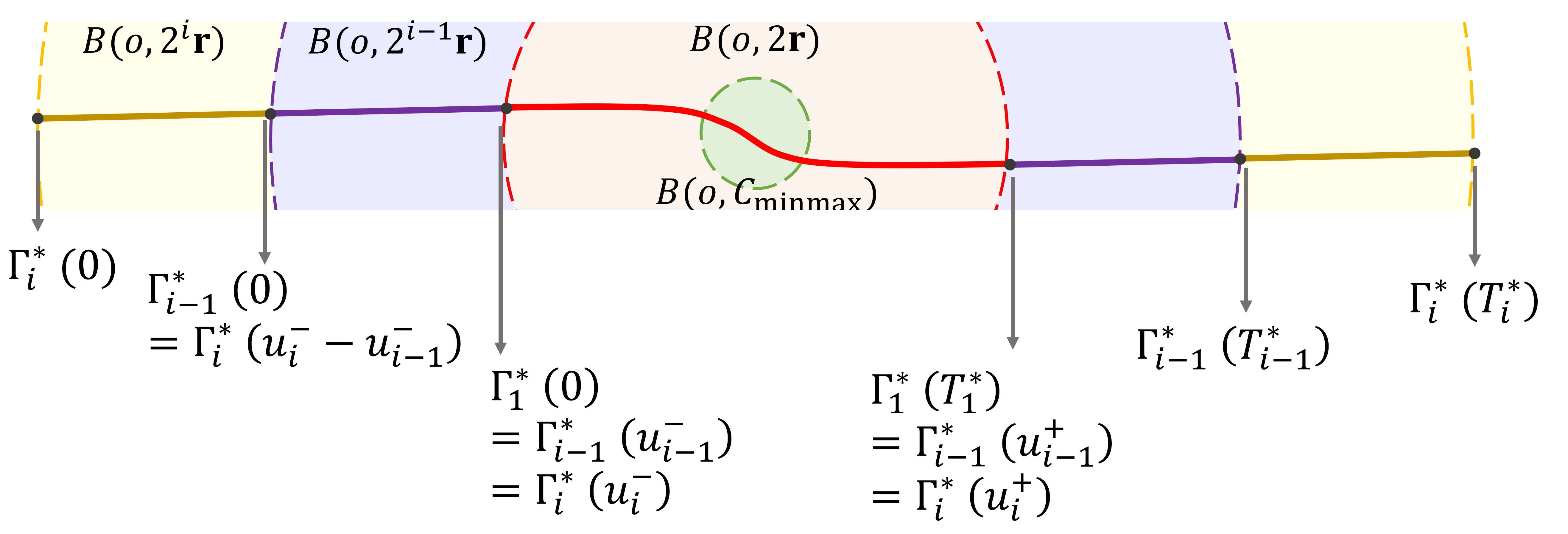}
\caption  {The definition of $u_i^\pm$.}
\label{GZ}
\end{figure}

Noting that as subsets
$\Gz_{i-1}^\ast|_{[0,u_{i-1}^-]} \subset \Gz_{i}^\ast|_{[0,u_{i}^-]}$, one has
$$  u_i^- = \ell_g(\Gz_{i}^\ast|_{[0,u_{i}^-]}) \ge \ell_g(\Gz_{i-1}^\ast|_{[0,u_{i-1}^-]}) =u_{i-1}^-, \quad i\ge 2.  $$
Moreover, one has
\begin{equation}\label{tl}
  \Gz_{i}^\ast(t+(u_i^- - u_{i-1}^-)) = \Gz_{i-1}^\ast(t), \quad t \in [0,T_{i-1}^\ast], \ i \ge 2.
\end{equation}


%

\medskip
\emph{Step 3.  Find a geodesic line $\Gz$ with   the Morse index $\le n-1$.}


First, note that
\begin{equation}\label{tl1}
  u_{i}^- \ge \ell_g(  \Gz^\ast_i|_{[0, u^-_i]}) \ge d_g(\Gz^\ast_i(0), \Gz^\ast_1(0)) \ge (1-\ez) (2^{i}-1){\bf r} \overset{i\to \fz}{\longrightarrow} \fz
\end{equation}
and
\begin{equation}\label{tl2}
  T_i^\ast-u_i^- \ge  T_i^\ast-u_i^+ \ge \ell_g(  \Gz^\ast_i|_{[u^+_i,T_i^\ast]}) \ge d_g(\Gz^\ast_1(T_1^\ast),\Gz^\ast_i(T_i^\ast)) \ge (1-\ez) (2^{i}-1){\bf r} \overset{i\to \fz}{\longrightarrow} \fz
\end{equation}
where we recall $\ez$ from \eqref{std1}.
Define $\Gz: (-\fz,+\fz) \to M$ by
\begin{equation*}
\Gamma(t):= \left \{ \begin{array}{lll} \Gamma^\ast_1(t)
& \quad  &  t \in [0,T_1^\ast] \\
\Gamma^\ast_i(t+u_i^- )
  & \quad & t \in [-u_i^-, - u_{i-1}^-) \cup (T_{i-1}^\ast-u_{i-1}^- ,T_i^\ast-u_i^-].
\end{array}
\right.
\end{equation*}
  Since all $ \Gamma_{i}^\ast \in C^2([0,T_i^\ast],M)$ are geodesic segments of $(M,g)$ with index $\le n-1$, thanks to \eqref{tl}, \eqref{tl1} and \eqref{tl2}, we deduce that
  $\Gz\in C^2((-\fz,\fz) ,M)$ is a well-defined geodesic line of $(M,g)$ with  Morse index $\le n-1$.

In addition, by the relation \eqref{tl}, one further observe that
$$  \Gz(t) = \Gamma^\ast_i(t+u_i^- ) , \quad  t \in [-u_i^-, T_i^\ast-u_i^-], \ i \in \nn. $$
In particular,
\begin{equation}\label{tl3}
  \Gz(t) = \Gamma^\ast_i(t+u_i^- ) , \quad  t \in [0, T_i^\ast-u_i^-] , \ i \in \nn.
\end{equation}

\medskip
\emph{Step 4.  Show $\pa_\fz \Gz=\{\tz,\tz^\star\}$.}

It suffices to show that for any sequence $s_k \to +\fz$ as $k \to +\fz$,
\begin{equation}\label{kfz}
  \lim_{k \to \fz}\frac{\Gz(-s_k)}{|\Gz(-s_k)|} = a =(1,\tz) \mbox{ and } \lim_{k \to \fz}\frac{\Gz(s_k)}{|\Gz(s_k)|} = a^\star = (1, \tz^\star).
\end{equation}
We only show the second equality in \eqref{kfz}; the proof for the first one is much similar.
Recalling that the exponential map of $M$ is proper, we know $\Gz^{-1}(B(o,r)) \subset \rr$ is compact. Thus
\begin{equation}\label{dk1}
  \lim_{k \to \fz}d_k:=\lim_{k \to \fz}|\Gz(s_k)| = \fz.
\end{equation}
Noting that $\Gz|_{[0,T^\ast_1]} = \Gz_1^\ast$ and $\Gz_1^\ast \cap \overline{B(o,C_\minm)} \ne \emptyset$, we know there exists $s_\ast \in [0,T^\ast_1]$ such that
\begin{equation}\label{sast}
  \Gz(s_\ast) \in \overline{B(o,C_\minm)}.
\end{equation}
In the following, we assume $s_k > T_1^\ast \ge  s_\ast$ for all $k \in \nn$.

For each $k \in \nn$, we know there exists $i_k \in \nn$ with $\lim_{k \to \fz}i_k =\fz$ such that
     $$  s_{k}  \in (T_{i_k-1}^\ast-u_{i_k-1}^- ,T_{i_k}^\ast-u_{i_k}^-].  $$
     By the definition of $\Gz$ and \eqref{tl3}, we know
\begin{equation}\label{dk2}
\Gz(t) =\Gz_{i_k}^\ast (t + u_{i_k}^-), \quad t \in [0,T_{i_k}^\ast-u_{i_k}^-].
\end{equation}
Recalling \eqref{ri} in Step 2, we further have,
$$ \Gz(t) =\Gz_{i_k}^\ast (t + u_{i_k}^-) = \lim_{j\in \nn_{i_k}:\rho_j > d_k}\Gz_j(t+u_i^-+ t_{j,r_{i_k}}^-), \quad t \in [0,T_{i_k}^\ast-u_{i_k}^-].  $$
In particular,
\begin{equation}\label{sk0}
  \Gz(s_k) = \lim_{j\in \nn_{i_k}:\rho_j > d_k}\Gz_j(s_k+u_i^-+ t_{j,r_{i_k}}^-)  \mbox{ and } \Gz(s_\ast) = \lim_{j\in \nn_{i_k}:\rho_j > d_k}\Gz_j(s_\ast+u_i^-+ t_{j,r_{i_k}}^-).
\end{equation}
From the first equality in \eqref{sk0}, for each $k \in \nn$, we can find some $j_k\in \nn_{i_k}$ with $\rho_{j_k} > d_k$ such that
$$ | \Gz(s_k) - \Gz_{j_k}(s_k+u_i^-+ t_{j_k,r_{i_k}}^-)| \le \frac1k.$$
Since \eqref{sast} gives   $ \Gz(s_\ast) \in \overline{B(o,C_\minm)}, $
from the second equality in \eqref{sk0} we deduce that,  for all    $j\in \nn_{i_k}$ with $\rho_j > d_k$,
$$  \Gz_{j}(s_\ast+u_i^-+ t_{j,r_{i_k}}^-) \in B(o, C_\ast C_\minm). $$
for some $C_\ast\ge 1$. Write
For $k\in\nn$,
$$y_k:= \Gz_{j_k}(s_k+u_i^-+ t_{j_k,r_{i_k}}^-)\quad\mbox{and}\quad \wz d_k|y_k|.$$
Recalling that $s_k > s_\ast$, applying Lemma \ref{delta} (v) with $t_\rho = s_\ast+u_i^-+ t_{j_k,r_{i_k}}^-$ and $\Gz_\rho = \Gz_{j_k}$ therein, we deduce that
$$s_k+u_i^-+ t_{j_k,r_{i_k}}^+ \in \Delta_{\rho_{j_k}, \wz d_k}^+, \quad k\in \nn.$$
Note that
$$|\frac{y_k}{d_k}-\frac{y_k}{\wz d_k}| \le |y_k|\frac{|\wz d_k-d_k|}{d_k\wz d_k} \le \frac{ | \Gz(s_k) - \Gz_{j_k}(s_k+u_i^-+ t_{j_k,r_{i_k}}^-)|}{d_k} \le \frac{1}{kd_k}.$$
Applying Corollary \ref{closure}, we conclude
\begin{align*}
 \lim_{k \to \fz} |\frac{\Gz(s_k)}{|\Gz(s_k)|}-a^\star|  & \le \lim_{k \to \fz} |\frac{\Gz(s_k)}{d_k}-\frac{y_k}{d_k}| +\lim_{k \to \fz} |\frac{y_k}{d_k}-\frac{y_k}{\wz d_k}|+ \lim_{k \to \fz}| \frac{y_k}{\wz d_k}-a^\star| \\
   & \le \lim_{k \to \fz} \frac{1}{kd_k} + \lim_{k \to \fz} \frac{1}{kd_k} + \lim_{k \to \fz} \sup \{|\frac{y}{r}-a^\star|: y \in \bigcup_{\rho > r}\Gz_\rho(\Delta^+_{\rho,r}) \} \\
& =0,
\end{align*}
which gives the second equality in \eqref{kfz} as desired.

This completes the proof of Step 4, also completes the whole proof.
\end{proof}

\noindent Jiayin Liu

\noindent Mathematic Area, SISSA,
Trieste, Italy

\noindent{\it E-mail }:  \texttt{jliu@sissa.it}

\bigskip

\noindent Shijin Zhang

\noindent School of Mathematical Science, Beihang University, Beijing, China

\noindent{\it E-mail }:  \texttt{shijinzhang@buaa.edu.cn}

\bigskip

\noindent  Yuan Zhou

\noindent
School of Mathematical Science, Beijing Normal University, Beijing, China

\noindent{\it E-mail }:  \texttt{yuan.zhou@bnu.edu.cn}


\end{document}